\documentclass{jocg}

\usepackage{hyperref}

\usepackage[textwidth=6.3in]{geometry}
\usepackage{enumitem}
\usepackage[utf8]{inputenc} 

\usepackage{amsmath,amssymb,amsthm}
\usepackage{array}
\usepackage{booktabs}
\usepackage{adjustbox}
\usepackage{multirow}

\usepackage[french,english]{babel}  
\usepackage{lmodern}             
\usepackage{mathtools}
\numberwithin{equation}{section}
\usepackage{subcaption}
\allowdisplaybreaks
\usepackage{color}
\definecolor{darkred}{RGB}{100,0,0}
\definecolor{darkgreen}{RGB}{0,100,0}
\definecolor{darkblue}{RGB}{0,0,150}
\definecolor{citecol}{RGB}{30,80,150}
\definecolor{tabcol}{RGB}{200,230,255}

\usepackage{colortbl}

\hypersetup{colorlinks=true, linkcolor=darkred, citecolor=citecol, urlcolor=darkblue}
\usepackage{url}

\newcommand{\p}[1]{\left(#1 \right)}

\newcommand{\CC}{\mathcal{C}}
\newcommand{\HH}{\mathcal{H}}
\newcommand{\OO}{\mathcal{O}}
\newcommand{\QQ}{\mathcal{Q}}
\newcommand{\RR}{\mathcal{R}}
\newcommand{\XX}{\mathcal{X}}

\newcommand{\E}{\mathbb{E}}
\newcommand{\R}{\mathbb{R}}

\newcommand{\defeq}{\vcentcolon=}

\newcommand{\fmin}{f_{\min}}
\newcommand{\fmax}{f_{\max}}
\newcommand{\taumin}{\tau_{\min}}

\newcommand{\QQm}{\QQ^d_{\taumin,\fmin,\fmax}}

\newcommand{\eps}{\varepsilon}

\DeclareMathOperator*{\Conv}{Conv}
\DeclareMathOperator*{\Rad}{Rad}

\newtheorem{theorem}{Theorem}[section]
\newtheorem{lemma}[theorem]{Lemma}
\newtheorem{cor}[theorem]{Corollary}
\newtheorem{prop}[theorem]{Proposition}
\newtheorem{rem}[theorem]{Remark}
\newtheorem{definition}[theorem]{Definition}

\begin{document}

\raggedbottom
% "Title of the paper"
\title{\MakeUppercase{Minimax adaptive estimation in manifold inference}}
\author{%
  Vincent~Divol %
  \thanks{\affil{Center for Data Science and Courant Institute of Mathematical Science, New York University}           \email{firstname.lastname@nyu.edu}}
}
\date{}
\maketitle

%
%
%\author{\fnms{Vincent} \snm{Divol}\ead[label=e1]{vincent.divol@nyu.edu}\ead[label=e2,url]{vincentdivol.github.io/}}
%
%\address{Center for Data Science, New York University\\
%Courant Institute of Mathematical Science, New York University\\
%\printead{e1,e2}}
%
%\runauthor{Vincent Divol}

\begin{abstract}
We focus on the problem of manifold estimation: given a set of observations sampled close to some unknown submanifold $M$, one wants to recover information about the geometry of $M$. Minimax estimators which have been proposed so far all depend crucially on the a priori knowledge of parameters quantifying the underlying distribution generating the sample (such as bounds on its density), whereas those quantities will be unknown in practice. Our contribution to the matter is twofold. First, we introduce a one-parameter family of manifold estimators $(\hat M_t)_{t\geq 0}$ based on a localized version of convex hulls, and show that for some choice of $t$, the corresponding estimator is minimax on the class of models of $\CC^2$ manifolds introduced in \cite{stat:genovese2012manifold}. Second, we propose a completely data-driven selection procedure for the parameter $t$, leading to a minimax adaptive manifold estimator on this class of models. This selection procedure actually allows us to recover the Hausdorff distance between the set of observations and $M$, and can therefore be used as a scale parameter in other settings, such as tangent space estimation.
\end{abstract}

%\begin{keyword}[class=MSC2020]
%\kwd[Primary ]{62G05}
%%\kwd{}
%\kwd[; secondary ]{62C20, 68U05}
%\end{keyword}
%
%\begin{keyword}
%\kwd{geometric inference}
%\kwd{manifold estimation}
%\kwd{adaptive estimation}
%\kwd{\v Cech complex}
%\end{keyword}

% history:
% \received{\smonth{1} \syear{0000}}

%\tableofcontents

\section{Introduction}

Manifold inference deals with the estimation of geometric quantities in a random setting. Given $\XX_n = \{X_1, \dots, X_n\}$ a set of i.i.d.~observations from some law $\mu$ on $\R^D$ supported on (or concentrated around) a $d$-dimensional manifold $M$, one wants to produce an estimator $\hat\theta$ that estimates accurately some quantity $\theta(M)$ related to the geometry of $M$ such as its dimension $d$ \cite{stat:hein2005intrinsic,stat:little2009multiscale, stat:kim2019minimax}, its homology groups \cite{geo:niyogi2008finding, stat:balakrishnan2012minimax}, its tangent spaces \cite{stat:aamari2019nonasymptotic, stat:cheng2016tangent}, or $M$ itself \cite{ stat:genovese2012manifold, stat:genovese2012minimax, stat:maggioni2016multiscale, stat:aamari2018stability, stat:aamari2019nonasymptotic, stat:puchkin2019structure}. Consider for instance the problem of estimating the manifold $M$ with respect to the Hausdorff distance $d_H$. The quality of an  estimator $\hat M$ with respect to some law $\mu$, called its $\mu$-risk, is given by the average Hausdorff distance $d_H$ between the estimator and $M$:

\begin{equation}
R_n(\hat M,\mu) \defeq \E[d_H(\hat M,M)],
\end{equation}
where $\hat M=\hat M(\XX_n)$ and $\XX_n$ is a $n$-sample of law $\mu$. In reality, the law $\mu$ generating the dataset is unknown, and it is more interesting to control the $\mu$-risk uniformly over a set $\QQ$ of laws $\mu$, that we call a statistical model. The uniform risk of the estimator $\hat M$ on the class $\QQ$ is given by,

\begin{equation}
R_n(\hat M,\QQ) \defeq \sup\{R_n(\hat M,\mu):\ \mu\in \QQ\}.
\end{equation}

while we say that an estimator is \emph{minimax} if it attains (up to a multiplicative constant as $n$ goes to $\infty$) the \emph{minimax risk}

\begin{equation}
\RR_n(\QQ) \defeq \inf\{R_n(\hat M,\QQ):\ \hat M \text{ is an estimator}\}.
\end{equation}

\begin{figure}
\centering
\includegraphics[width=0.2\textwidth]{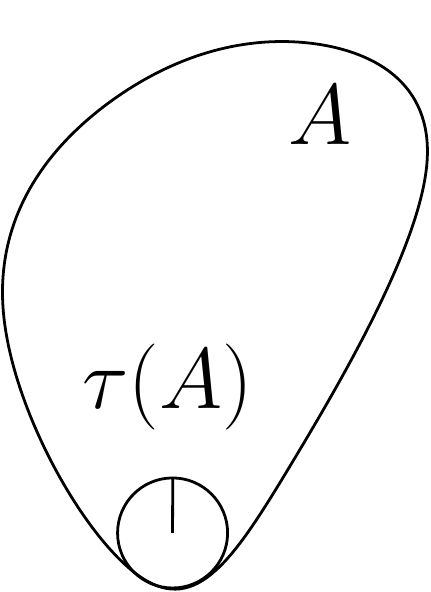}
\hspace{1cm}
\includegraphics[width=0.5\textwidth]{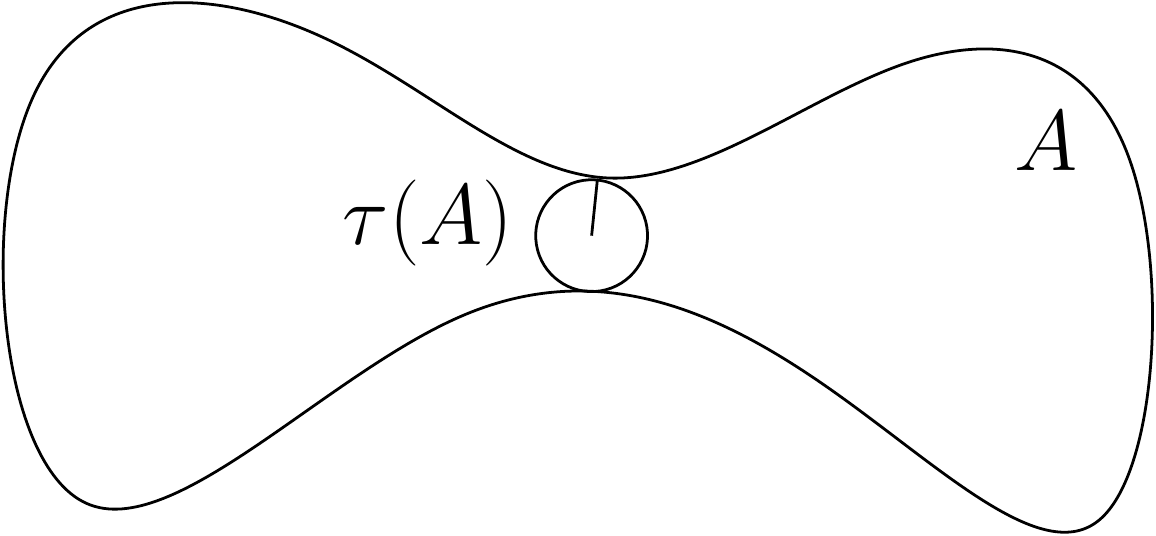}
\caption{If the reach of the curve $M$ is large, then the curve cannot be too pinched (left) and cannot present a tight bottleneck structure (right).}\label{fig:reach}
\end{figure}

In geometric inference, several statistical models were introduced, which take into account different noise models and regularities of the manifold $M$. Let us mention the family of models $\QQm$ introduced by Genovese et al. in \cite{stat:genovese2012manifold}, consisting of the laws $\mu$ supported on a $d$-dimensional manifold $M$ satisfying some additional properties. First, we assume that $\mu$ has a density $f$ on $M$, lower bounded by some constant $\fmin>0$ and upper bounded by another constant $\fmax$. This ensures that all the parts of the manifold $M$ are approximately evenly sampled: we then say that the law is "almost-uniform" on $M$. The parameter $\taumin$ gives a lower bound on the reach $\tau(M)$ of the manifold. The reach is a central notion in geometric inference, defined as the largest radius $r$ such that, if some point $x$ is at distance less than $r$ to $M$, then there exists a unique projection $\pi_M(x)$ of $x$ on $M$. As such, it controls both a local regularity of $M$ (a bound on its curvature radius) and a global regularity (namely the presence of a "bottleneck structure"), see also Figure \ref{fig:reach}.

On the statistical model $\QQm$, the minimax rate of convergence satisfies

\begin{equation}
c_0\p{\frac{\ln n}{n}}^{2/d} \leq \RR_n(\QQm) \leq c_1\p{\frac{\ln n}{n}}^{2/d},
\end{equation}
for two positive constants $c_0$, $c_1$ depending on $\taumin$, $\fmin$, $\fmax$ and $d$. The lower bound in this inequality was shown by Kim and Zhou \cite{stat:kim2015tight}, while the upper bound is obtained by exhibiting an estimator having a uniform risk of order $(\ln n/n)^{2/d}$. Such an estimator (although not computable in practice) was first proposed by Genovese et al. in \cite{stat:genovese2012manifold}, while another estimator
attaining this same minimax rate (computable in practice), and based on the Tangential Delaunay Complex \cite{geo:boissonnat2014manifold}, was proposed by Aamari and Levrard \cite{stat:aamari2018stability}. Although being minimax and computable, the Tangential Delaunay Complex depends on the tuning of several parameters (for instance a radius quantifying the size of neighborhoods which are used to compute local PCAs), while those parameters have to be calibrated in a precise manner with respect to the quantities $\taumin$, $\fmin$ and $\fmax$ defining the model for the Tangential Delauny Complex to be
minimax. However, those quantities are a priori unknown. A first possibility is to estimate those quantities in turn: if procedures are known to estimate the reach (although themselves
depending on the tuning of parameters \cite{stat:aamari2019estimating, stat:berenfeld2020estimating}), estimating $\fmin$ and $\fmax$ appears to be delicate. The problem of the practical choice of the parameters defining the estimator is then raised. This question of the tuning of parameters defining an estimator is not restricted to the framework of manifold estimation, but is a classical problem in statistics.

Let us cite for instance the question of the choice of the bandwidth for kernel density estimation. Let $X_1,\dots,X_n$ be a $n$-sample of some law $\mu$ having a density $f$ on $\R$, and suppose that we want to recover the value $f(x_0)$ of the density at some fixed point $x_0\in\R$. A standard method to achieve this goal is to consider the convolution of the empirical measure $\mu_n = \frac{1}{n} \sum_{i=1}^n \delta_{X_i}$ by some kernel $K_h$, where $K_h = h^{-1} K(\cdot /h)$ and $K$ satisfies $\int K = 1$. We then obtain a function $\hat f = K_h * \mu_n$. Assume that the density $f$ is of regularity $s$, that is $f \in \CC^s(\R)$, the set of $\lfloor s \rfloor$-times differentiable functions, whose $\lfloor s \rfloor$th derivative is $(s - \lfloor s \rfloor)$-H\"older continuous. Then, for a good choice of kernel $K$, it is optimal to choose the bandwidth $h_{\mathrm{opt}}$ of order $c\cdot n^{-1/(2s+1)}$, where $c$ depends on the $\CC^s$-norm of $f$ \cite[Chapter 1]{tsybakov2008introduction}. The associated risk is then of order $n^{-s/(2s+1)}$, which is the minimax rate of estimation on the class of densities of regularity $s$. In practice, it is impossible to know exactly the value of $s$, so that we must find another strategy to choose the bandwidth $h$. Adaptive methods consist in choosing a bandwidth $\hat h$ in a data-dependent way, such that the estimator $\hat f_{\hat h}$ has a $\mu$-risk almost as good as the optimal estimator $\hat f_{h_{\mathrm{opt}}}$ under weak hypotheses on $\mu$. One of such methods, the PCO method (for Penalized Comparison to Overfitting) introduced by Lacour, Massart and Rivoirard \cite{stat:lacour2016estimator} consists in comparing each estimator $\hat f_h$ to some degenerate estimator $\hat f_{h_{\min}}$ for some very small bandwidth $h_{\min}$. The selected bandwidth $\hat h$ is chosen among a family $\HH$ of bandwidths (all larger than $h_{\min}$), by minimizing a criterion depending on the distance $\|\hat f_h - \hat f_{h_{\min}}\|_{L_2(\R)}$, while penalizing small values of $h$. Lacour, Massart and Rivoirard then show an oracle inequality for their estimator, that is an inequality of the form
\begin{equation}\label{eq1.5}
\E\|\hat f_{\hat h}-f\|_{L_2(\R)}^2 \leq C \min\{ \E\|\hat f_h-f\|^2_{L_2(\R)}:\ h\in \HH\} + C(n,|\HH|)
\end{equation}
where $C(n, |\HH|)$ is a remainder term negligible in front of the optimal risk. Thus, we obtain that $\hat f_{\hat h}$ has a risk almost as good as the best estimator $\hat f_{h_{\mathrm{opt}}}$, while we never had to estimate the parameters defining the statistical model (that is the regularity $s$ of the density and the $\CC^s$-norm of $f$).

\begin{figure}
\centering
\includegraphics[width = 0.6\textwidth]{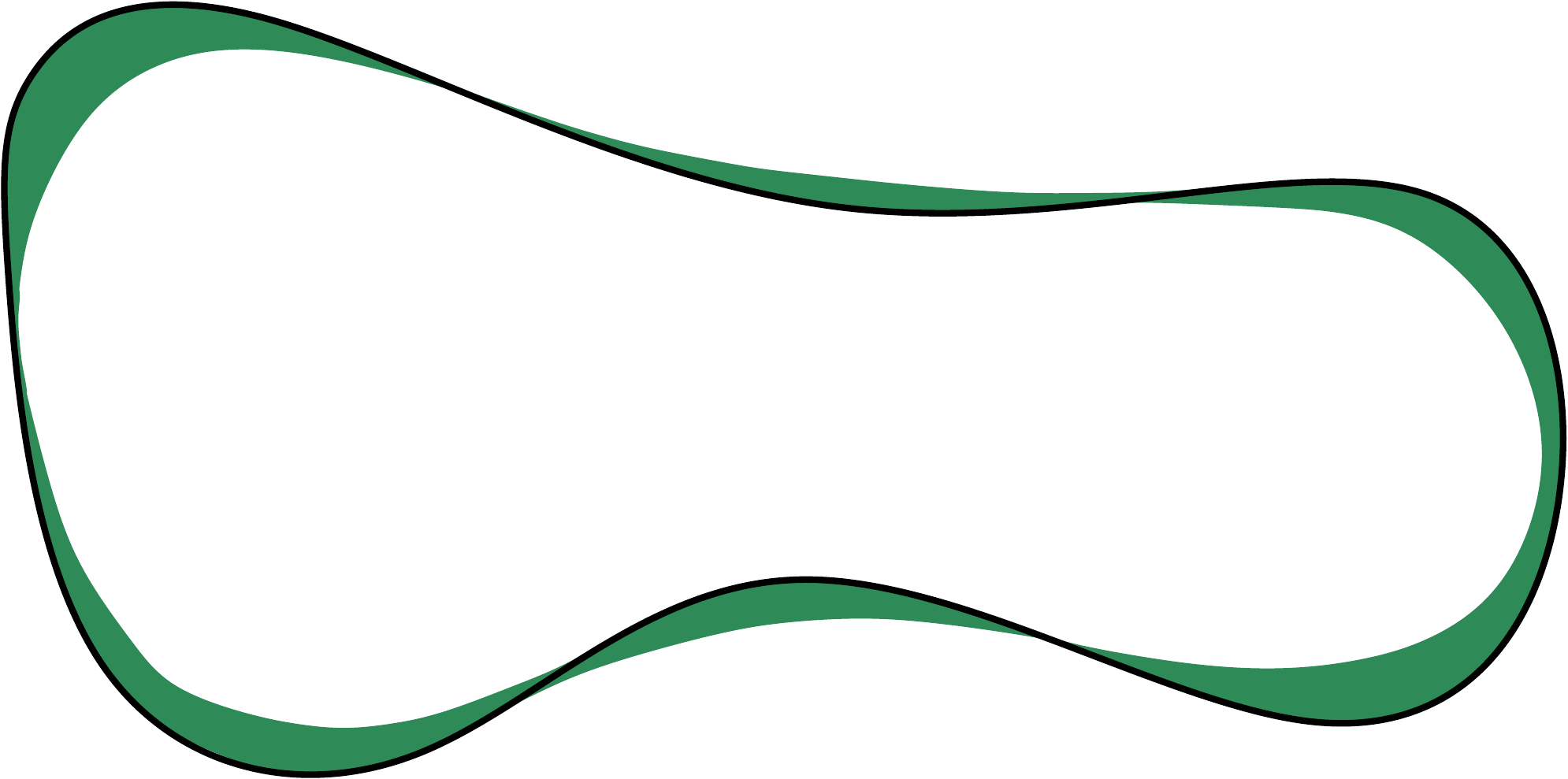}
\caption{The $t$-convex hull $\Conv(t,A)$ (in green) of a curve $A$ (in black).}\label{fig:tconvhull}
\end{figure}

Our main goal is to adapt the PCO method to the manifold inference setting. A first step consists in creating a family of estimators $(\hat M_t)_{t\geq 0}$ similar to kernel density estimators, but in the context of manifold estimation. This is made possible with $t$-convex hulls. For $t\geq 0$, the
$t$-convex hull $\Conv(t, A)$ of a set $A$ is an interpolation between the set $A$ ($t = 0$) and its convex hull $\Conv(A)$ ($t = \infty$). It is defined as

\begin{equation}
\Conv(t,A) \defeq \bigcup_{\substack{\sigma \subseteq A\\ r(\sigma)\leq t}} \Conv(\sigma),
\end{equation}
where $r(\sigma)$ is the radius of the set $\sigma$, that is the radius of the smallest enclosing ball of $\sigma$. See Figure \ref{fig:tconvhull} for an example. We prove in Section \ref{sec:approx} that for $A\subseteq M$, the Hausdorff distance between $\Conv(t, A)$ and $M$ can be efficiently controlled for values of $t$ a little larger than the approximation rate $\eps(A)\defeq \sup\{d(x, A):\ x \in M\}$ of $A$. More precisely, for such values of $t$,
Lemma \ref{lem3.3} states that $d_H(\Conv(t, A), M) \leq t^2 /\tau(M)$. Using this control on the $t$-convex hull enables us to show that the $t$-convex hull of the sample $\XX_n$ is a minimax estimator on the model $\QQm$ for a certain choice of $t$.

\begin{theorem}\label{thm1.1}
Let $\alpha_d$ be the volume of the $d$-dimensional unit ball. For the choice of scale $t_n = \frac{7}{4} (3 \ln n/(\alpha_d \fmin n))^{1/d}$, we have (for $n$ large enough)
\begin{equation}
R_n(\Conv(t_n,\XX_n), \QQm) \leq \frac{c_0}{\taumin(\alpha_d \fmin)^{2/d}} \p{\frac{\ln n}{n}}^{2/d}
\end{equation}
for some absolute constant $c_0$. In other words, $\Conv(t_n, \XX_n)$ is a minimax estimator of $M$ on $\QQm$.
\end{theorem}

To create an adaptive estimator, the next step is to build a selection procedure for the parameter $t$. An analog of the degenerate estimator $\hat f_{h_{\min}}$ is given by the choice $t = 0$, with $\Conv(0, \XX_n) = \XX_n$. The PCO method therefore suggests comparing the estimators $\Conv(t, \XX_n)$ with $\XX_n$, that is to study the function $t\mapsto h(t, \XX_n) \defeq d_H (\Conv(t, \XX_n), \XX_n)$. The function $h(\cdot, \XX_n)$ was actually already introduced under the name of "convexity defect function of the set $\XX_n$" in a paper by Attali, Lieutier and Salinas \cite{geo:attali2013vietoris}, where it was used to study the homotopy types of Rips complexes. The convexity defect function is nonnegative, nondecreasing, and satisfies $0 \leq h(t, A) \leq t$ for any set $A$. For $A = \XX_n$, this function is piecewise constant, while it may only change values at $t \in \Rad(\XX_n) \defeq \{r(\sigma):\ \sigma\subseteq \XX_n \}$. We show that the convexity defect function $h(t, \XX_n)$ of $\XX_n$ at scale $t$ exhibits different behaviors in two different regimes: for $t\leq \eps(\XX_n)$ it has a globally linear behavior (that is it stays close to its maximal value $t$), whereas roughly after $\eps(\XX_n)$, it is almost constant. The convexity defect function can be computed using only the dataset, so that we may in practice observe those two regimes. In practice, we fix a value $0 < \lambda < 1$, and let 

\begin{equation}
t_\lambda(\XX_n)\defeq \inf\{t\in \Rad(\XX_n):\ h(t,\XX_n) \leq \lambda t\}.
\end{equation}

Our main result states that $t_\lambda(\XX_n)$ is a little larger than $\eps(\XX_n)$ with high probability, so that we may control the risk of $\hat M = \Conv(t_\lambda(\XX_n), \XX_n)$, without having to know $d$, $\fmin$, $\fmax$ or the reach $\tau(M)$, leading to an adaptive estimator in a sense made precise in Theorem \ref{thm6.2}. 
The estimator $\hat M$ is to our knowledge the first minimax adaptive manifold estimator. 
Our procedure allows us to actually estimate (up to a multiplicative constant arbitrarily close to $1$) the approximation rate $\eps(\XX_n)$, while scale parameters in computational geometry typically have to be properly tuned with respect to this quantity. The parameter $t_{\lambda}(\XX_n)$ can therefore be used as a hyperparameter in different settings. To illustrate this general idea, we show how to create a data-driven minimax estimator of the tangent spaces of a manifold (see Corollary \ref{cor6.1}).

\subsection*{Related work}

"Localized" versions of convex hulls such as the $t$-convex hulls have already been introduced in the support estimation literature. For instance, slightly modified versions of the t-convex hull have been used as estimators in \cite{stat:aaron2016local} under the assumption that the support has a smooth boundary and in \cite{stat:rodriguez2007set} under reach constraints on the support, with different rates obtained in those models. Selection procedures were not designed in those two papers, and whether our selection procedure leads to an adaptive estimator in those frameworks is an interesting question. The statistical models we study in this article were introduced in \cite{stat:genovese2012manifold} and \cite{stat:aamari2018stability}, in which manifold estimators were also proposed. If the estimator in \cite{stat:genovese2012manifold} is of purely
theoretical interest, the estimator proposed by Aamari and Levrard in \cite{stat:aamari2018stability}, based on the Tangential Delaunay complex, is computable with $\OO(nD2^{\OO(d^2)})$ operations. Furthermore, it is a simplicial complex which is known to be ambient isotopic to the underlying manifold $M$ with high probability. It however requires the tuning of several hyperparameters in order to be
minimax, which may make its use delicate in practice. In contrast, the $t$-convex hull estimator with parameter $t_\lambda(\XX_n)$ is completely data-driven, computable in polynomial time (see Section \ref{sec:num}), while keeping the minimax property. However, unlike in the case of the Tangential Delaunay complex, we have no guarantees on the homotopy type of the corresponding estimator.

\section{Background on submanifold with positive reach}\label{sec:background}

Let us first introduce some notation. The Euclidean norm in $\mathbb{R}^{D}$ is denoted by $|\cdot|$ and $\langle\cdot, \cdot\rangle$ stands for the dot product. If $A \subseteq \mathbb{R}^{D}$ and $x \in \mathbb{R}^{D}$, then $d(x, A):=\inf \{|x-y|: y \in A\}$ is the distance to a set $A$ while $\operatorname{diam}(A):=\sup \{|x-y|: x, y \in A\}$ is its diameter. Given $r \geq 0$, $\mathcal{B}(x, r)$ is the open ball of radius $r$ centered at $x$ and we write $\mathcal{B}_{A}(x, r)$ for $\mathcal{B}(x, r) \cap A$. We let $\mathcal{M}^{d}$ be the set of $\mathcal{C}^{2}$ compact connected $d$-dimensional submanifolds of $\mathbb{R}^{D}$ without boundary. If $M \in \mathcal{M}^{d}$ and $x \in M$, then $T_{x} M$ is the tangent space of $M$ at $x$. It is identified with a $d$-dimensional subspace of $\mathbb{R}^{D}$, and we write $\pi_{x}$ for the orthogonal projection on $T_{x} M$, while $\pi_{x}^{\perp}=$ id $-\pi_{x}$ is the projection on the normal space $T_{x} M^{\perp}$. The asymmetric Hausdorff distance between sets $A, B \subseteq \mathbb{R}^{D}$ is defined as $d_{H}(A |B):=\sup \{d(x, B): x \in A\}$, while the Hausdorff distance is defined as $d_{H}(A, B)=\max \left\{d_{H}(A |B), d_{H}(B |A)\right\}$. For $A \subseteq M$, we denote by $\varepsilon(A):=d_{H}(A, M)$ the approximation rate of $A$.

The regularity of a submanifold $M \in \mathcal{M}^{d}$ is measured by its reach $\tau(M)$. This is the largest number $r$ such that if $d(x, M)<r$ for $x \in \mathbb{R}^{D}$, then there exists a unique point of $M$, denoted by $\pi_{M}(x)$, which is at distance $d(x, M)$ from $x$. Thus, the projection $\pi_{M}$ on the manifold $M$ is well-defined on the $r$-tubular neigborhood $M^{r}:=\{x \in M: d(x, M) \leq r\}$ for $r<\tau(M)$. The notion of reach was introduced for general sets by Federer in \cite{geo:federer1959curvature}, where it is also proven that $\mathcal{C}^{2}$ compact submanifolds without boundary have positive reach (see \cite[p.432]{geo:federer1959curvature}). Different geometric quantities of interest can be bounded in term of the reach. For instance, the volume $\operatorname{Vol}(M)$ of $M$ satisfies

\begin{equation}\label{eq2.1}
\operatorname{Vol}(M) \geq \omega_{d} \tau(M)^{d}
\end{equation}
where $\omega_{d}$ is the volume of a $d$-dimensional sphere (with equality obtained only for a sphere of radius $\tau(M)$), see \cite{geo:almgren1986optimal}. The reach also controls how points on $M$ deviate from their projections on some tangent space.

\begin{lemma}[Theorem 4.18 in \cite{geo:federer1959curvature}]\label{lem2.1} For $x, y \in M$, $\left|\pi_{x}^{\perp}(y-x)\right| \leq \frac{|y-x|^{2}}{2 \tau(M)}$.
\end{lemma}

The following lemma asserts that the projection from a manifold to its tangent space is well-behaved.

\begin{lemma}\label{lem2.2} Let $x \in M$.

\begin{enumerate}
\item Let $y \in \mathbb{R}^{D}$ with $d(y, M)<\tau(M)$. Then, $\pi_{M}(y)=x$ if and only if $y -x\in T_{x} M^{\perp}$.

\item Let $y_{1}, y_{2} \in \mathbb{R}^{D}$ be two points at distance less than $\gamma<\tau(M)$ from $M$. Then, $|\pi_{M}(y_{1})- \pi_{M}(y_{2})| \leq \frac{\tau(M)}{\tau(M)-\gamma} | y_{1}-y_{2} |$.

\item For $r<\tau(M) / 3$, the map $\tilde{\pi}_{x}: y \mapsto \pi_{x}(y-x)$ is a diffeomorphism from $\mathcal{B}_{M}(x, r)$ to its image, and, if $r \leq \tau(M) / 2$, we have $\mathcal{B}_{T_{x} M}(0,7 r / 8) \subseteq \tilde{\pi}_{x}\left(\mathcal{B}_{M}(x, r)\right)$. In particular, if $y \in \mathcal{B}_{M}(x, 7 \tau(M) / 24)$, then
\begin{equation}\label{eq2.2}
\frac{7}{8}|y-x| \leq\left|\pi_{x}(y-x)\right| \leq|y-x|.
\end{equation}
\end{enumerate}
\end{lemma}

\begin{proof}
$\bullet$ For 1 and 2, see \cite[Theorem 4.8]{geo:federer1959curvature}.

$\bullet$ We first show that $\tilde{\pi}_{x}$ is injective on $\mathcal{B}_{M}(x, \tau(M) / 3)$. Assume that $\tilde{\pi}_{x}(y)=\tilde{\pi}_{x}\left(y^{\prime}\right)$ for some $y \neq y^{\prime} \in M$. Consider without loss of generality that $|x-y| \geq\left|x-y^{\prime}\right|$. The goal is to show that $|x-y| \geq \tau(M) / 3$. If $|x-y|>\tau(M) / 2$, the conclusion obviously holds, so we may assume that $|x-y| \leq \tau(M) / 2$. Define the angle between $T_{x} M$ and $T_{y} M$ as $\left\|\pi_{x}-\pi_{y}\right\|_{o p}$ (where $\|\cdot\|_{\text {op }}$ denotes the operator norm). Lemma $3.4$ in \cite{stat:belkin2009constructing} states that if $|x-y| \leq \tau(M) / 2$, then $\angle\left(T_{x} M, T_{y} M\right) \leq 2 \frac{|x-y|}{\tau(M)}$. Also, by definition,

\begin{equation*}
\begin{aligned}
\angle\left(T_{x} M, T_{y} M\right) & \geq \frac{\left|\left(\pi_{x}-\pi_{y}\right)\left(y-y^{\prime}\right)\right|}{\left|y-y^{\prime}\right|} \\
&=\frac{\left|\pi_{y}\left(y-y^{\prime}\right)\right|}{\left|y-y^{\prime}\right|} \geq \frac{\left|y-y^{\prime}\right|-\left|\pi_{y}^{\perp}\left(y-y^{\prime}\right)\right|}{\left|y-y^{\prime}\right|} \\
& \geq 1-\frac{\left|y-y^{\prime}\right|}{2 \tau(M)} \text { by Lemma \ref{lem2.1}} \\
& \geq 1-\frac{|x-y|}{\tau(M)} \text { by the triangle inequality. }
\end{aligned}
\end{equation*}

Therefore, we have $3|x-y| / \tau(M) \geq 1$, i.e. $|x-y| \geq \tau(M) / 3$ and $\tilde{\pi}_{x}$ is injective on $\mathcal{B}_{M}(x, \tau(M) / 3)$. To conclude that $\tilde{\pi}_{x}$ is a diffeomorphism, it suffices to show that its differential is always invertible. As $\tilde{\pi}_{x}$ is an affine application, the differential $d \tilde{\pi}_{x}(y)$ is equal to $\pi_{x}$. Therefore, the Jacobian of the function $\tilde{\pi}_{x}: M \rightarrow T_{x} M$ at $y$ is given by the determinant of the projection $\pi_{x}$ restricted to $T_{y} M$. In particular, it is larger than the smallest singular value of $\pi_{x} \circ \pi_{y}$ to the power $d$, which is larger than

\begin{equation*}
\left(1-\angle\left(T_{x} M, T_{y} M\right)\right)^{d} \geq\left(1-2 \frac{|x-y|}{\tau(M)}\right)^{d} \geq\left(\frac{1}{3}\right)^{d}
\end{equation*}

thanks to \cite[Lemma 3.4]{stat:belkin2009constructing} and using that $|x-y| \leq \tau(M) / 3$. In particular, the Jacobian is positive, and $\tilde{\pi}_{x}$ is a diffeormorphism from $\mathcal{B}_{M}(x, \tau(M) / 3)$ to its image. The second statement is stated in \cite[Lemma A.2]{stat:aamari2019nonasymptotic}. 

The second inequality of the last statement follows from the projection being 1-Lipschitz continuous. For the first one, let $y \in \mathcal{B}_{M}(x, 7 \tau(M) / 24)$, and let $u=$ $\pi_{x}(y-x)$. The point $u$ is in $\mathcal{B}_{T_{x} M}(0, h)$ for $h>|u|$. We have $\mathcal{B}_{T_{x} M}(0, h) \subseteq \tilde{\pi}_{x}\left(\mathcal{B}_{M}(x, 8 h / 7)\right) \subseteq$ $\tilde{\pi}_{x}\left(\mathcal{B}_{M}(x, \tau(M) / 3)\right)$. As $\tilde{\pi}_{x}$ is injective on $\mathcal{B}_{M}(x, \tau(M) / 3)$, this means that we necessarily have $y \in \mathcal{B}_{M}(x, 8 h / 7)$. Therefore, $|x-y|<8 h / 7$, and the conclusion holds by letting $h$ goes to $|u|$.
\end{proof}

 It will also be necessary to have precise bounds on the volume of balls on $M$. As expected, the volume of a small ball is asymptotically equivalent to the volume of an Euclidean ball. Let $\alpha_{d}$ be the volume of the $d$-dimensional unit ball.

\begin{lemma}\label{lem2.3}
Let $r \leq \tau(M) / 4$ and $x \in M.$ Then,
\begin{equation}
\left(\frac{47}{48}\right)^{d} \leq\left(1-\frac{r^{2}}{3 \tau(M)^{2}}\right)^{d} \leq \frac{\operatorname{Vol}\left(\mathcal{B}_{M}(x, r)\right)}{\alpha_{d} r^{d}} \leq\left(1+\frac{4 r^{2}}{3 \tau(M)^{2}}\right)^{d} \leq\left(\frac{13}{12}\right)^{d}.
\end{equation}
\end{lemma}

\begin{proof}
 The proof of Proposition 8.7 in \cite{stat:aamari2018stability} implies that, if $\tilde{\mathcal{B}}_{M}(x, r)$ is the geodesic ball centered at $x$ of radius $r$, then
\begin{equation*}
\left(1-\frac{r^{2}}{3 \tau(M)^{2}}\right)^{d} \leq \frac{\operatorname{Vol}\left(\tilde{\mathcal{B}}_{M}(x, r)\right)}{\alpha_{d} r^{d}} \leq\left(1+\frac{r^{2}}{\tau(M)^{2}}\right)^{d}.
\end{equation*}

As $\tilde{\mathcal{B}}_{M}(x, r) \subseteq \mathcal{B}_{M}(x, r)$, we have in particular $\frac{\operatorname{Vol}\left(\mathcal{B}_{M}(x, r)\right)}{\alpha_{d} r^{d}} \geq\left(1-\frac{r^{2}}{3 \tau(M)^{2}}\right)^{d}$. Furthermore, by \cite[Lemma 3.12]{stat:arias2019unconstrained} and \cite[Proposition 6.3]{geo:niyogi2008finding}, if $|x-y| \leq \tau(M) / 4$, then the geodesic distance between $x$ and $y$ is smaller than

\begin{equation*}
|x-y|\left(1+\frac{\pi^{2}}{50 \tau(M)^{2}}|x-y|^{2}\right) \leq 1.05|x-y|.
\end{equation*}

This implies that $\mathcal{B}_{M}(x, r) \subseteq \tilde{\mathcal{B}}_{M}\left(x, r\left(1+\frac{\pi^{2} r^{2}}{50 \tau(M)^{2}}\right)\right)$. Therefore,

\begin{equation*}
\begin{aligned}
&\frac{\operatorname{Vol}\left(\mathcal{B}_{M}(x, r)\right)}{\alpha_{d} r^{d}} \leq\left(\left(1+\frac{\pi^{2} r^{2}}{50 \tau(M)^{2}}\right)\left(1+\frac{(1.05 r)^{2}}{\tau(M)^{2}}\right)\right)^{d} \\
&\leq\left(1+\left(\frac{\pi^{2}}{50}+(1.05)^{2}+\frac{\pi^{2}(1.05)^{2} r^{2}}{50 \tau(M)^{2}}\right) \frac{r^{2}}{\tau(M)^{2}}\right) \leq\left(1+\frac{4 r^{2}}{3 \tau(M)^{2}}\right)^{d},
\end{aligned}
\end{equation*}

where we used at the last line that $r \leq \tau(M) / 4$.
\end{proof}

\section{Approximation of manifolds with \texorpdfstring{$t$}{t}-convex hulls}\label{sec:approx}

Let $A \subseteq M$ be a finite set. We investigate in this section how the $t$-convex hull of $A$ approximates $M$ for different values of $t$, first in a deterministic setting, then in a random setting. The quantity of interest $d_{H}(\operatorname{Conv}(t, A), M)$ is by definition the maximum of the two quantities $d_{H}(\operatorname{Conv}(t, A) | M)$ and $d_{H}(M | \operatorname{Conv}(t, A))$. The first quantity $d_{H}(\operatorname{Conv}(t, A) | M)$ is given by the maximum of the distances $d_{H}(\operatorname{Conv}(\sigma) | M)$ over the simplexes $\sigma \subseteq A$ satisfying $r(\sigma) \leq t.$ A naive attempt to bound this quantity leads to a control of order $t$.

\begin{lemma}\label{lem3.1}
Let $\sigma \subseteq \mathbb{R}^{D}$ be a closed set. Then, $d_{H}(\operatorname{Conv}(\sigma) | \sigma) \leq r(\sigma)$.
\end{lemma}

\begin{proof}
Let $y \in \operatorname{Conv}(\sigma)$ and let $z$ be the center of the smallest enclosing ball of $\sigma$. The half-space $\left\{x \in \mathbb{R}^{D}:|x-z|^{2}-r(\sigma)^{2} \leq|x-y|^{2}-d(y, \sigma)^{2}\right\}$ contains $\sigma$. It thus contains $\operatorname{Conv}(\sigma)$, and in particular $y$. Therefore, $d(y, \sigma)^{2} \leq r(\sigma)^{2}-|y-z|^{2} \leq r(\sigma)^{2}$, concluding the proof.
\end{proof}

As $\sigma \subseteq M$, we have in particular that $d_{H}(\operatorname{Conv}(t, A) | M) \leq t$. We can actually obtain a much better bound by exploiting that $\sigma$ lies on $M$, which looks locally like a flat space. Consider for instance the case where $\sigma=\left\{x_{0}, x_{1}\right\}$ is made of two points. Then, the line $\left(x_{0}, x_{1}\right)$ should be approximately parallel to the tangent space $T_{x_{0}} M$, with the distance from $x_{1}$ to $T_{x_{0}} M$ being of order $\left|x_{0}-x_{1}\right|^{2}$. As a consequence, the distance from any point of the segment $\left[x_{0}, x_{1}\right]$ to $M$ is also of order $\left|x_{0}-x_{1}\right|^{2}$. More generally, we have the following result.

\begin{lemma}\label{lem3.2}
Let $\sigma \subseteq M$ with $r(\sigma)<\tau(M)$ and let $y \in \operatorname{Conv}(\sigma)$. Then,

\begin{equation}
d(y, M) \leq \tau(M)\left(1-\sqrt{1-\frac{r(\sigma)^{2}}{\tau(M)^{2}}}\right) \leq \frac{r(\sigma)^{2}}{2 \tau(M)}\left(1+\frac{r(\sigma)^{2}}{\tau(M)^{2}}\right).
\end{equation}

In particular, for any $t \geq 0$ and $A \subseteq M$,

\begin{equation}
d_{H}(\operatorname{Conv}(t, A) | M) \leq \frac{t^{2}}{2 \tau(M)}\left[\left(1+\frac{t^{2}}{\tau(M)^{2}}\right) \wedge 2\right] \leq \frac{t^{2}}{\tau(M)}.
\end{equation}
\end{lemma}

\begin{proof} Lemma 12 in \cite{geo:attali2013vietoris} states that if $\sigma \subseteq M$ satisfies $r(\sigma)<\tau(M)$ and $y \in \operatorname{Conv}(\sigma)$, then,

\begin{equation*}
d(y, M) \leq \tau(M)\left(1-\sqrt{1-\frac{r(\sigma)^{2}}{\tau(M)^{2}}}\right).
\end{equation*}

As $\sqrt{1-u} \geq 1-u / 2-u^{2} / 2$ for $u \in[0,1]$, one obtains the conclusion.
\end{proof}

The other asymmetric distance $d_{H}(M | \operatorname{Conv}(t, A))$ is apparently more delicate to handle. It can actually be controlled efficiently if the parameter $t$ is large enough. Indeed, assume that $t$ is large enough so that every point $x$ of $M$ is the projection of some point $y$ of $\operatorname{Conv}(t, A)$. Then we have

\begin{equation}\label{eq3.3}
\begin{split}
d(x, \operatorname{Conv}(t, A)) &\leq|x-y|=\left|\pi_{M}(y)-y\right|=d(y, M) \\
&\leq d_{H}(\operatorname{Conv}(t, A) | M) \leq \frac{t^{2}}{\tau(M)}.
\end{split}
\end{equation}

This suggests defining the parameter

\begin{equation}
t^{*}(A):=\inf \left\{t<\tau(M): \pi_{M}(\operatorname{Conv}(t, A))=M\right\}.
\end{equation}

Lemma \ref{lem3.2} and \eqref{eq3.3} imply directly the following  lemma. 

\begin{figure}\label{fig3}
\includegraphics[width=\textwidth]{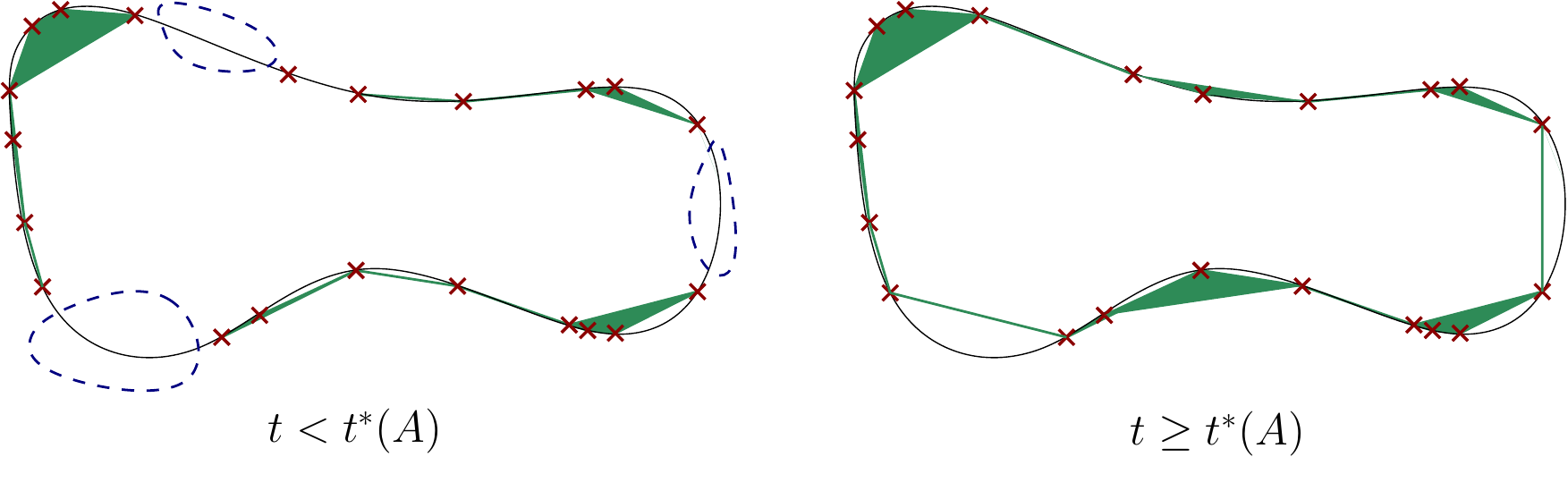}
\caption{The $t$-convex hull of the finite set $A$ (red crosses) is displayed (in green) for two values of $t$. The black curve represents the (one dimensional) manifold $M$. On the first display, the value of $t$ is smaller than $t^{*}(A)$, as there are regions of the manifold (circled in blue) which are not attained by the projection $\pi_{M}$ restricted to the $t$-convex hull. The value of $t$ is larger than $t^{*}(A)$ on the second display.}
\end{figure}

\begin{lemma}\label{lem3.3}
Let $A \subseteq M$ and $t > t^{*}(A)$. Then,
\begin{equation}
d_{H}(\operatorname{Conv}(t, A), M) \leq \frac{t^{2}}{\tau(M)}.
\end{equation}
\end{lemma}

A crucial result in the analysis of the $t$-convex hull estimator is given by the next proposition, that indicates that the quantity $t^{*}(A)$ is almost equal to the approximation rate $\varepsilon(A)$.

\begin{prop}\label{prop3.4}
Let $A \subseteq M$ be a finite set. Then, $\varepsilon(A) \leq t^{*}(A)\left(1+\frac{t^{*}(A)}{\tau(M)}\right).$ Furthermore, if $\varepsilon(A)<\tau(M) / 8$, then, $t^{*}(A) \leq \varepsilon(A)\left(1+6 \frac{\varepsilon(A)}{\tau(M)}\right)$
\end{prop}

The proof of Proposition \ref{prop3.4} relies on considering Delaunay triangulations. Given $d+1$ points $\sigma$ in $\mathbb{R}^{d}$ that do not lie on a hyperplane, there exists a unique ball that contains the points on its boundary. It is called the circumball of $\sigma$, and its radius is called the circumradius $\mathrm{circ}(\sigma)$ of $\sigma$. Given a finite set $A \subseteq \mathbb{R}^{d}$ that does not lie on a hyperplane, there exists a triangulation of $A$, called the Delaunay triangulation, such that for each simplex $\sigma$ in the triangulation, the circumball of $\sigma$ contains no point of $A$ in its interior. Note that there may exist several Delaunay triangulations of a set $A$, should the set $A$ not be in general position. With a slight abuse, we will still refer to "the" Delaunay triangulation of $A$, by simply choosing a Delaunay triangulation among the possible ones should several exist. If the set $A$ lies on a lower dimensional subspace, we consider the Delaunay triangulation of $A$ in the affine vector space spanned by $A$. Therefore, for every set $A$, the Delaunay triangulation is well defined (for instance, the Delaunay triangulation of three points aligned in the plane is the 1-dimensional triangulation obtained by joining the middle point with the two others).

\begin{proof}
 Let $x \in M$ be such that $d(x, A)=\varepsilon(A)$. By definition, there exists a simplex $\sigma \subseteq A$ of radius smaller than $t^{*}(A)$ with $x=\pi_{M}(y)$ for some point $y \in \operatorname{Conv}(\sigma)$. We have, using Lemma \ref{lem3.1} and Lemma \ref{lem3.2},

\begin{equation*}
\varepsilon(A)=d(x, A) \leq|x-y|+d(y, A) \leq \frac{t^{*}(A)^{2}}{\tau(M)}+t^{*}(A)
\end{equation*}

proving the first inequality.

To prove the other inequality, without loss of generality, we assume that $0 \in M$ and we show that $0 \in \pi_{M}(\operatorname{Conv}(t, A))$ for $t=\varepsilon(A)(1+6 \varepsilon(A) / \tau(M)).$ Let $\tilde{A}=\pi_{0}(A \cap \mathcal{B}(0, R))$ for $R=\varepsilon(A)\left(2+c_{0} \varepsilon(A) / \tau(M)\right)$ and $c_{0}=32 / 49$. Note that the condition $\varepsilon(A) \leq \tau(M) / 8$ implies that $R<7 \tau(M) / 24$. We first state two lemmas.

\begin{lemma}\label{lem3.5} Assume that $\varepsilon(A) \leq 7 \tau(M) / 24$. Let $\tilde{x} \in T_{0} M$ with $|\tilde{x}| \leq \varepsilon(A).$ Then $d(\tilde{x}, \tilde{A}) \leq$ $\varepsilon(A)$.
\end{lemma}

\begin{proof}
 By continuity, it suffices to prove the claim for $|\tilde{x}|<\varepsilon(A)$. In this case, according to Lemma \ref{lem2.2}, if $\varepsilon(A) \leq 7 \tau(M) / 24$, then there exists $x \in \mathcal{B}_{M}(0,8 \varepsilon(A) / 7)$ with $\pi_{0}(x)=\tilde{x}$. Furthermore, by Lemma \ref{lem2.1},

\begin{equation*}
|x| \leq|\tilde{x}|+|x-\tilde{x}| \leq \varepsilon(A)+\frac{|x|^{2}}{2 \tau(M)} \leq \varepsilon(A)\left(1+\frac{32 \varepsilon(A)}{49 \tau(M)}\right).
\end{equation*}

We have $d(x, A)=|x-a|$ for some point $a \in A$, and 
\[|a| \leq|x-a|+|x| \leq \varepsilon(A)\left(2+c_{0} \varepsilon(A) / \tau(M)\right).\]
 As $\pi_{0}(a) \in \tilde{A}$, we have $d(\tilde{x}, \tilde{A}) \leq\left|\tilde{x}-\pi_{0}(a)\right| \leq|x-a|=d(x, A) \leq \varepsilon(A)$.
\end{proof}

\begin{lemma}\label{lem3.6} Let $V \subseteq \mathbb{R}^{d}$ be a finite set and $t>0$. If $d_{H}(\mathcal{B}(0, t) | V) \leq t$, then $0 \in \operatorname{Conv}(V)$.
\end{lemma}

\begin{proof} We prove the contrapositive. If $0 \notin \operatorname{Conv}(V)$, then there exists an open half-space which contains $V$. Let $x$ be the unit vector orthogonal to this half-space. Then, $d(t x, V)>t$.
\end{proof}

Apply Lemma \ref{lem3.6} to $V=\tilde A$ and $t=\varepsilon(A)$. For $\tilde{x} \in \mathcal{B}_{T_{0} M}(0, \varepsilon(A))$, we have $d(\tilde{x}, \tilde A) \leq \varepsilon(A)$ according to Lemma \ref{lem3.5}. Therefore, we have $0 \in \operatorname{Conv}(\tilde A)$. Consider the Delaunay triangulation of $\tilde{A}$. The point 0 belongs to the convex hull of some simplex $\tilde{\sigma}$ of the triangulation, with circumradius $\mathrm{circ}(\tilde{\sigma})$ and center of the circumball $\tilde{q}$. The simplex $\tilde{\sigma}$ corresponds to some simplex $\sigma$ in $A$, and the point 0 is equal to $\pi_{0}(y)$ for some point $y \in \operatorname{Conv}(\sigma)$. By Lemma \ref{lem2.2}, we actually have $\pi_{M}(y)=0$, and to conclude, it suffices to show that $r(\sigma) \leq \varepsilon(A)\left(1+6 \frac{\varepsilon(A)}{\tau(M)}\right)$. To do so, we use the next lemma (recall that $\sigma \subseteq \mathcal{B}_{M}(0, R)$ with $\left.R<7 \tau(M) / 24\right)$.

\begin{lemma}\label{lem3.7} Let $\sigma \subseteq \mathcal{B}_{M}(0,7 \tau(M) / 24)$ and $\tilde{\sigma}=\tilde{\pi}_{0}(\sigma)$. Assume that $0 \in \operatorname{Conv}(\tilde{\sigma})$. Then,
\begin{equation}
r(\tilde{\sigma}) \leq r(\sigma) \leq r(\tilde{\sigma})\left(1+6 \frac{r(\tilde{\sigma})}{\tau(M)}\right).
\end{equation}
\end{lemma}

\begin{proof}
 As the projection is 1-Lipschitz, it is clear that $r(\tilde{\sigma}) \leq r(\sigma)$. Let us prove the other inequality. Let $\sigma=\left\{y_{0}, \ldots, y_{k}\right\}, \tilde{\sigma}=\left\{\tilde{y}_{0}, \ldots, \tilde{y}_{k}\right\}$ and fix $0 \leq i \leq k.$ As $y_{i} \in \mathcal{B}_{M}(0,7 \tau(M) / 24)$, we have by \eqref{eq2.2}

\begin{equation}\label{eq3.7}
\left|y_{i}\right| \leq \frac{8}{7}\left|\tilde{y}_{i}\right| \leq \frac{16}{7} r(\tilde{\sigma}).
\end{equation}

where we used that $\left|\tilde{y}_{i}\right| \leq 2 r(\tilde{\sigma})$ as $0 \in \operatorname{Conv}(\tilde{\sigma})$. Let $\tilde{z}$ be the center of the minimum enclosing ball of $\tilde{\sigma}$. Write $\tilde{z}=\sum_{j=0}^{k} \lambda_{j} \tilde{y}_{j}$ as a convex combination of the $\tilde{y}_{j}$s and let $z=\sum_{j=0}^{k} \lambda_{j} y_{j} \in$ $\operatorname{Conv}(\sigma)$. Then, we have

\begin{equation*}
\begin{aligned}
\left|z-y_{i}\right| & \leq|z-\tilde{z}|+\left|\tilde{z}-\tilde{y}_{i}\right|+\left|\tilde{y}_{i}-y_{i}\right| \\
& \leq \sum_{j=0}^{k} \lambda_{j}\left|y_{j}-\tilde{y}_{j}\right|+r(\tilde{\sigma})+\frac{\left|y_{i}\right|^{2}}{2 \tau(M)} \text { using Lemma \ref{lem2.1}}  \\
& \leq \sum_{j=0}^{k} \lambda_{j} \frac{\left|y_{j}\right|^{2}}{2 \tau(M)}+r(\tilde{\sigma})+\frac{128}{49} \frac{r(\tilde{\sigma})^{2}}{\tau(M)} \text { using Lemma \ref{lem2.1} and \eqref{eq3.7}} \\
& \leq r(\tilde{\sigma})+\frac{256}{49} \frac{r(\tilde{\sigma})^{2}}{\tau(M)} \leq r(\tilde{\sigma})+6 \frac{r(\tilde{\sigma})^{2}}{\tau(M)} \text { using \eqref{eq3.7}. }
\end{aligned}
\end{equation*}

We obtain the conclusion as $\sigma$ is included in the ball of radius $\max _{i}\left|z-y_{i}\right|$ and center $z$.
\end{proof}

Using the previous lemma, we are left with showing that $r(\tilde{\sigma}) \leq \varepsilon(A)$. We will actually show the stronger inequality $\operatorname{circ}(\tilde{\sigma}) \leq \varepsilon(A)$ (the radius of a set is always smaller than its circumradius). As 0 is in the circumball (that is centered at $\tilde{q}$), the ball centered at $\tilde{q}$ of radius $|\tilde{q}|$ does not intersect $\tilde A$. This enforces $|\tilde{q}| \leq \varepsilon(A)$: otherwise, there would exist a ball not intersecting $\tilde A$, of radius $\varepsilon(A)$, and whose center is at distance less than $\varepsilon(A)$ from 0, a contradiction with Lemma \ref{lem3.5} (see Figure \ref{fig4}). As $|\tilde{q}| \leq \varepsilon(A)$, we obtain, once again according to Lemma \ref{lem3.5}, that $\operatorname{circ}(\tilde{\sigma})=d(\tilde{q}, \tilde{A}) \leq \varepsilon(A)$ concluding the proof.
\end{proof}

\begin{figure}
\centering
\includegraphics[width=0.5\textwidth]{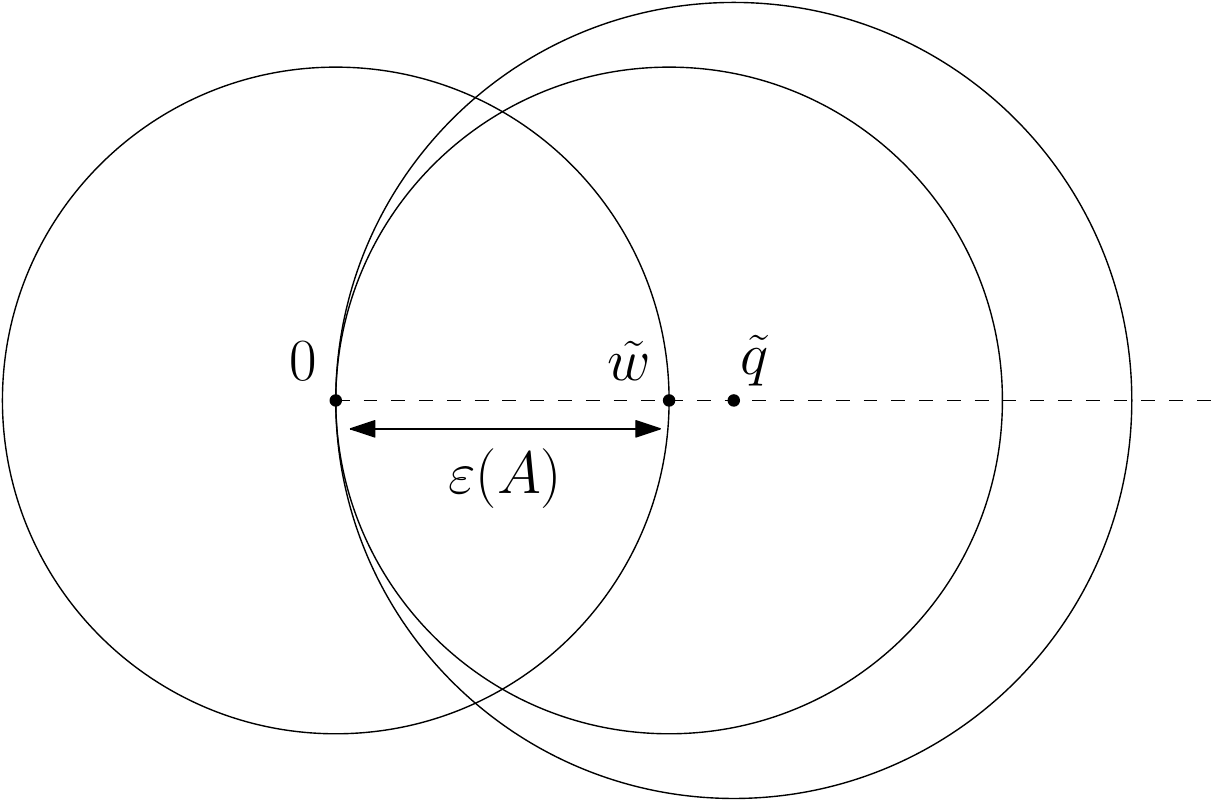}
\caption{If $|\tilde{q}|>\varepsilon(A)$, then the ball $\mathcal{B}_{T_{0} M}(\tilde{q},|\tilde{q}|)$ contains a ball of radius $\varepsilon(A)$ centered at a point (here denoted by $\tilde{w}$) at distance less than $\varepsilon(A)$ from 0.}\label{fig4}
\end{figure}

\begin{rem}\label{rem3.8} In the case where the dimension $d$ is known, one can consider a variant of the $t$-convex hull, $\operatorname{Conv}_{d}(t, A)$, where one restricts the union to be over simplices of dimension less than $d$. The set $\operatorname{Conv}_{d}(t, A)$ is simpler to compute as it contains less simplices (see Section \ref{sec:num}). Furthermore, if $t_{d}^{*}(A):=\inf \left\{t: \pi_{M}(\operatorname{Conv}_{d}(t, A))=M\right\}$, then both Lemma \ref{lem3.3} and Proposition \ref{prop3.4} hold with $t_{d}^{*}(A)$ and $\operatorname{Conv}_{d}(t, A)$ instead of $t^{*}(A)$ and $\operatorname{Conv}(t, A)$. Indeed, only simplices of dimension less than $d$ (corresponding to simplices of a Delaunay triangulation on a tangent space) were considered in the previous proof.
\end{rem}

We have now shown that the quality of the $t$-convex hull on $A$ can be controlled for $t \geq \varepsilon(A)(1+6 \varepsilon(A) / \tau(M))$ (that is slightly larger than the approximation rate $\varepsilon(A))$. In a random setting, the approximation rate is known to be of order $(\ln n / n)^{1 / d}$: this is enough to show that the $t$-convex hull is a minimax estimator. Recall the definition of the statistical model $\mathcal{Q}_{\tau_{\min }, f_{\min }, f_{\max }}^{d}$ from the introduction: it consists of laws $\mu$ supported on some manifold $M \in \mathcal{M}^{d}$ 
with $\tau(M) \geq \tau_{\min }$, having a density $f$ lower bounded by $f_{\min }$ and upper bounded by $f_{\max }$. The minimax result will actually hold on the larger model $\mathcal{Q}_{\tau_{\min }, f_{\min }}^{d}:=\bigcup_{f_{\max }} \mathcal{Q}_{\tau_{\min }, f_{\min }, f_{\max }}^{d}$ (that is without imposing any upper bound on $f$).

Let $\mu \in \mathcal{Q}_{\tau_{\min }, f_{\min }}^{d}$ and let $\mathcal{X}_{n}$ be a $n$-sample from law $\mu$. We consider the estimator $\operatorname{Conv}(t, \mathcal{X}_{n})$. Note first that $\operatorname{Conv}(t, \mathcal{X}_{n})$ is indeed an estimator, that is the application

\begin{equation*}
\left(x_{1}, \ldots, x_{n}\right) \in\left(\mathbb{R}^{D}\right)^{n} \mapsto \operatorname{Conv}(t,\left\{x_{1}, \ldots, x_{n}\right\})
\end{equation*}

is measurable (with respect to the Borel $\sigma$-field associated with the metric $d_{H}$ on the set $\mathcal{K}\left(\mathbb{R}^{D}\right)$ of all nonempty compact subsets of $\mathbb{R}^{D}$). Indeed, for $E$ a measurable subset of $\mathcal{K}\left(\mathbb{R}^{D}\right)$ and $A, B \in \mathcal{K}\left(\mathbb{R}^{D}\right)$, introduce the notation $G_{E}(A, B)=A$ if $A \in E$ and $B$ otherwise. This function is measurable, and $\operatorname{Conv}(t,\left\{x_{1}, \ldots, x_{n}\right\})$ can be written as

\begin{equation*}
\bigcup_{I \subseteq\{1, \ldots, n\}} G_{E}\left(\operatorname{Conv}(\left\{x_{i}\right\}_{i \in I}\right),\left\{x_{i}\right\}_{i \in I})
\end{equation*}

where $E$ is the subset of $\mathcal{K}\left(\mathbb{R}^{D}\right)$ given by $\left\{K \in \mathcal{K}\left(\mathbb{R}^{D}\right): r(K) \leq t\right\}$, which is closed \cite[Lemma 16]{geo:attali2013vietoris}. As the functions $\cup$ and Conv are measurable, the measurability follows \cite[Proposition III.7]{stat:aamari2017vitesses}.
\medskip

For a fixed $t>0$, we obtain the following control of $\mathbb{E}\left[d_{H}(\operatorname{Conv}(t, \mathcal{X}_{n}),M)\right].$

\begin{equation*}
\begin{aligned}
\mathbb{E}\left[d_{H}(\operatorname{Conv}(t, \mathcal{X}_{n}),M)\right] &=\mathbb{E}\left[d_{H}(\operatorname{Conv}(t, \mathcal{X}_{n}),M) \mathbf{1}\left\{t \geq t^{*}(\mathcal{X}_{n})\right\}\right]\\
&\qquad\qquad+\mathbb{E}\left[d_{H}(\operatorname{Conv}(t, \mathcal{X}_{n}),M) \mathbf{1}\left\{t<t^{*}(\mathcal{X}_{n})\right\}\right]\\
& \leq \frac{t^{2}}{\tau(M)}+\operatorname{diam}(M) \mathbb{P}\left(t^{*}(\mathcal{X}_{n})>t\right)
\end{aligned}
\end{equation*}

By Proposition \ref{prop3.4}, if $\varepsilon(\mathcal{X}_{n})<\tau(M) / 8$, then 
\[t^{*}(\mathcal{X}_{n}) \leq \varepsilon(\mathcal{X}_{n})\left(1+6 \frac{\varepsilon(\mathcal{X}_{n})}{ \tau(M)}\right) \leq \frac{7}{4} \varepsilon(\mathcal{X}_{n}).\]
 Therefore,   if $t$ is small enough,
 \begin{align*}
 \mathbb{P}\left(t^{*}(\mathcal{X}_{n})>t\right) &\leq \mathbb{P}\left(\varepsilon(\mathcal{X}_{n})>\tau(M) / 8\right)+\mathbb{P}\left(\varepsilon(\mathcal{X}_{n})>4 t / 7\right) \\
 &\leq 2 \mathbb{P}\left(\varepsilon(\mathcal{X}_{n})>4 t / 7\right).
 \end{align*}
We obtain

\begin{equation}\label{eq3.8}
\mathbb{E}\left[d_{H}(\operatorname{Conv}(t, \mathcal{X}_{n}),M)\right] \leq \frac{t^{2}}{\tau(M)}+2 \operatorname{diam}(M) \mathbb{P}\left(\varepsilon(\mathcal{X}_{n})>4 t / 7\right)
\end{equation}

Hence, to control the risk, it suffices to bound the tail of $\varepsilon(\mathcal{X}_{n})$.

\begin{prop}\label{prop3.9} Let $\mu \in \mathcal{Q}_{\tau_{\min }, f_{\min }}^{d}$ and let $\mathcal{X}_{n}=\left\{X_{1}, \ldots, X_{n}\right\}$ be a $n$-sample of law $\mu.$ If $r \leq \tau_{\min } / 4$, then, for any $\eta \in(0,1)$

\begin{equation}
\mathbb{P}\left(\varepsilon(\mathcal{X}_{n})>r\right) \leq \frac{c_{d, \eta}}{f_{\min } r^{d}} \exp \left(-n \alpha_{d} f_{\min }\left(1-\frac{r^{2}}{3 \tau_{\min }^{2}}\right)^{d} \eta r^{d}\right),
\end{equation}

where $c_{d, \eta}$ depends on $d$ and $\eta$. Furthermore, for any $a>0$, for $n$ large enough (with respect to $d$, $f_{\min }, \tau_{\min }$ and $a$), with probability $1-c(\ln n)^{d-1} n^{1-a}$ (where c depends also on those parameters), we have

\begin{equation}
\varepsilon(\mathcal{X}_{n}) \leq\left(\frac{a \ln n}{\alpha_{d} f_{\min } n}\right)^{1 / d}.
\end{equation}
\end{prop}

\begin{proof} A measure $\nu$ is said to be $(a, b)$-standard at scale $r_{0}$ if $\nu(\mathcal{B}(x, r)) \geq a r^{b}$ for all $r \leq r_{0}$ and $x$ in the support of $\nu$. Let $\mu \in \mathcal{Q}_{\tau_{\min }, f_{\min }}^{d}$ with support $M$. Lemma \ref{lem2.3} indicates that the measure $\mu$ is $(a, b)$-standard at scale $r_{0}$ for any $r_{0} \leq \tau(M) / 4$, with $a=f_{\min } \alpha_{d}\left(1-\frac{r_{0}^{2}}{3 \tau(M)^{2}}\right)^{d}$ and $b=d$. It is stated in the proof \cite[Proposition III.14]{stat:aamari2017vitesses} that for such a measure, and for any $\delta \leq 2 r_{0}$ with $0<r-\delta \leq r_{0}$, we have

\begin{equation*}
\mathbb{P}\left(\varepsilon(\mathcal{X}_{n})>r\right) \leq \frac{2^{b}}{a \delta^{b}} \exp \left(-n a(r-\delta)^{b}\right).
\end{equation*}

Letting $r=r_{0}$ and $\delta=\left(1-\eta^{1 / d}\right) r$ for some $\eta \in(0,1)$, we obtain that

\begin{equation}\label{eq3.11}
\begin{aligned}
&\mathbb{P}\left(\varepsilon(\mathcal{X}_{n})>r\right) \leq \frac{\left(2 /\left(1-\eta^{1 / d}\right)\right)^{d}}{f_{\min } \alpha_{d}\left(1-\frac{r^{2}}{3 \tau(M)^{2}}\right)^{d} r^{d}} \exp \left(-n f_{\min } \alpha_{d}\left(1-\frac{r^{2}}{3 \tau(M)^{2}}\right)^{d} \eta r^{d}\right) \\
&\leq \frac{c_{0}^{d}}{f_{\min } \alpha_{d} r^{d}} \exp \left(-n f_{\min } \alpha_{d}\left(1-\frac{r^{2}}{3 \tau(M)^{2}}\right)^{d} \eta r^{d}\right)
\end{aligned}
\end{equation}

for $c_{0}=96 /\left(47\left(1-\eta^{1 / d}\right)\right)$, where we used at the last line that $r \leq \tau(M) / 4$.

To prove the second statement, we let $r=\left(\frac{a \ln n}{\alpha_{d} f_{\min } n}\right)^{1 / d}$. Then, we have $n \alpha_{d} f_{\min } r^{d}=a \ln n$. Letting $\eta=1-1 / \ln n$, we obtain that

\begin{equation*}
\begin{aligned}
n f_{\min } \alpha_{d}\left(1-\frac{r^{2}}{3 \tau(M)^{2}}\right)^{d} \eta r^{d}&=(a \ln n)\left(1-\frac{1}{\ln n}\right)\left(1-c\left(\frac{\ln n}{n}\right)^{2 / d}\right) \\
&\geq(a \ln n)-C_{a}.
\end{aligned}
\end{equation*}

In particular, we obtain that the upper bound in \eqref{eq3.11} is of order $(\ln n)^{d-1} n^{1-a}$. 
\end{proof}

 Choose $t$ such that $4 t / 7=\left(\frac{3 \ln n}{\alpha_{d} f_{\min n} n}\right)^{1 / d}$. Then, according to Proposition \ref{prop3.9}, we have $\mathbb{P}\left(\varepsilon(\mathcal{X}_{n})>32 t / 7\right) \leq c(\ln n)^{d-1} n^{-2}.$ As $\operatorname{diam}(M)$ is also bounded by a constant depending on $\tau_{\min }, f_{\min }$ and $d$ (see \cite[Lemma III.24]{stat:aamari2017vitesses}), we obtain Theorem \ref{thm1.1} from \eqref{eq3.8} (without even the need of assuming that the density $f$ is upper bounded).

\section{Selection procedure for the \texorpdfstring{$t$}{t}-convex hulls}\label{sec:selection}

Assuming that we have observed a $n$-sample $\mathcal{X}_{n}$ having a distribution $\mu \in \mathcal{P}_{\tau_{\min }, f_{\min }}^{d}$, we were able in the previous section to build a minimax estimator of the underlying manifold $M$. The tuning of this estimator requires the knowledge of $d$ and $f_{\min }$: if the dimension $d$ can be efficiently estimated, this is not the case for $f_{\min }$, which will likely not be accessible in practice. An idea to overcome this issue is to design a selection procedure for the family of estimators $\left(\operatorname{Conv}(t, \mathcal{X}_{n})\right)_{t \geq 0}$. As the loss of the estimator $\operatorname{Conv}(t, \mathcal{X}_{n})$ is controlled efficiently for $t \geq t^{*}(\mathcal{X}_{n})$ a good idea is to select a scale $t$ larger than $t^{*}(\mathcal{X}_{n})$. We however do not have access to this quantity based on the observations $\mathcal{X}_{n}$, as the manifold $M$ is unknown. To select a scale close to $t^{*}(\mathcal{X}_{n})$, we monitor how the estimators $\operatorname{Conv}(t, \mathcal{X}_{n})$ deviate from $\mathcal{X}_{n}$ as $t$ increases. Namely, we use the convexity defect function introduced in \cite{geo:attali2013vietoris}.

\begin{definition} Let $A \subseteq \mathbb{R}^{D}$ and $t>0$. The convexity defect function at scale t of $A$ is defined $a s$
\begin{equation}
h(t, A):=d_{H}(\operatorname{Conv}(t, A), A).
\end{equation}
\end{definition}

As its name indicates, the convexity defect function measures the (lack of) convexity of a set $A$ at a given scale $t$. The next proposition states preliminary results on the convexity defect function.

\begin{prop} Let $A \subseteq \mathbb{R}^{D}$ be a closed set and $t \geq 0$.
\begin{enumerate}
\item We have $0 \leq h(t, A) \leq t$.
\item The set $A$ is convex if and only if $h(\cdot, A) \equiv 0$.
\item If $M \in \mathcal{M}^{d}$, then $h(t, M) \leq \frac{t^{2}}{2 \tau(M)}\left(1+\frac{t^{2}}{\tau(M)^{2}}\right)$.
\end{enumerate}
\end{prop}

\begin{proof}
Point 1 follows from Lemma \ref{lem3.1}. Point 2 is clear and Point 3 is a consequence of Lemma \ref{lem3.2}.
\end{proof} 

\begin{figure}
\centering

\raisebox{0.1\height}{\includegraphics[width=0.3\textwidth]{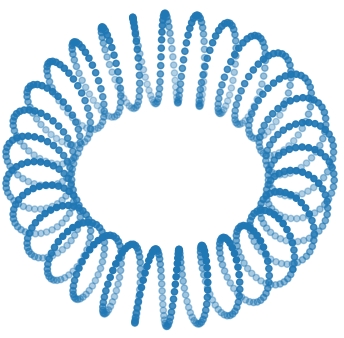}} \hspace{1cm}
\includegraphics[width=0.5\textwidth]{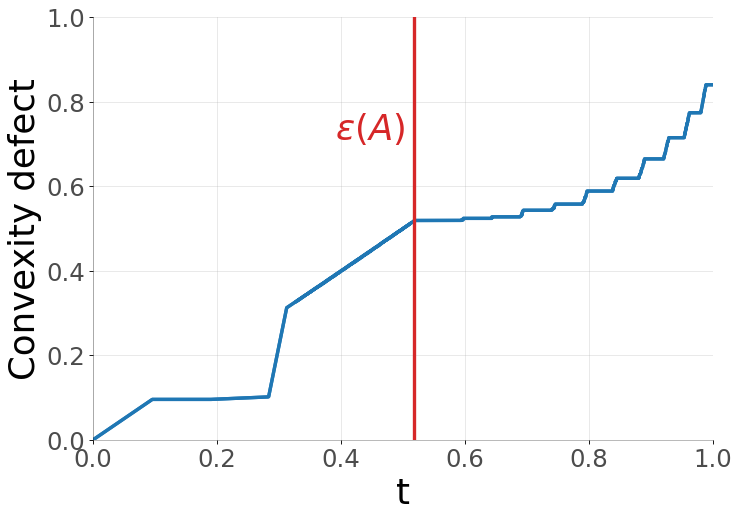}

\raisebox{0.1\height}{\includegraphics[width=0.3\textwidth]{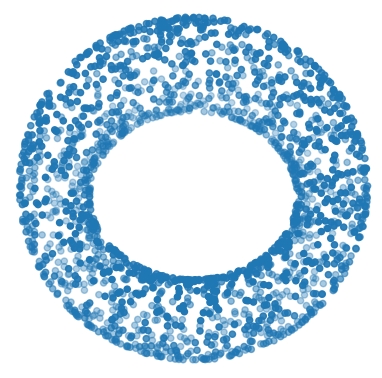}}
\hspace{1cm}
\includegraphics[width=0.5\textwidth]{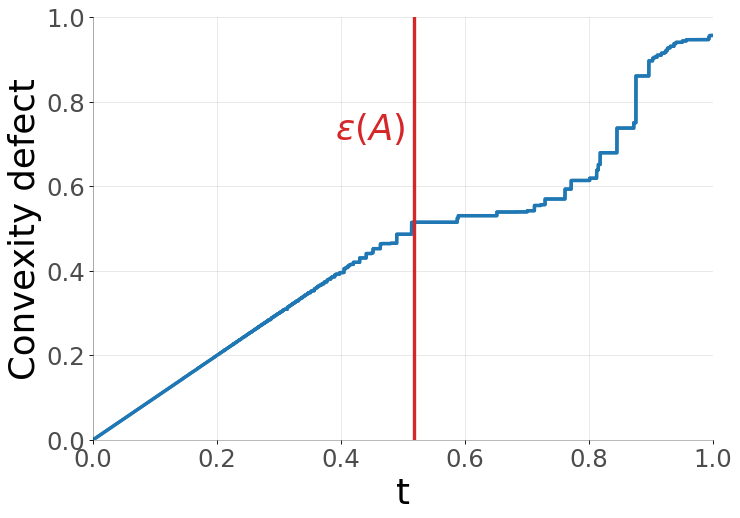}

\caption{ Two subsets of the torus having the same approximation rate, but whose convexity defect functions exhibit different behaviors on $\left[0, t^{*}(A)\right]$.}\label{fig5}
\end{figure}

As expected, the convexity defect of a convex set is null, whereas for small values of $t$, the convexity defect of a manifold $h(t, M)$ is very small (compared to the maximum value possible, which is $t$): when looked at locally, $M$ is "almost flat" (and thus "almost locally convex"). As already noted in the introduction, if $A$ is a finite set, then the convexity defect function is a piecewise constant function, whose value may only change at $t$ if $t \in \operatorname{Rad}(A):=\{r(\sigma): \sigma \subseteq A\}$.

For a set $A \subseteq M$, we recover the subquadratic behavior of the convexity defect function for values of $t$ above the threshold value $t^{*}(A)$. Namely, we have the following proposition.

\begin{prop}\label{prop4.3} Let $A \subseteq M$. For $t^{*}(A)<t<\tau(M)$,
\begin{equation}
h(t, A) \leq \frac{t^{2}}{2 \tau(M)}\left(1+\frac{t^{2}}{\tau(M)^{2}}\right)+t^{*}(A)\left(1+\frac{t^{*}(A)}{\tau(M)}\right).
\end{equation}
\end{prop}

\begin{proof}
 By using that $h(t, A) \leq t$ and Lemma \ref{lem3.3}, for any $t^{*}(A)<s<t$,

\begin{equation*}
\begin{aligned}
h(t, A) &=d_{H}(\operatorname{Conv}(t, A), A) \\
& \leq d_{H}(\operatorname{Conv}(t, A), M)+d_{H}(M, \operatorname{Conv}(s, A))+d_{H}(\operatorname{Conv}(s, A), A) \\
& \leq \frac{t^{2}}{2 \tau(M)}\left(1+\frac{t^{2}}{\tau(M)^{2}}\right)+\frac{s^{2}}{2 \tau(M)}\left(1+\frac{s^{2}}{\tau(M)^{2}}\right)+s
\end{aligned}
\end{equation*}

The conclusion is obtained by letting $s$ go to $t^{*}(A)$.
\end{proof}

For $0<t<t^{*}(A)$, the convexity defect function may exhibit very different behaviors, as shown in Figure \ref{fig5}. However, when the set $A=\mathcal{X}_{n}$ is a random $n$-sample, it appears that the graph of the convexity defect function stays close to the diagonal $\{x=y\}$ for small values of $t$. This is explained by the fact that for two points $X_{1}, X_{2}$ in the sample at very small distance $2 t$ from one another, it is very unlikely that there is a third point at distance of order $t$ from $X_{1}$ and $X_{2}$, so that $d_{H}(\operatorname{Conv}(\left\{X_{1}, X_{2}\right\}) |\mathcal{X}_{n})=d_{H}(\operatorname{Conv}(\left\{X_{1}, X_{2}\right\}) |\left\{X_{1}, X_{2}\right\})=t$.

This suggests the following strategy to select a value of $t$ larger than $t^{*}(\mathcal{X}_{n})$ using the convexity defect function:

\begin{definition} Let $A \subseteq M$ be a finite set and $0<\lambda \leq 1$. We define
\begin{equation}
t_{\lambda}(A):=\inf \{t \in \operatorname{Rad}(A): h(t, A) \leq \lambda t\}.
\end{equation}
\end{definition}

Restricting to values $t \in \operatorname{Rad}(A)$ is necessary, for otherwise we would always have $t_{\lambda}(A)=$ 0 (as $h(t, A)=0$ for $t$ small enough). Proposition \ref{prop4.3} implies that $t_{\lambda}(A)$ cannot be too large. More precisely, we have the following lemma.

\begin{lemma}\label{lem4.5} Let $A \subseteq M$ with $t^{*}(A) \leq \lambda^{2} \tau(M) / 4$. Let $r_{0}=\frac{t^{*}(A)}{\lambda}\left(1+\frac{8}{\lambda^{2}} \frac{t^{*}(A)}{\tau(M)}\right)$ and $r_{1}=$ $\lambda \tau(M) / 2.$ If $t \in \operatorname{Rad}(A) \cap\left[r_{0}, r_{1}\right]$, then $t_{\lambda}(A) \leq t$.
\end{lemma}

\begin{proof} By Proposition \ref{prop4.3}, we have, for $t^{*}(A) < t \leq \lambda \tau(M) / 2$,
\begin{equation*}
\begin{aligned}
h(t, A) & \leq \frac{t^{2}}{2 \tau(M)}\left(1+\frac{t^{2}}{\tau(M)^{2}}\right)+t^{*}(A)\left(1+\frac{t^{*}(A)}{\tau(M)}\right) \\
& \leq \frac{t^{2}}{2 \tau(M)}\left(1+\lambda^{2} / 4\right)+t^{*}(A)\left(1+\frac{t^{*}(A)}{\tau(M)}\right)-\lambda t+\lambda t=: P(t)+\lambda t.
\end{aligned}
\end{equation*}

Let $u=2 t^{*}(A)\left(1+\frac{t^{*}(A)}{\tau(M)}\right)\left(1+\lambda^{2} / 4\right) /\left(\lambda^{2} \tau(M)\right)$. The condition $t^{*}(A) \leq \lambda^{2} \tau(M) / 4$ ensures that $u \leq 1$. The quantity $P(t)$ is nonpositive if $t$ is between $t_{0}$ and $t_{1}$, where

\begin{equation*}
t_{0}=\frac{\tau(M) \lambda}{1+\lambda^{2} / 4}(1-\sqrt{1-u}) \quad \text { and } \quad t_{1}=\frac{\tau(M) \lambda}{1+\lambda^{2} / 4}(1+\sqrt{1-u}).
\end{equation*}

We have $t_{1} \geq r_{1}$ and, using the inequality $\sqrt{1-u} \geq 1-\frac{u}{2}-\frac{u^{2}}{2}$ for $0\leq u \leq 1$, we obtain that $t_{0} \leq r_{0}$. Therefore, any $t \in\left[r_{0}, r_{1}\right]$ satisfies $h(t, A) \leq \lambda t$ (note that $r_0>t^*(A)$). In particular, if $t$ is also in $\operatorname{Rad}(A)$, we have $t_{\lambda}(A) \leq t$.
\end{proof}

Our main theorem states that, with high probability, the parameter $t_{\lambda}(\mathcal{X}_{n})$ is larger than $t^{*}(\mathcal{X}_{n})$

\begin{theorem}\label{thm4.6}
\begin{enumerate}
\item Let $\mu \in \mathcal{Q}_{\tau_{\min }, f_{\min }, f_{\max }}^{d}$. Let $0<b \leq 2$ and let $\mathcal{X}_{n}$ be a n-sample of law $\mu$. Let $a=(d-1) \vee 2$ if $b=2$, and $a=d-1$ otherwise. For $n$ large enough, and with probability larger than $1-c(\ln n)^{a} n^{-b}$, we have for $0<\lambda<(1+b)^{-1 / d}$,

\begin{equation}\label{eq4.4}
t^{*}(\mathcal{X}_{n}) \leq t_{\lambda}(\mathcal{X}_{n}) \leq \frac{t^{*}(\mathcal{X}_{n})}{\lambda}\left(1+C\left(\frac{(\ln n)^{2}}{n}\right)^{1 / d}\right)
\end{equation}

where the constant $c$ depends on $b$, and $\mu$, and $C$ depends on $f_{\min }$, $f_{\max }$, $d$, $\tau_{\min }$ and $\lambda$.
\item  Furthermore, if $\mu$ is the uniform distribution on the circle of radius $\tau_{\min }$, then, for $\lambda> (1+b)^{-1}$, we have

\begin{equation}\label{eq4.5}
\mathbb{P}\left(t^{*}(\mathcal{X}_{n})>t_{\lambda}(\mathcal{X}_{n})\right) \geq c n^{-b}
\end{equation}

for some constant c depending on $\tau_{\min }$ and $b$.
\end{enumerate} 
\end{theorem}

Inequality \eqref{eq4.5} implies that the probability $1-c(\ln n)^{a} n^{-b}$ appearing in the theorem is close to being tight.

\subsection*{Proof of the upper bound in \eqref{eq4.4}}
 Let $\mu \in \mathcal{Q}_{\tau_{\min }, f_{\min }, f_{\max }}^{d}$ be a probability distribution with support $M$ and density $f$. We assume without loss of generality that $f_{\min }$ is the essential infimum of $f$. Recall the notation $r_{1}=\lambda \tau(M) / 2$ and $r_{0}=\frac{t^{*}(\mathcal{X}_{n})}{\lambda}\left(1+\frac{8}{\lambda^{2}} \frac{t^{*}(\mathcal{X}_{n})}{\tau(M)}\right)$ from Lemma \ref{lem4.5}. The proof of the upper bound is based on the following lemma.

\begin{lemma}\label{lem4.7} There exists a positive constant $\beta>0$ (depending on $f_{\min }, f_{\max }, d$ and $\left.\tau_{\min }\right)$ such that the following holds. Let $\alpha>0$ and let $I=[a, b]$ be an interval of length at least $\ell=\alpha\left(\frac{\ln n}{n}\right)^{2 / d}$ with $b \leq \beta \alpha\left(\frac{\ln n}{n}\right)^{1 / d}$ and $a \geq \ell / 2$. Then, the probability that $\operatorname{Rad}(\mathcal{X}_{n})$ does not intersect $I$ is smaller than $n^{-1 / 2}$.
\end{lemma}

Before proving the lemma, let us use it to obtain the upper bound in \eqref{eq4.4}. By Proposition \ref{prop3.9} and Proposition \ref{prop3.4}, we have $t^{*}(\mathcal{X}_{n}) \leq\left(\frac{4 \ln n}{\alpha_{d} f_{\min n}}\right)^{1 / d}$ everywhere but on a set of probability smaller than $n^{-2}$. We will assume that this condition is satisfied. In particular, the condition $t^{*}(\mathcal{X}_{n}) \leq \lambda^{2} \tau(M) / 4$ of Lemma \ref{lem4.5} is satisfied. Let $u=\delta\left(\frac{(\ln n)^{2}}{n}\right)^{1 / d}$ (for some constant $\delta$ to fix) and let
\begin{equation*}
R_{0}:=r_{0}(1+u) \leq \frac{t^{*}(\mathcal{X}_{n})}{\lambda}\left(1+2 \delta\left(\frac{(\ln n)^{2}}{n}\right)^{1 / d}\right) \leq \frac{2}{\lambda}\left(\frac{4 \ln n}{\alpha_{d} f_{\min } n}\right)^{1 / d} \leq r_{1}.
\end{equation*}

\begin{lemma}\label{lem4.8} Let $A \subseteq M$ be a finite set of cardinality $n$. Then,
\begin{equation*}
\varepsilon(A) \geq c_{d} \tau(M) n^{-1 / d}.
\end{equation*}
\end{lemma}

\begin{proof} If $\varepsilon(A) \geq \tau(M) / 4$, the conclusion holds. Otherwise, as $M \subseteq \bigcup_{x \in A} \mathcal{B}_{M}(x, \varepsilon(A))$, one has $\operatorname{Vol}(M) \leq n c_{d} \varepsilon(A)^{d}$ (using Lemma \ref{lem2.3}). We conclude with inequality \eqref{eq2.1}.
\end{proof}

According to Lemma \ref{lem4.8} and Proposition \ref{prop3.4}, the interval $\left[r_{0}, R_{0}\right]$ is of length $r_{0} u \geq$ $C_{1} \delta\left(\frac{\ln n}{n}\right)^{2 / d}=: \ell$ for some constant $C_{1}$. Choose $\delta$ large enough so that $\frac{2}{\lambda}\left(\frac{4}{\alpha_{d} f_{\min }}\right)^{1 / d} \leq \beta C_{1} \delta$. Then, as $r_{0} \geq \ell / 2$ (once again by Lemma \ref{lem4.8}), one can apply Lemma \ref{lem4.7}: the interval $\left[r_{0}, R_{0}\right]$ intersects $\operatorname{Rad}(\mathcal{X}_{n})$ with probability $1-n^{-2}$. Lemma \ref{lem4.5} then yields the conclusion.

\begin{proof}[Proof of Lemma \ref{lem4.7}]
 Let $I_{k}=[k \ell / 2,(k+1) \ell / 2]$ for $k$ an integer. Assume that we show that $\operatorname{Rad}(\mathcal{X}_{n})$ intersects every interval $I_{k}$ for $k=1, \ldots, K$, where $K$ is chosen so that $b \leq$ $\beta \alpha\left(\frac{\ln n}{n}\right)^{1 / d} \leq(K+1) \ell / 2$, say $K+1=\left\lceil 2 \beta \alpha\left(\frac{\ln n}{n}\right)^{1 / d} / \ell\right\rceil=\left\lceil 2 \beta(n / \ln n)^{1 / d}\right\rceil$. As the interval $I$ is of length at least $\ell$, and as $\ell / 2 \leq a$, the interval $I$ contains one of the interval $I_{k}$ for some $1 \leq k \leq K$. In particular, the interval $I$ also intersects $\mathcal{X}_{n}.$ Therefore, it suffices to bound the probability that $\operatorname{Rad}(\mathcal{X}_{n})$ does not intersect $I_{k}$. If we show that this probability is of order at most $n^{-3}$, we may then conclude by a union bound: the probability that $\operatorname{Rad}(\mathcal{X}_{n})$ intersects all the $I_{k}$ is larger than $1-2 K n^{-3} \geq 1-4 \beta n^{1 / d-3} /(\ln n)^{1 / d} \geq 1-n^{-2}.$

To bound the probability that $\operatorname{Rad}(\mathcal{X}_{n})$ does not intersect $I_{k}$, we split the set $\mathcal{X}_{n}$ into two groups: the set $\mathcal{X}_{n}^{0}=\left\{X_{1}, \ldots, X_{L}\right\}$ (for some integer $L$ to fix), and the set $\mathcal{X}_{n}^{1}=\left\{X_{L+1}, \ldots, X_{n}\right\}$. If some distance $\left|X_{i}-X_{j}\right|$ is between $k \ell$ and $(k+1) \ell$, then $\operatorname{Rad}(\mathcal{X}_{n})$ intersects $I_{k}$. We will show that it is very likely that $\left|X_{i}-X_{j}\right| \in[k \ell,(k+1) \ell]$ for some $i \leq L$ and $j>L$. To do so, we consider the ball $B_{i}$ centered at the point $X_{i}$, of radius $(k+1) \ell$. Let $Y$ be a point sampled according to $\mu$, conditioned on being in $B_{i}$. Then, according to Lemma \ref{lem2.3}, we have
\begin{equation*}
\begin{aligned}
&\mathbb{P}\left(\left|Y-X_{i}\right| \in[k \ell,(k+1) \ell] | X_{i}\right)=\frac{\mu\left(\mathcal{B}\left(X_{i},(k+1) \ell\right) \backslash \mathcal{B}\left(X_{i}, k \ell\right)\right)}{\mu\left(\mathcal{B}\left(X_{i},(k+1) \ell\right)\right)} \\
&\geq \frac{c_{d} f_{\min }}{f_{\max }(k+1)^{d}}\left((k+1)^{d}\left(1-\frac{(k+1)^{2} \ell^{2}}{3 \tau(M)^{2}}\right)^{d}-k^{d}\left(1+\frac{4 k^{2} \ell^{2}}{3 \tau(M)^{2}}\right)^{d}\right) \\
&\geq \frac{c_{d} f_{\min }}{f_{\max }(k+1)^{d}}\Bigg(d k^{d-1}\left(1-\frac{(k+1)^{2} \ell^{2}}{3 \tau(M)^{2}}\right)^{d}-\\
&\qquad \qquad k^{d}\left(\left(1+\frac{4 k^{2} \ell^{2}}{3 \tau(M)^{2}}\right)^{d}-\left(1-\frac{(k+1)^{2} \ell^{2}}{3 \tau(M)^{2}}\right)^{d}\right)\Bigg) \\
&\geq \frac{c_{d} f_{\min }}{f_{\max }(k+1)^{d}}\left(d k^{d-1} / 2-C_{4} k^{d} \frac{k^{2} \ell^{2}}{\tau(M)^{2}}\right) \geq \frac{C_{5}}{k}
\end{aligned}
\end{equation*}

where we used the inequality $C_{4} \frac{k^{2} \ell^{2}}{\tau(M)^{2}} \leq d k^{-1} / 4$ at the last line: this inequality holds as $\ell^{2}$ is of order $(\ln n / n)^{4 / d}$ and $k^{-3}$ is at least of order $(\ln n / n)^{3 / d}$.

If $Y_{1}, \ldots, Y_{N}$ are i.i.d. random variables of law $\mu$, conditioned on being in $B_{i}$, we therefore have
\begin{equation*}
\mathbb{P}\left(\forall j \in\{1, \ldots, N\},\left|Y_{j}-X_{i}\right| \notin[k \ell,(k+1) \ell] | X_{i}\right) \leq \exp \left(-C_{5} N / k\right)
\end{equation*}

For each ball $B_{i}$, we let $J_{i} \subseteq\{L+1, \ldots, n\}$ be the set of indexes $j>L$ such that $X_{j} \in B_{i}$. Assume that there exists a set of $A$ balls $B_{i_{1}}, \ldots, B_{i_{A}}$ that are pairwise disjoint. Then, the corresponding sets $J_{i}$ are also pairwise disjoint. Conditionally on $\mathcal{X}_{n}^{0}$ and on $N_{a}:=\left|J_{i_{a}}\right|$, the sets $\left\{X_{j}: j \in J_{i_{a}}\right\}$ are independent for $a=1, \ldots, A$, and each consists of a sample of $N_{i_{a}}$ independent points sampled according to $\mu$ conditioned on being in $B_{i_{a}}$. Therefore, if $E$ is the event that $\operatorname{Rad}(\mathcal{X}_{n})$ does not intersect $I_{k}$, we have
\begin{equation*}
\mathbb{P}(E | \mathcal{X}_{n}^{0},\left(N_{i}\right)_{i \leq L}) \leq \exp \left(-\sum_{a=1}^{A} N_{i_{a}} C_{5} / k\right).
\end{equation*}

The random variable $\sum_{a=1}^{A} N_{i_{a}}$ is the number of points of $\mathcal{X}_{n}^{1}$ in $\bigcup_{a=1}^{A} B_{i_{a}}$. It follows a binomial disribution of parameters $n-L$ and $p=\sum_{a=1}^{A} \mu\left(B_{i_{a}}\right) \geq C_{6} A(k \ell)^{d}$, so that we have

\begin{equation*}
\begin{aligned}
\mathbb{P}\left(E | \mathcal{X}_{n}^{0}\right) &\leq \mathbb{E}\left[\exp \left. \left(-\sum_{a=1}^{A} N_{i_{a}} C_{5} / k\right) \right| \mathcal{X}_{n}^{0}\right] \\
& \leq \exp \left(-C_{6}(n-L) A(k \ell)^{d}\left(1-e^{-C_{5} / k}\right)\right) \\
& \leq \exp \left(-C_{7}(n-L) A(k \ell)^{d} / k\right).
\end{aligned}
\end{equation*}

The quantity $A$ can be chosen equal to the maximal number of balls $B_{i}$ that are pairwise disjoint. A procedure to create a set of pairwise disjoint balls is the following. Start with $X_{i_{1}}=X_{1}$, and throw away all the points of $\mathcal{X}_{n}^{0}$ at distance less than $2(k+1) \ell$ from $X_{1}$. Take any point $X_{i_{2}}$ that has not been thrown away, and throw away all the remaining points that are distance less than $2(k+1) \ell$ from $X_{i_{2}}$. Repeating this procedure for $\tilde{A}$ steps until no points are left, we obtain a set of indexes for which the corresponding balls are pairwise disjoint. In particular, $\tilde{A} \leq A$. The number of points that are thrown away at the step $a$ follows a binomial distribution of parameters $m$ and $q$, where $m \leq L$ is the number of points in $M_{a}:=M \backslash \bigcup_{a^{\prime}<a} \mathcal{B}\left(X_{i_{a^{\prime}}}, 2(k+1) \ell\right)$, and, as long as $f_{\max } c_{d} a(k \ell)^{d} \leq 1 / 2$

\begin{equation*}
q=\frac{\mu\left(B_{i_{a}}\right)}{\mu\left(M_{a}\right)} \leq \frac{c_{d} f_{\max }(k \ell)^{d}}{1-a c_{d} f_{\max }(k \ell)^{d}} \leq C_{8}(k \ell)^{d}.
\end{equation*}

In particular, the number of points that have been thrown away after $a$ steps is stochastically dominated by the sum of $a$ independent binomial random variables of parameter $L$ and $C_{8}(k \ell)^{d}$, that is a binomial random variable $Z_{a}$ of parameters $a L$ and $C_{8}(k \ell)^{d}$. This implies that

\begin{equation*}
\mathbb{P}(A \leq a) \leq \mathbb{P}(\tilde{A} \leq a) \leq \mathbb{P}\left(Z_{a} \geq L\right).
\end{equation*}

Let $a=\left\lfloor 1 /\left(C_{9}(k \ell)^{d}\right)\right\rfloor$, where $C_{9}$ is choosen so that $f_{\max } c_{d} a(k \ell)^{d} \leq 1 / 2$ and $\mathbb{E} Z_{a}=a L C_{8}(k \ell)^{d} \leq$ $L / 2$. Then, $\mathbb{P}\left(Z_{a} \geq L\right) \leq \mathbb{P}\left(Z_{a}-\mathbb{E} Z_{a} \geq L / 2\right) \leq \exp \left(-C_{10} L\right)$ using Bernstein's inequality. Therefore, letting $L=\left(3 / C_{10}\right)(\ln n)$, we obtain that

\begin{equation*}
\begin{aligned}
\mathbb{P}(E) & \leq \mathbb{E}\left[\exp \left(-C_{7}(n-L) A(k \ell)^{d} / k\right)\right] \\
&\leq \mathbb{E}\left[\exp \left(-C_{7}(n-L) A(k \ell)^{d} / k\right) 1\{A \geq a\}\right]+\mathbb{P}(A \leq a) \\
& \leq \exp \left(-C_{11} n\left\lfloor 1 /\left(C_{9}(k \ell)^{d}\right)\right\rfloor(k \ell)^{d} / k\right)+n^{-3} \\
& \leq \exp \left(-C_{12} n^{1-1 / d}(\ln n)^{1 / d} /(2 \beta)\right)+n^{-3}
\end{aligned}
\end{equation*}

If $d \geq 2$, the first term in the above sum is smaller than $n^{-3}$. If $d=1$, it is equal to $n^{-C_{12} /(2 \beta)}$ and we choose $\beta=C_{12} / 6$ to conclude. 
\end{proof}

\subsection*{Proof of the lower bound in \eqref{eq4.4}}

 A first naive attempt to lower bound $t_{\lambda}(\mathcal{X}_{n})$ is the following. Remark that if two points $X_{1}, X_{2}$, at distance $2 t$, are such that the ball centered at their middle, of radius $t$, does not contain any point of $\mathcal{X}_{n}$, then $d_{H}(\operatorname{Conv}(\left\{X_{1}, X_{2}\right\}) | \mathcal{X}_{n})=t$. Fix $t>0$, and assume that there is some $t^{\prime} \in \operatorname{Rad}(\mathcal{X}_{n})$ smaller than $t$ such that $h(t^{\prime}, \mathcal{X}_{n})<t^{\prime}$. There must then exist a simplex of size at least 3 of radius smaller than $t$ in $\mathcal{X}_{n}$. In particular, there are three points $X_{1}, X_{2}$ and $X_{3}$ of $\mathcal{X}_{n}$ so that $X_{2}, X_{3} \in \mathcal{B}\left(X_{1}, 2 t\right)$. Therefore, according to Lemma \ref{lem2.3}, if $t \leq \tau(M) / 8$,

\begin{equation}\label{eq4.7}
\begin{aligned}
\mathbb{P}\left(t_{\lambda}(\mathcal{X}_{n})<t\right) & \leq \mathbb{P}\left(\exists X_{1}, X_{2}, X_{3} \text { with } X_{2}, X_{3} \in \mathcal{B}\left(X_{1}, 2 t\right)\right) \\
& \leq \mathbb{E}\left[\mathbb{P}\left(\exists X_{2}, X_{3} \in \mathcal{B}\left(X_{1}, 2 t\right) | X_{1}\right)\right] \\
&\leq \mathbb{E}\left[(n \mu(\mathcal{B}\left(X_{1}, 2 t))^{2}\right] \leq (\alpha_{d} f_{\max } n(13 t / 6)^{d})^{2}\right.  \leq C_{0}(n t^{d})^{2}.
\end{aligned}
\end{equation}

We know from the previous section that $t^{*}(\mathcal{X}_{n})$ is of order $t \simeq(\ln n / n)^{1 / d}$, while $\left(n t^{d}\right)^{2} \simeq(\ln n)^{2}$ for such a value of $t$. Hence, the previous inequality is far from sufficient to obtain Theorem \ref{thm4.6}. We therefore consider a more elaborate construction.

\begin{lemma}\label{lem4.9} Let $\delta>0$. For $t$ small enough (depending on $\mu$ and $\delta)$, there exist $K$ pairwise disjoint measurable subsets $U_{1}, \ldots, U_{K}$, so that $K \geq c_{u, \delta} t^{-d}$ and each set $U_{k}$ contains a ball $V_{k}$ of radius $t$ and satisfies
\begin{equation}\label{eq4.8}
\mu\left(U_{k}\right)=m(t):=\alpha_{d}(1+\delta) f_{\min } t^{d}.
\end{equation}
\end{lemma}

Before proving the lemma, note that we also have $K m(t) \leq 1$ by a union bound.

\begin{proof}
 Consider the collection $\mathcal{F}$ of balls $V$ of radius $t$ centered at a point of $M$ satisfying $\mu(V) \leq \alpha_{d}(1+\delta) f_{\min } t^{d}$, and let $A_{t}$ be the set of the centers of such balls. By Besicovitch's covering theorem \cite[Theorem 2.8.14]{federer1969geometric}, there exist $N_{M}$ collections $\mathcal{G}_{1}, \ldots, \mathcal{G}_{N_{M}}$ of disjoint balls in $\mathcal{F}$ such that
\begin{equation*}
A_{t} \subseteq \bigcup_{l=1}^{N_{M}} \bigcup_{V \in \mathcal{G}_{l}} V.
\end{equation*}

Letting $K_{t}$ be the maximal number of pairwise disjoint balls in $\mathcal{F}$, we have $\mu\left(A_{t}\right) \leq N_{M} K_{t} \alpha_{d}(1+ \delta) f_{\min } t^{d}$. By the Lebesgue differentiation theorem, for almost all points $x \in M$ with $f(x)<$ $(1+\delta) f_{\min }$, we have 
\[\lim _{t \rightarrow 0} \frac{\mu(\mathcal{B}(x, t))}{ \alpha_{d} t^{d}} <f_{\min }(1+\delta).\]
 For such a $x$, we then have $x \in$ $\liminf _{t \rightarrow 0} A_{t}.$ Therefore,
\begin{align*}
c_{\mu}&=\mu\left(\left\{x \in M: f(x)<(1+\delta) f_{\min }\right\}\right) \leq \mu\left(\liminf _{t \rightarrow 0} A_{t}\right) \leq \liminf _{t \rightarrow 0} \mu\left(A_{t}\right) \\
&\leq N_{M} \alpha_{d}(1+\delta) \liminf _{t \rightarrow 0} K_{t} t^{d}.
\end{align*}

By the definition of $f_{\min }, c_{\mu}>0$. Therefore, for $t$ small enough, we have the inequality $K_{t} \geq \frac{c_{\mu}}{2 N_M \alpha_{d}(1+\delta)} t^{-d}$. Let $V_{1}, \ldots, V_{K(t)}$ be a set of pairwise disjoint balls in $\mathcal{F}$. By construction, each ball $V_{k}$ satisfies $\mu\left(V_{k}\right) \leq m(t)$. Also, we have $\mu\left(V_{k}\right) \geq \alpha_{d} f_{\min } t^{d} / 2$ for $t$ small enough by Lemma \ref{lem2.3}. This implies by a union bound that $1 \geq K(t) \alpha_{d} f_{\min } t^{d} / 2$. Therefore, $K(t) m(t) \leq 2(1+\delta)$. We define $K=\lfloor K(t) /(2(1+\delta))\rfloor$, a number that satisfies $K \geq c_{\mu, \delta} t^{-d}$ and $K m(t) \leq 1$.

Eventually, we build the sets $U_{k}$ by induction by choosing any measurable set $W_{k}$ in $M \backslash(\bigcup_{k^{\prime}<k} U_{k^{\prime}} \cup V_{k} )$ with $\mu\left(W_{k}\right)=m(t)-\mu\left(V_{k}\right) \geq 0$ and letting $U_{k}=V_{k} \cup W_{k}$. This is possible as
\begin{equation*}
\mu\left(M \backslash(\bigcup_{k^{\prime}<k} U_{k^{\prime}} \cup V_{k}) \right) \geq 1-(k-1) m(t)-\mu\left(V_{k}\right) \geq m(t)-\mu\left(V_{k}\right).
\end{equation*}

By construction, $\mu\left(U_{k}\right)=m(t)$ for every $k$. Note that we used the fact that for any $A \subseteq M$ and $0 \leq p \leq \mu(A)$, there exists a subset $V \subseteq A$ with $\mu(V)=p$: this holds as $\mu$ is absolutely continuous with respect to the volume measure on $M$.
\end{proof}

\begin{figure}
\centering
\includegraphics[width=0.7\textwidth]{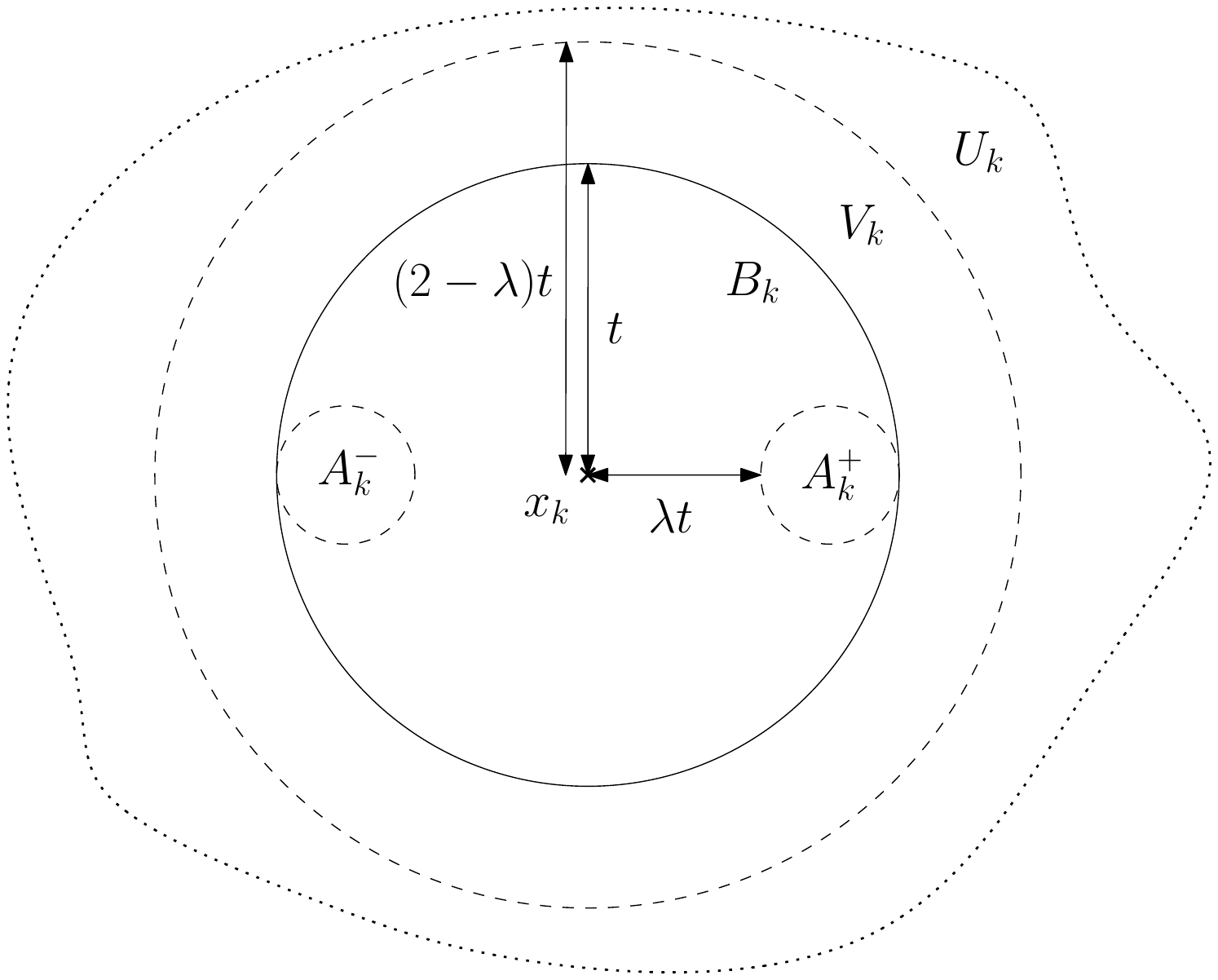}
\caption{ Any ball with diameter whose one extremity is in $A_{k}^{-}$and the other in $A_{k}^{+}$is included in $U_{k}$.
}\label{fig6}
\end{figure}

We fix such a partition in the following, with balls $V_{k}$ of radius $(2-\lambda) t$. We write $m$ for $m((2-\lambda) t)$. Let $B_{k}$ be the ball sharing its center with $V_{k}$, of radius $t$. For $W \subseteq M$, let $N(W)$ be the number of points of $\mathcal{X}_{n}$ in $W$. We also write $N_{k}$ for $N\left(U_{k}\right)$. Let $x_{k}$ be the center of $B_{k}$ and $e$ be a unit vector in $T_{x} M$, and denote by $A_{k}^{+}$(resp. $\left.A_{k}^{-}\right)$the ball of radius $(1-\lambda) t / 2$ centered at $x^{+}=x_{k}+e(1+\lambda) t / 2$ (resp. $\left.x^{-}=x_{k}-e(1+\lambda) t / 2\right)$, see Figure \ref{fig6}.

\begin{lemma}\label{lem4.10}
Fix $k=1, \ldots, K$. If $h(t, \mathcal{X}_{n})<\lambda t$ and $N_{k}=2$, then we cannot have both $N(A_{k}^{+})=1$ and $N(A_{k}^{-})=1$
\end{lemma}

\begin{proof} Let $\sigma=\mathcal{X}_{n} \cap U_{k}$. Assume by contradiction that $h(t, \mathcal{X}_{n})<\lambda t$, $N_{k}=2$, and $N(A_{k}^{+})=N(A_{k}^{-})=1$. Then, $\sigma$ is made of two points, $x_{1}$ and $x_{2}$, respectively in $A_{k}^{+}$ and $A_{k}^{-}.$ As both points belong to $B_{k}$, we have $r(\sigma) \leq t$. Therefore, $d_{H}(\operatorname{Conv}(\sigma) | \mathcal{X}_{n}) \leq h(t, \mathcal{X}_{n})<\lambda t$. 
 In particular, the middle point $x_{0}$ of $x_{1}$ and $x_{2}$ is at distance less than $\lambda t$ from $\mathcal{X}_{n}$. Let us show that $\mathcal{B}_{M}\left(x_{0},\left|x_{1}-x_{0}\right|\right) \subseteq V_{k}$. 
  If this is the case, then $d(x_{0}, \mathcal{X}_{n})=\left|x_{1}-x_{2}\right| / 2 \geq \lambda t$, a contradiction with having $d_{H}(\operatorname{Conv}(\sigma) | \mathcal{X}_{n})<\lambda t$. Let $z \in \mathcal{B}_{M}\left(x_{0},\left|x_{1}-x_{0}\right|\right)$ and denote by $\pi_{e}$ the projection on $e$. Then,
\begin{equation*}
\begin{aligned}
\left|z-x_{k}\right| \leq &\left|z-x_{0}\right|+\left|x_{0}-x_{k}\right| \leq \frac{\left|x_{1}-x_{2}\right|}{2}+\left|\pi_{e}\left(x_{0}-x_{k}\right)\right|+\left|\pi_{e}^{\perp}\left(x_{0}-x_{k}\right)\right| \\
& \leq t+\frac{(1-\lambda) t}{2}+\frac{(1-\lambda) t}{2} \leq(2-\lambda) t
\end{aligned}
\end{equation*}
concluding the proof.
\end{proof}

Denote by $F_{k}$ the complementary event of the event $N(A_{k}^{+})=N(A_{k}^{-})=1$. We obtain the bound
\begin{equation*}
\begin{aligned}
&\mathbb{P}\left(h(t, \mathcal{X}_{n})<\lambda t\right) \leq \mathbb{P}\left(\forall k=1, \ldots, K, N_{k} \neq 2 \text { or }\left(N_{k}=2 \text { and } F_{k}\right)\right) \\
&=\mathbb{E}\left[\mathbb{P}\left(\forall k=1, \ldots, K, N_{k} \neq 2 \text { or }\left(N_{k}=2 \text { and } F_{k}\right) |\left(N_{k}\right)_{k=1, \ldots, K}\right)\right] \\
& \leq \mathbb{E}\left[\prod_{k=1}^{K}\left(\mathbf{1}\left\{N_{k} \neq 2\right\}+\mathbb{P}\left(F_{k} | N_{k}=2\right) \mathbf{1}\left\{N_{k}=2\right\}\right)\right] \\
& \leq \mathbb{E}\left[\prod_{k=1}^{K}\left(1-\left(1-\mathbb{P}\left(F_{k} | N_{k}=2\right)\right) \mathbf{1}\left\{N_{k}=2\right\}\right)\right].
\end{aligned}
\end{equation*}

\begin{lemma}\label{lem4.11}
There exists a positive constant $C_{1}$ such that
\begin{equation}
\mathbb{P}\left(F_{k} | N_{k}=2\right) \leq e^{-C_{1}} \text { for } k=1, \ldots, K.
\end{equation}
\end{lemma}

\begin{proof}
 If $\left|x_{+}-x_{k}\right| \leq t \leq 7 \tau(M) / 24$, then there exists $y_{+} \in M$ that satifies $\pi_{x_{k}}\left(y_{+}-x_{k}\right)=x_{+}-x_{k}$ by Lemma \ref{lem2.2}. Furthermore, we have $\left|y_{+}-x_{k}\right| \leq 8 t / 7$ and, by Lemma \ref{lem2.1}, we have $\left|y_{+}-x_{+}\right| \leq$ $(8 t / 7)^{2} /(2 \tau(M))=32 t^{2} /(49 \tau(M))$. In particular,
\begin{equation*}
\mathcal{B}\left(x_{+},(1-\lambda) t / 2\right) \supseteq \mathcal{B}\left(y_{+},(1-\lambda) t / 2-32 t^{2} /(49 \tau(M))\right) \supseteq \mathcal{B}\left(y_{+},(1-\lambda) t / 4\right)
\end{equation*}

if $t \leq 49(1-\lambda) \tau(M) / 128$. According to Lemma \ref{lem2.3}, we therefore have, also assuming that $t \leq \tau(M) / 4$
\begin{equation*}
\mu\left(\mathcal{B}\left(x_{+},(1-\lambda) t / 2\right)\right) \geq f_{\min } \alpha_{d}\left(\frac{(1-\lambda) t}{4} \frac{47}{48}\right)^{d}
\end{equation*}
and the same inequality holds for $x_{-}$.

Let $Y_{1}, Y_{2}$ be two independent random variables sampled according to $\mu$, conditioned on being in $U_{k}$. Then, as $\mu\left(U_{k}\right)=m=\alpha_{d}(1+\delta) f_{\min }(2-\lambda)^{d} t^{d}$,

\begin{equation*}
\begin{aligned}
\mathbb{P}\left(F_{k} | N_{k}=2\right) &=1-2 \mathbb{P}\left(Y_{1} \in A_{k}^{+}\right) \mathbb{P}\left(Y_{2} \in A_{k}^{-}\right) \\
&=1-2 \frac{\mu\left(\mathcal{B}\left(x^{+},(1-\lambda) t / 2\right)\right) \mu\left(\mathcal{B}\left(x^{-},(1-\lambda) t / 2\right)\right)}{\mu\left(U_{k}\right)^{2}} \\
& \leq 1-2\left(\frac{\left(\frac{47}{48} \frac{1-\lambda}{4}\right)^{d}}{(1+\delta)(2-\lambda)^{d}}\right)^{2} \leq e^{-C_{1}}
\end{aligned}
\end{equation*}

where $C_{1}=2\left(\frac{\left(\frac{47}{48} \frac{1-\lambda}{4}\right)^{d}}{(1+\delta)(2-\lambda)^{d}}\right)^{2}$.
\end{proof}

We finally obtain

\begin{equation}
\mathbb{P}\left(h(t, \mathcal{X}_{n})<\lambda t\right) \leq \mathbb{E}\left[\exp \left(-C_{1} \sum_{k=1}^{K} \mathbf{1}\left\{N_{k}=2\right\}\right)\right].
\end{equation}

\begin{lemma}\label{lem4.12}
Assume that $n m \leq \max \left(m^{-1},(\ln n)^{2}\right)$. Let $\phi: x \in[0,+\infty) \mapsto \min (1, x) e^{-x}$. Then,
\begin{equation}
\mathbb{E}\left[\exp \left(-C_{1} \sum_{k=1}^{K} 1\left\{N_{k}=2\right\}\right)\right] \leq C_{2} \exp \left(-C_{3} n \phi(n m)\right)
\end{equation}

for some positive constants $C_{2}, C_{3}$.
\end{lemma}

Lemma \ref{lem4.12} relies on concentration inequalities and is proved in Appendix \ref{appA}. As $m$ is of order $t^{d}$, the condition $n m \leq \max \left(m^{-1},(\ln n)^{2}\right)$ is satisfied as long as $t^{d} \ll(\ln n)^{2} / n$. Remark also that the function $\phi$ is increasing on $[0,1]$ and decreasing on $[1,+\infty)$.

Assume that $t_{1} \leq t \leq t_{2}$, where

\begin{equation*}
t_{1}=\frac{1}{2-\lambda}\left(\frac{1}{\alpha_{d} f_{\min }(1+\delta) n}\right)^{1 / d} \text { and } t_{2}=\frac{1}{2-\lambda}\left(\frac{\beta \ln n}{\alpha_{d} f_{\min }(1+\delta) n}\right)^{1 / d}
\end{equation*}

for some $0<\beta<1$. Then, $1 \leq n m \leq \beta \ln n$, so that $\phi(n m) \geq \phi(\beta \ln n)=n^{-\beta}$ and

\begin{equation}\label{eq4.11}
\forall t \in\left[t_{1}, t_{2}\right], \quad \mathbb{P}\left(h(t, \mathcal{X}_{n})<\lambda t\right) \leq C_{2} \exp \left(-C_{3} n^{1-\beta}\right) \leq n^{-2}
\end{equation}

for $n$ large enough. Assume now that $t \in\left[t_{0}, t_{1}\right]$, where \[t_{0}=\frac{1}{2-\lambda}\left(\frac{\kappa \ln n}{\alpha_{d} f_{\min }(1+\delta) n^{2}}\right)^{1 / d}\] for some $\kappa>0$. Then, $\kappa(\ln n / n) \leq n m \leq 1$, so that $\phi(n m) \geq \phi(\kappa(\ln n / n)) \geq \kappa \ln n /(2 n)$ for $n$ large enough. Choosing $\kappa \geq 4 / C_{3}$, we obtain that

\begin{equation}\label{eq4.12}
\forall t \in\left[t_{0}, t_{1}\right], \quad \mathbb{P}\left(h(t, \mathcal{X}_{n})<\lambda t\right) \leq C_{2} n^{-C_{3} \kappa / 2} \leq C_{2} n^{-2}.
\end{equation}

The picture is now as follows. We know from \eqref{eq4.7} that $t_{\lambda}(\mathcal{X}_{n}) \geq t_{0}$ with probability at least $1-c_{3}(\ln n / n)^{2}$. For each $t$ between $t_{0}$ and $t_{2}$, we also have $h(t, \mathcal{X}_{n}) \geq \lambda t$ with probability at least $n^{-2}$ (at least for $n$ large enough with respect to $\lambda$ and $\mu$). Consider a sequence $t^{(i)}$ with $t^{(0)}=t_{0}$ and $t^{(i+1)}=t^{(i)} / \lambda$ for $i=0, \ldots, I$, with $I$ chosen so that

\begin{equation*}
t_{2} / \lambda \leq t^{(I)} \leq t_{2}
\end{equation*}

Assume that $t_{\lambda}(\mathcal{X}_{n}) \geq t_{0}$ and that $h(t^{(i)}, \mathcal{X}_{n}) \geq \lambda t^{(i)}$ for every $i$. If $t$ belongs to the interval $\left[t^{(i)}, t^{(i+1)}\right]$, then $h(t, \mathcal{X}_{n}) \geq h(t^{(i)}, \mathcal{X}_{n}) \geq \lambda t^{(i)} \geq \lambda^{2} t$. Therefore, $t_{\lambda^{2}}(\mathcal{X}_{n}) \geq t_{2}$ Let $\lambda^{\prime}=\lambda^{2}$. As $I$ is of order $\ln n$, by a union bound, we obtain that, for any $0<\beta, \delta<1$, $\lambda^{\prime} \in(0,1)$ and $n$ large enough

\begin{equation}
\label{eq4.13}
\begin{aligned}
\mathbb{P}&\left(t_{\lambda^{\prime}}(\mathcal{X}_{n}) \leq \frac{1}{2-\sqrt{\lambda^{\prime}}}\left(\frac{\beta \ln n}{\alpha_{d} f_{\min }(1+\delta) n}\right)^{1 / d}\right) \\
 &\leq \mathbb{P}\left(t_{\lambda}(\mathcal{X}_{n})<t_{0}\right)+\sum_{i=0}^{I} \mathbb{P}\left(h(t^{(i)}, \mathcal{X}_{n})<\lambda t^{(i)}\right) \\
& \leq c_{3}(\ln n / n)^{2}+c_{4}(\ln n) n^{-2} \\
& \leq c_{5}(\ln n / n)^{2}.
\end{aligned}
\end{equation}

\begin{lemma}\label{lem4.13} Let $A \subseteq M$. Let $0<\lambda \leq \lambda^{\prime}<1$. Then, $t_{\lambda}(A) \geq \frac{\lambda^{\prime}}{\lambda} t_{\lambda^{\prime}}(A)$.
\end{lemma}

\begin{proof} The function $h(\cdot, A)$ is nondecreasing, and is therefore larger than $\lambda^{\prime} t_{\lambda^{\prime}}(A)$ for $t \geq t_{\lambda^{\prime}}(A)$. Therefore, for $t \in\left[t_{\lambda^{\prime}}(A),\left(\lambda^{\prime} / \lambda\right) t_{\lambda^{\prime}}(A)\right]$, we have $h(t, A) \geq \lambda^{\prime} t_{\lambda^{\prime}}(A) \geq \lambda t$, yielding the conclusion.
\end{proof}

Let $0<\lambda<(1+b)^{-1 / d}$. From Proposition \ref{prop3.9}, we know that with probability $1-$ $c(\ln n)^{d-1} n^{-b}$, we have $\varepsilon(\mathcal{X}_{n}) \leq(1+b)^{1 / d}\left(\ln n /\left(n \alpha_{d} f_{\min }\right)\right)^{1 / d}$. For any $r>1$, if $n$ is large enough, by Proposition \ref{prop3.4}, this entails that $t^{*}(\mathcal{X}_{n}) \leq r \varepsilon(\mathcal{X}_{n})$. Choose $\lambda^{\prime}, \beta$ and $r$ close enough to 1, and $\delta$ small enough, so that
\begin{equation*}
\frac{\lambda^{\prime}}{\lambda\left(2-\sqrt{\lambda^{\prime}}\right)} \frac{\beta^{1 / d}}{(1+\delta)^{1 / d}} \geq r(1+b)^{1 / d}.
\end{equation*}

Such a choice is possible as $\frac{1}{\lambda}>(1+b)^{1 / d}$. Then, assuming that the complementary of the event described in \eqref{eq4.13} also holds, we have
\begin{equation*}
t_{\lambda}(\mathcal{X}_{n}) \geq \frac{\lambda^{\prime}}{\lambda} t_{\lambda^{\prime}}(\mathcal{X}_{n}) \geq \frac{\lambda^{\prime}}{\lambda} \frac{1}{2-\sqrt{\lambda^{\prime}}}\left(\frac{\beta \ln n}{\alpha_{d} f_{\min }(1+\delta) n}\right)^{1 / d} \geq r \varepsilon(\mathcal{X}_{n}) \geq t^{*}(\mathcal{X}_{n})
\end{equation*}

As the probability appearing in \eqref{eq4.13} is smaller than a quantity of order $(\ln n)^{2} n^{-2} \leq(\ln n)^{a} n^{-b}$ for any $0<b \leq 2$, we obtain inequality \eqref{eq4.4}, concluding the proof of the first statement of Theorem \ref{thm4.6}.

\subsection*{Proof of \eqref{eq4.4}}
 Consider a $n$-sample $\left\{X_{1}, \ldots, X_{n}\right\}$ on the circle $M$ of radius 1. Without loss of generality, we assume that $X_{1}=(0,1)$. Each point $X_{i}$ is equal to $\exp \left(i \theta_{i}\right)$ where $\theta_{i} \in[0,2 \pi)$. Consider the ordering
\begin{equation*}
0=\theta_{(1)} \leq \cdots \leq \theta_{(n)}
\end{equation*}

and the associated points $X_{(1)}, \ldots, X_{(n)}$. Define the spacings $V_{i}=\theta_{(i+1)}-\theta_{(i)}$ for $i=1, \ldots, n$ (with by convention $\left.\theta_{(n+1)}=2 \pi\right)$. The corresponding edge lenth $\left|X_{(i+1)}-X_{(i)}\right|=: 2 t_{i}$ satisfies $V_{i}=\arccos \left(1-2 t_{i}^{2}\right)$.

We write $V_{(1)} \leq \cdots \leq V_{(n)}$ for the ordered spacings (and $t_{(1)} \leq \cdots \leq t_{(n)}$ for the associated lengths). Note that we have $t^{*}(\mathcal{X}_{n})=t_{(n)}$. The next lemma asserts that the convexity defect function cannot increase too much between two consecutive $t_{(i)} \mathrm{s}$.

\begin{lemma}\label{lem4.14} For $t \in\left[t_{(i)}, t_{(i+1)}\right)$, we have $h(t, \mathcal{X}_{n}) \leq t_{(i)}+t_{(i+1)}^{2}$.
\end{lemma}

\begin{proof}
 Let $\left[X_{(k)}, X_{(l)}\right]$ be an edge of length smaller than $2 t$ with $k<l$. We assume without loss of generality that $X_{0}$ does not lie on the arc between $X_{(k)}$ and $X_{(l)}$. Let $x$ be a point on this edge, of the form $r e^{i \theta}$ for some angle $\theta_{(k)} \leq \theta \leq \theta_{(l)}$. The angle $\theta$ belongs to the segment $\left[\theta_{(j)}, \theta_{(j+1)}\right]$ for some index $j$. As $t<t_{(i+1)}$ we have $t_{j}<t_{(i+1)}$, that is $t_{j} \leq t_{(i)}$. The ray of angle $\theta$ hits the line $\left[X_{(j)}, X_{(j+1)}\right]$ at some point $y$, and

\begin{equation*}
d(x, \mathcal{X}_{n}) \leq|x-y|+d(y, \mathcal{X}_{n}) \leq d(x, M)+t_{j} \leq t^{2}+t_{(i)}
\end{equation*}

by Lemma \ref{lem3.2}. As $t \leq t_{(i+1)}$, we obtain the conclusion.
\end{proof}

Let $\lambda>(1+b)^{-1}$ and fix an arbitrary $\lambda^{\prime}$ satisfying $(1+b)^{-1}<\lambda^{\prime}<\lambda$. Assume that $t_{(n-1)} \leq \lambda^{\prime} t_{(n)}$. Then, for $t \in\left[t_{(n-1)}, t_{(n)}\right)$, we have according to the previous lemma that
\begin{equation*}
h(t, \mathcal{X}_{n}) \leq \lambda^{\prime} t_{(n)}+t_{(n)}^{2}.
\end{equation*}
Choosing $t \in I:=\left[t_{(n)}\left(\lambda^{\prime}+t_{(n)}\right) / \lambda, t_{(n)}\right)$, we have $h(t, \mathcal{X}_{n})<\lambda t$. The interval $I$ satisfies the conditions of Lemma \ref{lem4.7}, and therefore intersects $\operatorname{Rad}(\mathcal{X}_{n})$ with probability $1-n^{-2}$. In particular, $t_{\lambda}(\mathcal{X}_{n})$ is smaller than the upper endpoint of $I$, that is $t_{(n)}=t^{*}(\mathcal{X}_{n})$. Note that such a choice of $t$ is possible as long as $\lambda-\lambda^{\prime}>t_{(n)}$. To put it another way, we have
\begin{equation}\label{eq4.14}
\mathbb{P}\left(t_{\lambda}(\mathcal{X}_{n})<t^{*}(\mathcal{X}_{n})\right) \leq \mathbb{P}\left(t_{(n-1)} \leq \lambda^{\prime} t_{(n)}\right)+\mathbb{P}\left(\lambda-\lambda^{\prime}<t_{(n)}\right)+n^{-2}.
\end{equation}

\begin{figure}
\centering
\includegraphics[width=0.5\textwidth]{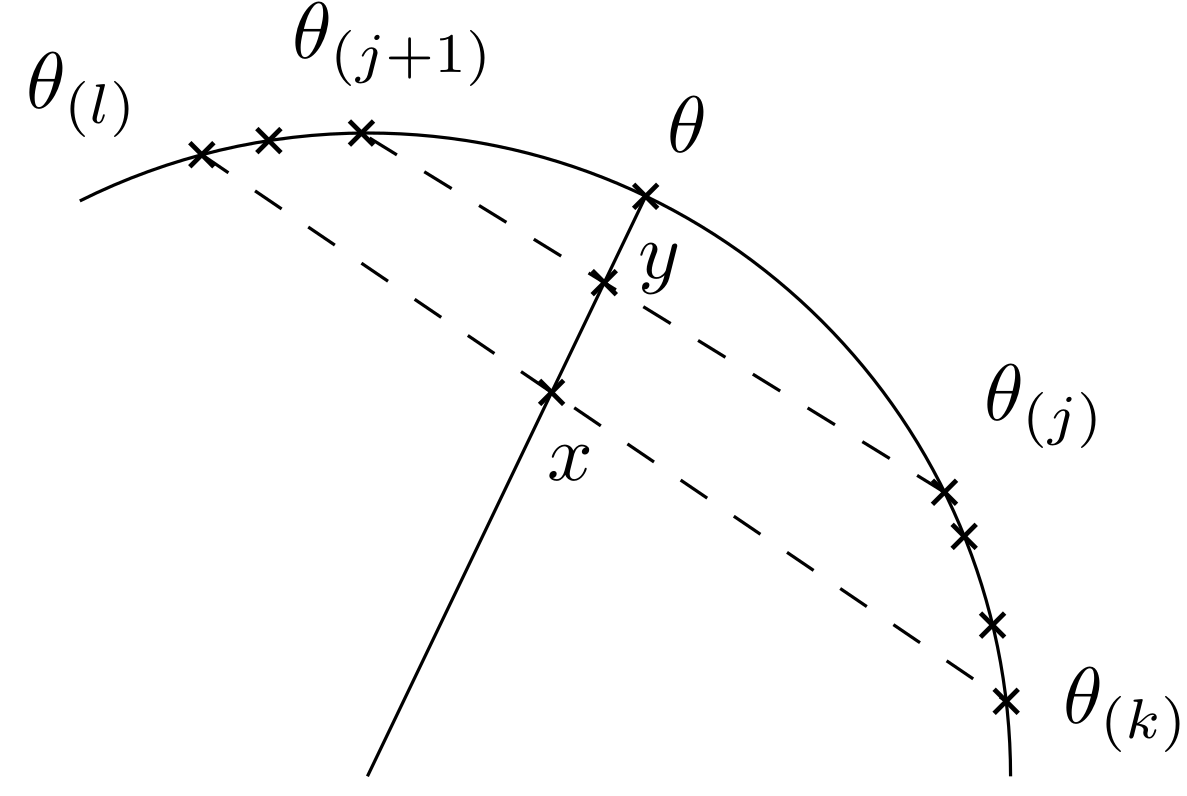}
\caption{Construction used in the proof of Lemma \ref{lem4.14}. The distance between $X_{(j)}=e^{i \theta_{(j)}}$ and $X_{(j+1)}=e^{i \theta_{(j+1)}}$ is equal to $2 t_{j}$, while the distance between $X_{(k)}=e^{i \theta_{(k)}}$ and $X_{(l)}=e^{i \theta_{(l)}}$ is smaller than $2 t$.}
\end{figure}

The second probability in the above equation is exponentially small by Proposition \ref{prop3.9}. It remains to study the probability that $t_{(n-1)} \leq \lambda^{\prime} t_{(n)}$. Let $A_{1}, \ldots, A_{n}$ be a $n$-sample following an exponential distribution. According to \cite[Section 6.4]{david1970order}, we have
\begin{equation*}
\left(V_{1}, \ldots, V_{n}\right) \sim 2 \pi\left(\frac{A_{1}}{\sum_{i=1}^{n} A_{i}}, \ldots, \frac{A_{n}}{\sum_{i=1}^{n} A_{i}}\right)
\end{equation*}

In particular, the law of $V_{(n)} / V_{(n-1)}$ is equal to the law of $A_{(n)} / A_{(n-1)}$, the largest of the $A_{i}$s divided by the second largest. Furthermore, according to \cite[Theorem 2.1]{zhang2017limit}, we have, for any $s>1$,
\begin{equation}\label{eq4.15}
\begin{aligned}
\mathbb{P}\left(A_{(n)} / A_{(n-1)} \geq s\right) &=n(n-1) \sum_{k=0}^{n-2} \binom{n-2}{k} \frac{(-1)^{n-2-k}}{n-1-k+s} \\
&=n(n-1) \sum_{k=0}^{n-2} \binom{n-2}{k}(-1)^{n-2-k} \int_{0}^{1} x^{n-2-k+s} \mathrm{~d} x \\
&=n(n-1) \int_{0}^{1} x^{s}(1-x)^{n-2} \mathrm{~d} x \\
&=n(n-1) B(s+1, n-1)\\
& \sim n^{2} \Gamma(s+1) n^{-(s+1)} \sim \Gamma(s+1) n^{1-s},
\end{aligned}
\end{equation}

where $B$ is the Beta function. Also, by writing a Taylor expansion of arccos at 1, we obtain that for $t_{(n)}$ small enough,
\begin{equation*}
\frac{V_{(n)}}{V_{(n-1)}}=\frac{\arccos \left(1-2 t_{(n)}^{2}\right)}{\arccos \left(1-2 t_{(n-1)}^{2}\right)} \leq \frac{t_{(n)}}{t_{(n-1)}}\left(1+\frac{5 t_{(n)}^{2}}{24}\right)
\end{equation*}

If $t_{(n-1)} \leq \lambda^{\prime} t_{(n)}$, then we have $\frac{V_{(n)}}{V_{(n-1)}} \geq\left(\lambda^{\prime}\right)^{-1}\left(1+\frac{5 t_{(n)}^{2}}{24}\right) \geq 1+b$ if $t_{(n)}$ is smaller than some constant $c_{0}$ (recall that $\left.\lambda^{\prime}>(1+b)^{-1}\right)$. Therefore,
\begin{equation*}
\mathbb{P}\left(t_{(n-1)} \leq \lambda^{\prime} t_{(n)}\right) \geq \mathbb{P}\left(t_{(n)}>c_{0}\right)+\mathbb{P}\left(\frac{V_{(n)}}{V_{(n-1)}} \geq 1+b\right)
\end{equation*}

for some small constant $c_{0}$ (depending on the distance between $(1+b)^{-1}$ and $\lambda^{\prime}$ ). The first probability is exponentially small, and the second one is of order $n^{-b}$ by \eqref{eq4.15}. Inequality \eqref{eq4.14} then yields the conclusion.

\section{Sampling with noise}

So far, we have always considered that the point cloud $\mathcal{X}_{n}$ lies exactly on the manifold $M$. However, all the constructions presented are stable with respect to tubular noise. 

Let $0<\gamma<\tau_{\min }$. Let $X=Y+Z$, with the law $\nu$ of $Y$ being in $\mathcal{Q}_{\tau_{\min }, f_{\min }, f_{\max }}^{d}$ and $Z \in T_{Y} M^{\perp}$ satisfying $|Z| \leq \gamma.$ We let $\mathcal{Q}_{\tau_{\min }, f_{\min }, f_{\max }}^{d, \gamma}$ be the set of laws of such random variables $X$. Observe that, as we do not assume that the conditional noise $Z  |  Y$ is centered, the model is not identifiable, that is $M$ is not determined by the law $\mu$ of $X$. To simplify matters, for each law $\mu \in \mathcal{Q}_{\tau_{\min },\fmin,\fmax}^{d, \gamma}$, we will make an arbitrary choice among the admissible couples $(Y, Z)$ with $Y+Z \sim \mu$. The "underlying manifold $M$ of the law $\mu$ " will be the support of the law of $Y$, while the results of this section will hold for any choice of couple $(Y, Z)$.

\begin{rem}[On the orthogonality assumption] The assumption that the noise is orthogonal (that is $\left.Z \in T_{Y} M^{\perp}\right)$ is not restrictive. Let $\gamma<\tau_{\min }$, $\nu \in \mathcal{Q}_{\tau_{\min }, f_{\min }, f_{\max }}^{d}$ with density $f$ and $Y \sim \nu$. Let $Z$ be any random variable supported on $\mathcal{B}(0, \gamma)$, and $X=Y+Z$ (without necessarily $\left.Z \in T_{Y} M^{\perp}\right)$. We may write $X=\pi_{M}(X)+\left(X-\pi_{M}(X)\right)=Y^{\prime}+Z^{\prime}$. By Lemma \ref{lem2.2}, we have $Z^{\prime} \in T_{Y}, M^{\perp}$. Furthermore, the density of $Y^{\prime}$ can be explicitely computed in terms of the density of $f$ and of the Jacobian of the function $G_{z}: y \in M \mapsto \pi_{M}(y+z)$. More precisely, one can show that $G_{z}$ is bijective, of class $\mathcal{C}^{1}$, and, by a change of variable, that the density $f^{\prime}$ of $Y^{\prime}$ at $y$ is given by $\mathbb{E}\left[f(G_{Z}^{-1}(y)) J(G_{Z}^{-1})(y)\right]$. The derivative of $G_{Z}$ is expressed in terms of the second fundamental form of $M$ (whose operator norm is bounded by the reach $\tau(M)$ \cite{geo:niyogi2008finding}). In particular, the Jacobian is upper and lower bounded, so that $f^{\prime}$ is lower and upper bounded on $M$. In other words, the law of $X$ belongs to $\mathcal{Q}_{\tau_{\min }, a f_{\min }, f_{\max } / a}^{d, \gamma}$ for some $0<a<1$ depending on $d, \tau_{\min }$ and $\tau_{\min }-\gamma$.
\end{rem}

We first show that the $t$-convex hull with parameter $t$ of order $(\ln n / n)^{1 / d}$ has a risk of the same order if tubular noise is added.

\begin{prop}\label{prop5.2}
 Let $A, B \subseteq \mathbb{R}^{D}$ and let $d_{H}(A, B) \leq \gamma$. Then,
\begin{equation}
d_{H}(\operatorname{Conv}(t, A)  |  \operatorname{Conv}(t+\gamma, B)) \leq \gamma
\end{equation}
\end{prop}

\begin{proof} Let $\sigma \subseteq A$. By definition, there exists $\sigma^{\prime} \subseteq B$ such that $d_{H}(\sigma  |  \sigma^{\prime}) \leq \gamma$. We have $r\left(\sigma^{\prime}\right) \leq r(\sigma)+\gamma \leq t+\gamma$ (see \cite[Lemma 16]{geo:attali2013vietoris}) and $d_{H}(\operatorname{Conv}(\sigma)  |  \operatorname{Conv}(\sigma^{\prime})) \leq \gamma$.
\end{proof}

Let $\mathcal{X}_{n}=\left\{X_{1}, \ldots, X_{n}\right\}$ be a $n$-sample of law $\mu$, with $\mathcal{Y}_{n}=\left\{Y_{1}, \ldots, Y_{n}\right\}$ the corresponding sample on $M$ (that is $Y_{i}=\pi_{M}(X_{i})$. If $t \geq t^{*}(\mathcal{Y}_{n})+\gamma$, then
\begin{equation*}
\begin{aligned}
d_{H}(M  |  \operatorname{Conv}(t, \mathcal{X}_{n})) &\leq d_{H}(M  |  \operatorname{Conv}(t-\gamma, \mathcal{Y}_{n}))+d_{H}(\operatorname{Conv}(t-\gamma, \mathcal{Y}_{n})  |  \operatorname{Conv}(t, \mathcal{X}_{n})) \\
& \leq \frac{(t-\gamma)^{2}}{\tau(M)}+\gamma \\
\text { and } d_{H}(\operatorname{Conv}(t, \mathcal{X}_{n})  |  M) & \leq d_{H}(\operatorname{Conv}(t, \mathcal{X}_{n})  |  \operatorname{Conv}(t+\gamma, \mathcal{Y}_{n}))+d_{H}(\operatorname{Conv}(t+\gamma, \mathcal{Y}_{n})  |  M) \\
& \leq \gamma+\frac{(t+\gamma)^{2}}{\tau(M)}.
\end{aligned}
\end{equation*}

Therefore, we obtain that, for $t \geq t^{*}\left(\mathcal{Y}_{n}\right)+\gamma$,
\begin{equation}
d_{H}(\operatorname{Conv}(t, \mathcal{X}_{n}), M) \leq \frac{(t+\gamma)^{2}}{\tau(M)}+\gamma.
\end{equation}

Assume that $\gamma \leq \eta(\ln n / n)^{2 / d}$ for some $\eta>0$ and let $\tilde t_{n}=2 t_{n}$, where $t_{n}$ is the radius appearing in Theorem \ref{thm1.1}. The probability that $\tilde{t}_{n} \leq t^{*}\left(\mathcal{Y}_{n}\right)+\gamma$ is smaller than the probability that $t_{n} \leq t^{*}\left(\mathcal{Y}_{n}\right)$, a probability that we control by Proposition \ref{prop3.9}. As $\tilde{t}_{n}+\gamma \leq 3 t_{n}$ for $n$ large enough, we obtain that
\begin{equation*}
\mathbb{E}\left[d_{H}(\operatorname{Conv}(\tilde{t}_{n}, \mathcal{X}_{n}), M)\right] \leq\left(\frac{c_{1}}{\tau_{\min }\left(\alpha_{d} f_{\min }\right)^{2 / d}}+\eta\right)\left(\frac{\ln n}{n}\right)^{2 / d}
\end{equation*}

for some absolute constant $c_{1}$.

Let us now analyze how the selection procedure is impacted by the presence of noise. We mimick the proof of Theorem \ref{thm4.6}. Let $0<b \leq 2$ and let $0<\lambda<(1+b)^{-1 / d}$. If $t_{\lambda}(\mathcal{X}_{n})<t$, then in particular there exist three points $X_{1}, X_{2}$ and $X_{3}$ such that $X_{2}, X_{3} \in \mathcal{B}\left(X_{1}, 2 t\right)$. We then have by Lemma \ref{lem2.2} that $Y_{2}, Y_{3} \in \mathcal{B}\left(Y_{1}, \frac{2 \tau(M)}{\tau(M)-\gamma} t\right)$. We obtain as in \eqref{eq4.7} that
\begin{equation}
\mathbb{P}\left(t_{\lambda}(\mathcal{X}_{n})<t\right) \leq C_{1} \frac{\left(n t^{d}\right)^{2}}{\left(\tau_{\min }-\gamma\right)^{2}}.
\end{equation}

Fix $t \in\left[t_{0}, t_{2}\right]$ (where $t_{0}$ and $t_{2}$ are defined in the proof of Theorem \ref{thm4.6}) and let $0<\gamma<t$. We have, by Proposition \ref{prop5.2},
\begin{equation*}
\begin{aligned}
&h(t, \mathcal{X}_{n}) =d_{H}(\operatorname{Conv}(t, \mathcal{X}_{n})  |  \mathcal{X}_{n}) \\
& \geq d_{H}(\operatorname{Conv}(t-\gamma, \mathcal{Y}_{n})  |  \mathcal{Y}_{n})-d_{H}(\mathcal{X}_{n}  |  \mathcal{Y}_{n})-d_{H}(\operatorname{Conv}(t-\gamma, \mathcal{Y}_{n})  |  \operatorname{Conv}(t, \mathcal{X}_{n})) \\
& \geq h(t-\gamma, \mathcal{Y}_{n})-2 \gamma.
\end{aligned}
\end{equation*}

Therefore, if $\frac{\lambda t+2 \gamma}{t-\gamma} \leq \lambda^{\prime}<1$ and $h(t-\gamma, \mathcal{Y}_{n}) \geq \lambda^{\prime}(t-\gamma)$, then $h(t, \mathcal{X}_{n}) \geq \lambda t$. Assume that $\gamma \leq \eta(\ln n / n)^{2 / d}$ for some $\eta>0$ and fix $\lambda^{\prime}=(1+\lambda) / 2$. Then, for $t \geq \tilde{t}_{0}:=6 \gamma /(1-\lambda)$, the condition $\frac{\lambda t+\gamma}{t-\gamma} \leq \lambda^{\prime}$ is satisfied. Furthermore, according to the proof of Theorem \ref{thm4.6}, for such a $t$, the condition $h(t-\gamma, \mathcal{Y}_{n}) \geq \lambda^{\prime}(t-\gamma)$ is satisfied with probability at least $1-c n^{-2}$. Using the same argument than in the proof of Theorem \ref{thm4.6}, we then obtain that

\begin{equation}
\begin{aligned}
\mathbb{P}\left(t_{\lambda}(\mathcal{X}_{n}) \leq \frac{1}{2-\sqrt{\lambda}}\left(\frac{\beta \ln n}{\alpha_{d} f_{\min }(1+\delta) n}\right)^{1 / d}\right) & \leq c_{1}(\ln n) n^{-2}+c_{2}\left(n \tilde{t}_{0}^{d}\right)^{2} \\
& \leq 2 c_{2} \frac{(\ln n)^{4}}{n^{2}}.
\end{aligned}
\end{equation}

We may conclude as in the previous proof that we have $t_{\lambda}(\mathcal{X}_{n}) \geq t^{*}\left(\mathcal{Y}_{n}\right)+\gamma$ with probability equal to $1-c(\ln n)^{\tilde{a}} n^{-b}$, where $\tilde{a}=4 \vee(d-1)$ if $b=2$ and $d-1$ otherwise.

 Let us now provide an upper bound on $t_{\lambda}(\mathcal{X}_{n})$. Consider the interval \[I=[(1-\lambda / 8) t_{\lambda / 2}(\mathcal{Y}_{n}),(1-\lambda / 16) t_{\lambda / 2}(\mathcal{Y}_{n})).\] By Theorem \ref{thm4.6}, Proposition \ref{prop3.4}, Lemma \ref{lem4.8} and Proposition \ref{prop3.9}, $t_{\lambda / 2}\left(\mathcal{Y}_{n}\right)$ is at least of order $n^{-1 / d}$ and at most of order $(\ln n / n)^{1 / d}$ with probability $1-n^{-2}$. By Lemma \ref{lem4.7}, this implies that $\operatorname{Rad}\left(\mathcal{Y}_{n}\right)$ intersects $I$ with the same probability. Let $t^{\prime} \in \operatorname{Rad}\left(\mathcal{Y}_{n}\right) \cap I$. This scale corresponds to some simplex $\sigma^{\prime}=\left\{y_{1}, \ldots, y_{K}\right\}$, and we let $\sigma=\left\{x_{1}, \ldots, x_{K}\right\} \subseteq \mathcal{X}_{n}$ where $y_{i}=\pi_{M}\left(x_{i}\right)$. We have $t:=r(\sigma) \leq \gamma+t^{\prime}$ according to \cite[Lemma 16]{geo:attali2013vietoris}. Furthermore, if $z$ is the center of the smallest enclosing ball of $\sigma$, we have using Lemma \ref{lem2.2}, $\left|y_{i}-\pi_{M}(z)\right| \leq \frac{\tau(M)}{\tau(M)-\gamma}\left|x_{i}-z\right| \leq \frac{t \tau(M)}{\tau(M)-\gamma}$, indicating that $t^{\prime} \leq \frac{t \tau(M)}{\tau(M)-\gamma}$. Recalling that $\gamma$ is of order $(\ln n / n)^{2 / d} \ll n^{-1 / d} \ll t_{\lambda / 2}\left(\mathcal{Y}_{n}\right)$, this means we have found a scale $t \in \operatorname{Rad}(\mathcal{X}_{n})$ satisfying

\begin{equation}\label{eq5.6}
(1-\lambda / 4) t_{\lambda / 2}(\mathcal{Y}_{n}) \leq\left(1-\frac{\gamma}{\tau(M)}\right) t^{\prime} \leq t \leq t^{\prime}+\gamma \leq(1-\lambda / 8) t_{\lambda / 2}(\mathcal{Y}_{n}).
\end{equation}

Using Proposition \ref{prop4.3} and \eqref{eq5.6}, we obtain
\begin{equation*}
\begin{aligned}
h(t, \mathcal{X}_{n}) & \leq d_{H}(\operatorname{Conv}(t, \mathcal{X}_{n})  |  \operatorname{Conv}(t_{\lambda / 2}(\mathcal{Y}_{n}), \mathcal{Y}_{n}))+h(t_{\lambda / 2}(\mathcal{Y}_{n}), \mathcal{Y}_{n})+d_{H}(\mathcal{X}_{n}, \mathcal{Y}_{n}) \\
& \leq (t_{\lambda / 2}(\mathcal{Y}_{n})-t)+\frac{\lambda}{2} t_{\lambda / 2}(\mathcal{Y}_{n})+\gamma \\
& \leq \frac{\lambda}{4} t_{\lambda / 2}(\mathcal{Y}_{n})+\frac{\lambda}{2} t_{\lambda / 2}(\mathcal{Y}_{n})+\gamma \leq \frac{3 \lambda}{4} t_{\lambda / 2}(\mathcal{Y}_{n})+\gamma \leq \lambda(1-\lambda / 4) t_{\lambda / 2}(\mathcal{Y}_{n}) \\
& \leq \lambda t
\end{aligned}
\end{equation*}

where at the second to last line we used that $\gamma \leq \lambda \frac{(1-\lambda)}{4} t_{\lambda / 2}(\mathcal{Y}_{n})$ (as $t_{\lambda / 2}(\mathcal{Y}_{n})$ is of order at least $\left.n^{-1 / d}\right)$. This implies that $t_{\lambda}(\mathcal{X}_{n}) \leq t \leq t_{\lambda / 2}(\mathcal{Y}_{n})$. Using the upper bound on $t_{\lambda / 2}(\mathcal{Y}_{n})$ given in Theorem \ref{thm4.6}, we have that, with probability $1-c(\ln n)^{\tilde{a}} n^{-b}$,

\begin{equation}\label{eq5.7}
t^{*}(\mathcal{Y}_{n})+\gamma \leq t_{\lambda}(\mathcal{X}_{n}) \leq \frac{2 t^{*}(\mathcal{Y}_{n})}{\lambda}\left(1+C\left(\frac{(\ln n)^{2}}{n}\right)^{1 / d}\right)
\end{equation}

that is an analog of Theorem \ref{thm4.6} also holds in a setting where tubular noise of size $(\ln n / n)^{2 / d}$ is present.

\section{Adaptive estimation with the selected scale}

In this section, we show that the estimator $\hat{M}=\operatorname{Conv}(t_{\lambda}(\mathcal{X}_{n}), \mathcal{X}_{n})$ is minimax adaptive on the scale of models $\mathcal{Q}_{\tau_{\min }, f_{\min }, f_{\max }}^{d}$. For the sake of exposition, we focus on the noiseless case $\gamma=0$. We first have to be careful when defining the scale of models. Indeed, by \eqref{eq2.1}, we have for $\mu \in \mathcal{Q}_{\tau_{\min }, f_{\min }, f_{\max }}^{d}$ supported on $M$

\begin{equation*}
1=\mu(M) \geq f_{\min } \omega_{d} \tau_{\min }^{d},
\end{equation*}

so that the model $\mathcal{Q}_{\tau_{\min }, f_{\min }, f_{\max }}^{d}$ is empty if $f_{\min } \omega_{d} \tau_{\min }^{d}>1$. Also, if we have $f_{\min } \omega_{d} \tau_{\min }^{d}=1$, then $\mu$ is the uniform distribution on a sphere. In this case $d+1$ observations characterize $M$, and the $\operatorname{minimax}$ rate on $\mathcal{Q}_{\tau_{\min }, f_{\min }, f_{\max }}^{d}$ is zero for $n \geq d+1$. To discard such degenerate cases, we will assume that there exists a constant $\kappa<1$ so that $\omega_{d} f_{\min } \tau_{\min }^{d}<\kappa$. We have already mentioned in the introduction that Kim and Zhou \cite{stat:kim2015tight} showed that the minimax risk $\mathcal{R}_{n}(\mathcal{Q}_{\tau_{\min }, f_{\min }, f_{\text {max }}}^{d})$ is of order $(\ln n / n)^{2 / d}$. They were however not concerned with precise constants. We indicate in Appendix \ref{appB} how to modify their proof to obtain a more precise result.

\begin{prop}\label{prop6.1} There exists a constant $C$ depending only on $\kappa$ such that

\begin{equation*}
\lim _{n} \inf \frac{\mathcal{R}_{n}(\mathcal{Q}_{\tau_{\min }, f_{\min }, f_{\max }}^{d})}{(\ln n / n)^{2 / d}} \geq \frac{C}{\left(\alpha_{d} f_{\min }\right)^{2 / d} \tau_{\min }}.
\end{equation*}
\end{prop}

Our adaptivity result then reads as follows.

\begin{theorem}\label{thm6.2}
Let $d \geq 2$. Let $\mu \in \mathcal{Q}_{\tau_{\min }, f_{\min }, f_{\max }}^{d}$ and let $0<\lambda<(1+2 / d)^{-1 / d}$. Then, for $n$ large enough, we have
\begin{equation}\label{eq6.2}
\begin{aligned}
\mathbb{E}\left[d_{H}(\operatorname{Conv}(t_{\lambda}(\mathcal{X}_{n}), \mathcal{X}_{n}), M)\right] & \leq \frac{c_{0}}{\lambda^{2}\left(\alpha_{d} f_{\min }\right)^{2 / d} \tau_{\min }}\left(\frac{\ln n}{n}\right)^{2 / d} \\
&\leq \frac{c_{1}}{\lambda^{2}} \mathcal{R}_{n}(\mathcal{Q}_{\tau_{\min }}^{d}, f_{\min }, f_{\max }),
\end{aligned}
\end{equation}

where $c_{0}$ is a numerical constant and $c_{1}$ only depends on $\kappa$.
\end{theorem}

\begin{proof}
 Choose $b \in(0,2]$ such that $\lambda<(1+b)^{-1 / d}<(1+2 / d)^{-1 / d}$. Assume that the event described in \eqref{eq4.4} is satisfied (that is with probability larger than $\left.1-c(\ln n)^{d-1} n^{-b}\right)$. Then, we have by Lemma \ref{lem3.3}
\begin{equation*}
d_{H}(\operatorname{Conv}(t_{\lambda}(\mathcal{X}_{n}), \mathcal{X}_{n}), M) \leq \frac{t_{\lambda}(\mathcal{X}_{n})^{2}}{\tau_{\min }} \leq \frac{t^{*}(\mathcal{X}_{n})^{2}}{\lambda^{2} \tau_{\min }}\left(1+C\left(\frac{(\ln n)^{2}}{n}\right)^{1 / d}\right)^{2}
\end{equation*}

We also assume that $\varepsilon(\mathcal{X}_{n}) \leq\left(\frac{4 \ln n}{\alpha_{d} f_{\min n} n}\right)^{1 / d}$, an event that happens with probability $1-(\ln n)^{d-1} n^{-3}$ by Proposition \ref{prop3.9}. Then, for $n$ large enough, we have $t^{*}(\mathcal{X}_{n}) \leq 2 \varepsilon(\mathcal{X}_{n})$ by Proposition \ref{prop3.4}. In particular, we obtain that, for $n$ large enough

\begin{equation*}
d_{H}(\operatorname{Conv}(t_{\lambda}(\mathcal{X}_{n}), \mathcal{X}_{n}), M) \leq \frac{c_{0}}{\lambda^{2}\left(\alpha_{d} f_{\min }\right)^{2 / d} \tau_{\min }}\left(\frac{\ln n}{n}\right)^{2 / d}
\end{equation*}

for some absolute constant $c_{0}$. The probability that this inequality is not satisfied is of order $(\ln n)^{d-1} n^{-b} \ll(\ln n / n)^{2 / d}$, and if this is the case we bound the risk by $\operatorname{diam}(M)$ (that is bounded by a constant depending on $\tau_{\min }, f_{\min }$ and $d$ \cite[Lemma III.24]{stat:aamari2017vitesses}). We therefore obtain the first inequality of \eqref{eq6.2}, while the second one follows directly from Proposition \ref{prop6.1}. 
\end{proof}

\begin{rem}
In the case $d=1$, the minimax risk is of order $(\ln n / n)^{2} /\left(\left(\alpha_{d} f_{\min }\right)^{2} \tau_{\min }\right)$, whereas, with $b=2$, the probability with which \eqref{eq5.7} holds is of order $(\ln n / n)^{2}$. As such, one can show that the risk of $\operatorname{Conv}(t_{\lambda}(\mathcal{X}_{n}), \mathcal{X}_{n})$ is of order $(\ln n / n)^{2}$ for $d=1$, but with a leading constant that will depend on the constants appearing in Theorem \eqref{thm4.6}. This leading constant is therefore not anymore of order $1 /\left(\left(\alpha_{d} f_{\min }\right)^{2} \tau_{\min }\right)$, and we do not have a clean inequality of the form \eqref{eq6.2}. Still, $\operatorname{Conv}(t_{\lambda}(\mathcal{X}_{n}), \mathcal{X}_{n})$ is a data-driven minimax estimator even in this case.
\end{rem}

With a choice of $\lambda$ smaller than $1 / \sqrt{2}$ (say $\lambda=1 / 2$ ), the condition $\lambda<(1+2 / d)^{-1 / d}$ is satisfied for every $d \geq 2$. With such a choice, we obtain a completely data-driven estimator that attains asymptotically the minimax rate $\mathcal{R}_{n}(\mathcal{Q}_{\tau_{\min }, f_{\min }, f_{\max }}^{d})$ up to an absolute constant, for every admissible choice of $\tau_{\min }, f_{\min }, f_{\max }$ and $d \geq 2$. The slope $\lambda$ in our selection procedure is akin to a regularization parameter that appears in most selection methods (such as in the LASSO \cite{stat:tibshirani1996regression}, or the PCO and Goldenshluger-Lepski methods already mentioned). If every choice of parameter $\lambda<1 / \sqrt{2}$ is admissible from a theoretical point of view, the practical choice of the parameter $\lambda$ is more delicate. We develop in Section \ref{sec:num} a heuristic, similar to the slope heuristics \cite{arlot2019minimal}, to choose the parameter $\lambda$.

\begin{rem} We insist that our result is of an asymptotic nature, as the "large enough" in the above theorem depends on the probability measure $\mu$. A similar behavior occurs with the PCO method mentioned in the introduction \cite{stat:lacour2016estimator} (or with the Goldenshluger-Lepski method \cite[Proposition 1]{lacour2016minimal}). Indeed, the remainder term $C(n,|\mathcal{H}|)$ appearing in \eqref{eq1.5} depends on $\mu$ through the $\infty$-norm of its density function, whereas the minimax risk does not depend on this $\infty$-norm (see \cite[Theorem 2.8]{tsybakov2008introduction}). As such, the remainder term $C(n,|\mathcal{H}|)$ becomes negligible in front of the minimax risk only for $n$ large enough with respect to $\mu$ (and not only with respect to the parameters defining the statistical model), as this is the case in Theorem \ref{thm6.2}.
\end{rem}

The parameter $t_{\lambda}(\mathcal{X}_{n})$ actually gives us the approximation rate $\varepsilon(\mathcal{X}_{n})$ up to a multiplicative constant (roughly equal to $\lambda^{-1}$). As such, it can be also used to design other data-driven estimators. As an example, we consider the estimation of the tangent spaces of a manifold. Let $x \in M$ and $A \subseteq M$ be a finite set. We denote by $T_{x}(A, t)$ the $d$-dimensional vector space $U$ that minimizes $d_{H}(A \cap \mathcal{B}(x, t)  |  x+U)$. This estimator was originally studied in \cite{stat:belkin2009constructing}. Recall that the angle between subspaces is denoted by $\angle$.

\begin{cor}\label{cor6.1} Let $\mu \in \mathcal{Q}_{\tau_{\min }, f_{\min }, f_{\text {max }}}^{d}$ with support $M$ and let $0<\lambda<(1+1 / d)^{-1 / d}$. Then, for $n$ large enough (with respect to $\mu)$, we have

\begin{equation*}
\mathbb{E} \angle\left(T_{x} M, T_{p}\left(\mathcal{X}_{n}, 11 t_{\lambda}(\mathcal{X}_{n})\right)\right) \leq c\left(\frac{\ln n}{n}\right)^{1 / d}
\end{equation*}

for some constant $c$ depending on $\lambda, d, \tau_{\min }$ and $f_{\text {min }}$
\end{cor}

This rate is the minimax rate (up to logarithmic factors) according to \cite[Theorem 3]{stat:aamari2019nonasymptotic}.

\begin{proof}
 Theorem $3.2$ in \cite{stat:belkin2009constructing} states that for $A \subseteq M$, if $t<\tau(M) / 2$ and $t \geq 10 \varepsilon(A)$, then

\begin{equation*}
\angle\left(T_{x}(A, t), T_{x} M\right) \leq 6 \frac{t}{\tau(M)}.
\end{equation*}

As in the previous proof, we may choose $b \in(0,2]$ such that $\lambda<(1+b)^{-1 / d}<(1+1 / d)^{-1 / d}$, and assume that the event described in Theorem \ref{thm4.6} is satisfied. We also assume that $\varepsilon(\mathcal{X}_{n}) \leq$ $\left(\frac{4 \ln n}{\alpha_{d} f_{\min } n}\right)^{1 / d}$. Then, the quantity $t=11 t_{\lambda}(\mathcal{X}_{n})$ is larger than $10 \varepsilon(\mathcal{X}_{n})$ for $n$ large enough, and furthermore satisfies $t \leq c_{0}\left(\frac{\ln n}{\alpha_{d} f_{\min n} n}\right)^{1 / d}$ for some absolute constant $c_{0}$ if $n$ is large enough. We then have
\begin{equation*}
\angle\left(T_{x}\left(\mathcal{X}_{n}, t\right), T_{x} M\right) \leq \frac{c_{1}}{\left(\alpha_{d} f_{\min }\right)^{1 / d} \tau_{\min }}\left(\frac{\ln n}{n}\right)^{1 / d}
\end{equation*}

for some absolute constant $c_{1}$ large enough. If one of the two conditions does not hold (this happens with a probability smaller than $\left.(\ln n)^{a} n^{-b}=o\left(n^{-1 / d}\right)\right)$, we bound the angle by 2, concluding the proof.
\end{proof}

\begin{rem} Authors in \cite{stat:berenfeld2020estimating} also propose to use the convexity defect function of a set $A \subseteq M$ to estimate the reach of $M$, while their method requires only the knowledge of $\varepsilon(A)$. As such, we may use their technique by using the scale $t_{\lambda}(\mathcal{X}_{n})$ instead of $\varepsilon(\mathcal{X}_{n})$. This leads to a reach estimator that attains a risk of order $(\ln n / n)^{1 /(3 d)}$. As the minimax risk is of order $n^{-1 / d}$ up to logarithmic factors for this problem (at least on a statistical model made of $\mathcal{C}^{3}$ manifolds), this is far from being minimax. Still, this yields a consistent fully data-driven reach estimator. We refer to \cite{stat:berenfeld2020estimating} for details on the construction.
\end{rem}

\section{Numerical considerations\protect\footnote{Code is made available at  \href{https://github.com/vincentdivol/local-convex-hull}{github.com/vincentdivol/local-convex-hull}.}}\label{sec:num}

There are two distinct procedures to investigate: first, the computation of the $t$-convex hull $\operatorname{Conv}(t, \mathcal{X}_{n})$, and second, the computation of the scale $t_{\lambda}(\mathcal{X}_{n})$. To compute the $t$-convex hull $\operatorname{Conv}(t, \mathcal{X}_{n})$, it suffices to compute the Čech complex $\operatorname{Cech}(t, \mathcal{X}_{n}):=\left\{\sigma \subseteq \mathcal{X}_{n}: \ r(\sigma) \leq t\right\}.$ 
For $x \in \mathbb{R}^{D}$, let $N(x)$ be the number of points of $\mathcal{X}_{n}$ at distance less than $2 t$ of $x$. Assume that one has access to the set $E_{t}(\mathcal{X}_{n})$ of edges of $\mathcal{X}_{n}$ of length smaller than $2 t$. Then, authors in \cite{le2015construction} propose an algorithm of complexity $C_{D} \sum_{i=1}^{n} N\left(X_{i}\right)^{D}$ to compute $\operatorname{Cech}(t, \mathcal{X}_{n})$. 
When $t$ is of order $(\ln n / n)^{1 / d}$, $N(X_{i})$ is on average of order $\ln n$ and we obtain an average complexity of order $C_{D} n(\ln n)^{D}$. In high dimension, the complexity can be reduced if one has access to the dimension $d$ by computing $\operatorname{Conv}_{d}(t, \mathcal{X}_{n})$ instead (see Remark \ref{rem3.8}). Indeed, according to \cite{le2015construction}, the set of simplices of $\operatorname{Cech}(t, \mathcal{X}_{n})$ of dimension smaller than $d$ can be computed with average time complexity of order $C_{d} D n(\ln n)^{d}$. We also have to consider the computation of the edges $E_{t}(\mathcal{X}_{n})$. A naive algorithm to compute this set leads to a complexity of order $D n^{2}$, but in practice this can be considerably sped up by using e.g. a RP tree \cite{dasgupta2008random}.

 We now adress the selection procedure described in Section \ref{sec:selection}. To choose the scale $t_{\lambda}(\mathcal{X}_{n})$, we have to compute the convexity defect function of $\mathcal{X}_{n}$. To do so, we need for each simplex $\sigma \subseteq \mathcal{X}_{n}$ to (i) compute its radius $r(\sigma)$ and (ii) compute $d_{H}(\operatorname{Conv}(\sigma)  |  \mathcal{X}_{n})$. We will simplify this problem by considering only simplexes $\sigma$ of dimension 1 (i.e. edges). Let $\operatorname{Graph}(t, \mathcal{X}_{n})$ be the union of the edges of $\mathcal{X}_{n}$ of length $2 t$. We may define a graph convexity defect function $\tilde h(t, \mathcal{X}_{n})=d_{H}(\operatorname{Graph}(t, \mathcal{X}_{n}), \mathcal{X}_{n})$, as well as a graph scale parameter

\begin{equation*}
\tilde{t}_{\lambda}(\mathcal{X}_{n}):=\inf \left\{t \in \widetilde{\operatorname{Rad}}(\mathcal{X}_{n}): \tilde{h}(t, \mathcal{X}_{n}) \leq \lambda t\right\},
\end{equation*}

where $\widetilde{\operatorname{Rad}}(\mathcal{X}_{n}):=\left\{\left|X_{i}-X_{j}\right| / 2: 1 \leq i, j \leq n\right\}.$ A careful read of the proof of Theorem \ref{thm4.6} shows that only edges are considered to obtain the different inequalities of the theorem. In particular, this theorem also holds with $\tilde{t}_{\lambda}(\mathcal{X}_{n})$ instead of $t_{\lambda}(\mathcal{X}_{n})$. When $e$ is an edge of $\mathcal{X}_{n}$, the distance $d_{H}(\operatorname{Conv}(e)  |  \mathcal{X}_{n})$ can be computed in $\mathcal{O}(n(D+\ln n))$ operations \cite{alt2003computing}. By looping over the $\mathcal{O}\left(n^{2}\right)$ edges of the dataset, we may compute $\tilde{h}(\cdot, \mathcal{X}_{n})$ with a time complexity of $\mathcal{O}\left(n^{3}(D+\ln n)\right)$.

The choice of the slope value $\lambda$ has an impact on the selection procedure. Ideally, we would like to choose $\lambda$ so that it is just below 
\[\lambda_{\max }(\mathcal{X}_{n}):=\max \left\{\lambda: t_{\lambda}(\mathcal{X}_{n})>t^{*}(\mathcal{X}_{n})\right\}.\]

 Let $t_{\max }(\mathcal{X}_{n})=t_{\lambda_{\max }(\mathcal{X}_{n})}(\mathcal{X}_{n})$. According to Proposition \ref{prop4.3}, the function $h(\cdot, \mathcal{X}_{n})$ is almost constant after $t^{*}(\mathcal{X}_{n})$, and therefore also almost constant after $t_{\max }(\mathcal{X}_{n})$. This implies that $t_{\lambda}(\mathcal{X}_{n})$ should increase proportionally with $1 / \lambda$ for $\lambda<\lambda_{\max }(\mathcal{X}_{n})$ (at least approximately). On the opposite, for $\lambda>\lambda_{\max }(\mathcal{X}_{n})$, we expect $t_{\lambda}(\mathcal{X}_{n})$ to go to 0 quickly. By plotting the graph of the function $g_{\mathcal{X}_{n}}: \lambda \mapsto 1 / t_{\lambda}(\mathcal{X}_{n})$, those two behaviors should be observed (first linear and then diverging), so that a "jump" should occur around the value $\lambda_{\max }(\mathcal{X}_{n})$. We indeed observe such a phenomenon, see Figure \ref{fig8}. In practice, we use a grid $0=\lambda_{1} \leq \cdots \leq \lambda_{L}=1$ and the jump is defined by the smallest $l$ such that the condition $g_{\mathcal{X}_{n}}\left(\lambda_{l+1}\right)-g_{\mathcal{X}_{n}}\left(\lambda_{l}\right)>0.5 g_{\mathcal{X}_{n}}(0)$ is satisfied. We then select $\lambda_{\text {choice}}(\mathcal{X}_{n})=0.8 \lambda_{\text {jump }}(\mathcal{X}_{n})$ and let $t_{\text {sel}}(\mathcal{X}_{n}):=t_{\lambda_{\text {choice}}(\mathcal{X}_{n})}(\mathcal{X}_{n})$ (other constants than $0.5$ and $0.8$ would work as well).

\begin{figure}
\includegraphics[width=0.37\textwidth]{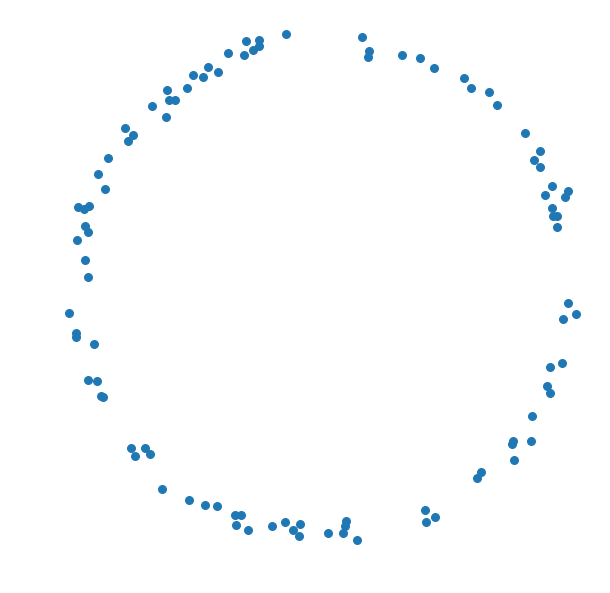}
\hfill
\includegraphics[width=0.5\textwidth]{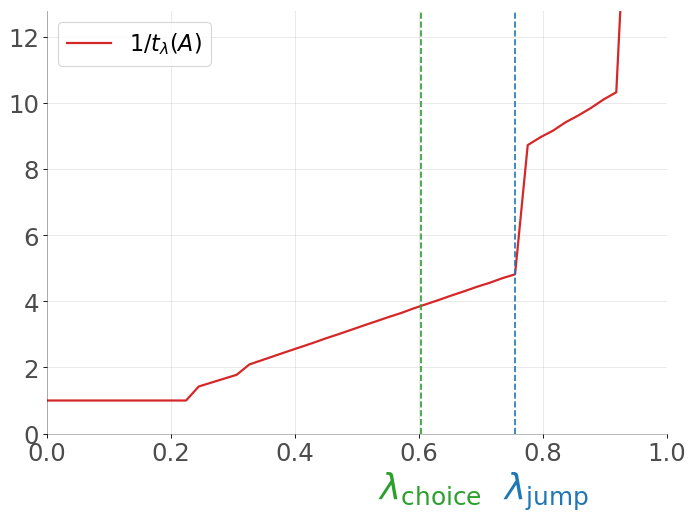}
\includegraphics[width=0.4\textwidth]{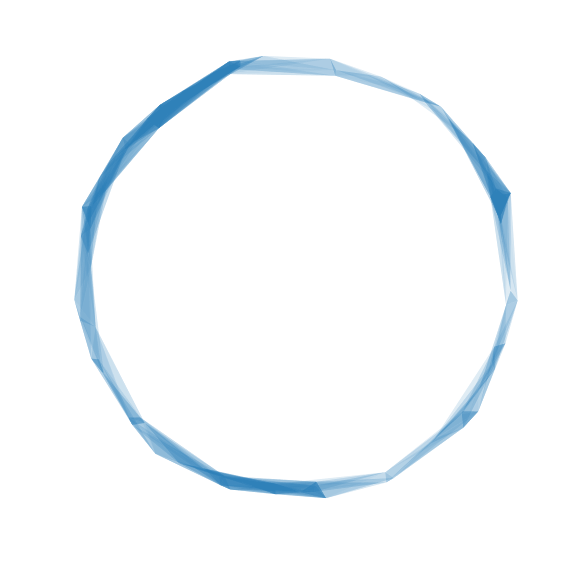}
\hfill
\includegraphics[width=0.5\textwidth]{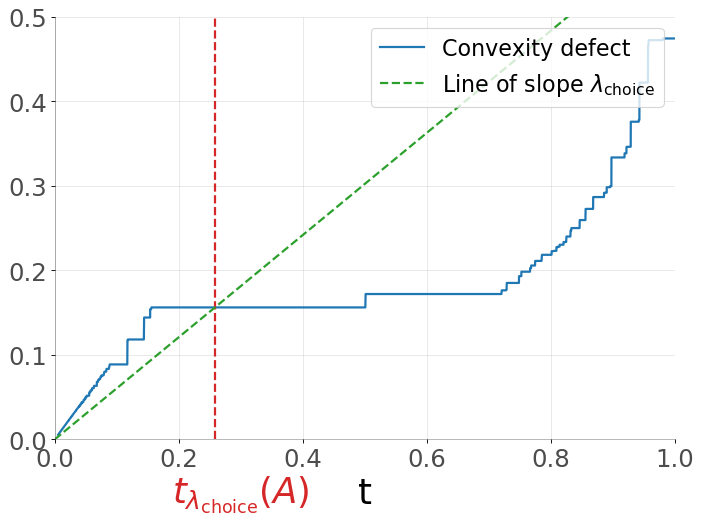}

\caption{Top left. Sample $\mathcal{X}_{n}$. Top right. The value of $\lambda_{\text {choice}}(\mathcal{X}_{n})$ is equal to $0.8 \lambda_{\text {jump }}(\mathcal{X}_{n})$. Bottom left. The set $\operatorname{Conv}(t_{\text {sel}}(\mathcal{X}_{n}), \mathcal{X}_{n})$. Bottom right. The graph convexity defect function $h(\cdot, \mathcal{X}_{n})$.}\label{fig8}
\end{figure}

\begin{rem}
 This method to select the slope $\lambda$ is similar to the slope heuristics in model selection. Consider for instance the fixed-design regression setting where $Y=F+\varepsilon \in \mathbb{R}^{n}$ is observed with a Gaussian noise $\varepsilon \sim \mathcal{N}\left(0, \sigma^{2} \mathrm{Id}\right)$. The goal is to reconstruct the signal $F$ for the $\ell_{2}$-loss, by selecting an estimator among the estimators $\hat{F}_{m}=\pi_{S_{m}}(Y)$, where $\left\{S_{m}: m\right\}$ is a collection of linear subspaces, each $S_{m}$ being of dimension $D_{m}$. A classical method to select the estimator $F_{m}$ is to choose

\begin{equation*}
\hat{m}(C) \in \underset{m}{\arg \min }\left\{\left|\hat{F}_{m}-Y\right|^{2}+C D_{m}\right\}
\end{equation*}

where $C$ is a constant to fix. In theory, any value of $C$ smaller than $\sigma^{2}$ will lead to overfitting, whereas values of $C$ larger than $\sigma^{2}$ are admissible. We then say that $C=\sigma^{2}$ is the minimal penalty. The exact value of the minimal penalty $C=\sigma^{2}$ is of an asymptotic nature. However, we still see a minimal penalty phenomenon occurring in practice: for $C$ too small, the selected dimension $D_{\hat{m}(C)}$ will be very large, whereas at some value $\hat{C}_{\text {jump }}$ it will suddendly decrease and gets smaller. This jump is detected and is used to select the value of $C$. We refer to \cite{arlot2019minimal} for details. A similar phenomenon occurs in our setting: the slope $\lambda$ plays the role of the parameter $C$ (or rather $1 / C$), and we have a maximal penalty phenomenon: every value of $\lambda$ smaller than 1 is theoretically admissible. The quantity $1 / t$ is the analog of the dimension $D_{m}$, as it is a measure of the complexity of the estimator $\operatorname{Conv}(t, \mathcal{X}_{n})$: choosing $t=+\infty$ amounts to assuming that $M$ is a convex set, whereas choosing very small values of $t$ amounts to assuming that $M$ has a small reach. In practice, we observe a jump in the function $g_{\mathcal{X}_{n}}: \lambda \mapsto 1 / t_{\lambda}(\mathcal{X}_{n})$, and we use this phenomenon to choose the parameter $\lambda$.
\end{rem}

\begin{figure}

\centering
     \begin{subfigure}[b]{0.48\textwidth}
         \centering
\includegraphics[width=\textwidth]{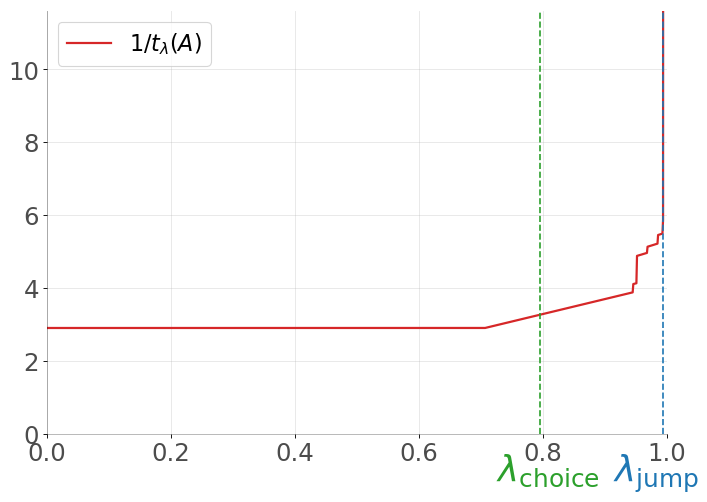}
         \caption{Choice of $\lambda$ - Torus}
     \end{subfigure}
     \hfill
     \begin{subfigure}[b]{0.48\textwidth}
         \centering
\includegraphics[width=\textwidth]{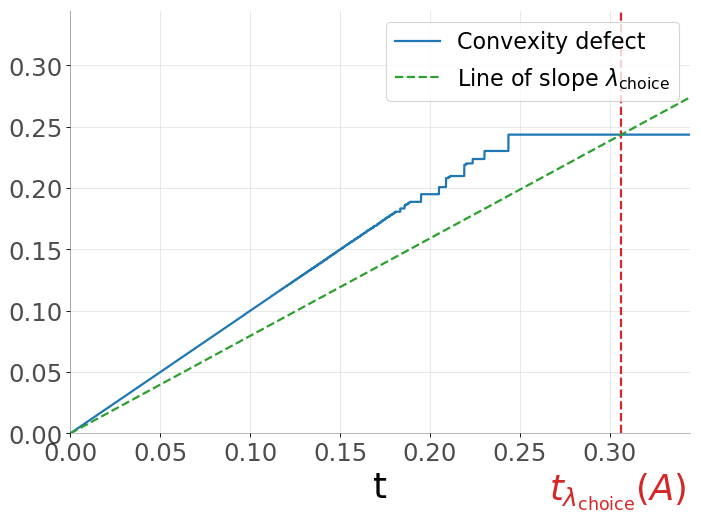}
         \caption{Choice of $t$ - Torus}
     \end{subfigure}
     \hfill
     \begin{subfigure}[b]{0.48\textwidth}
         \centering
\includegraphics[width=\textwidth]{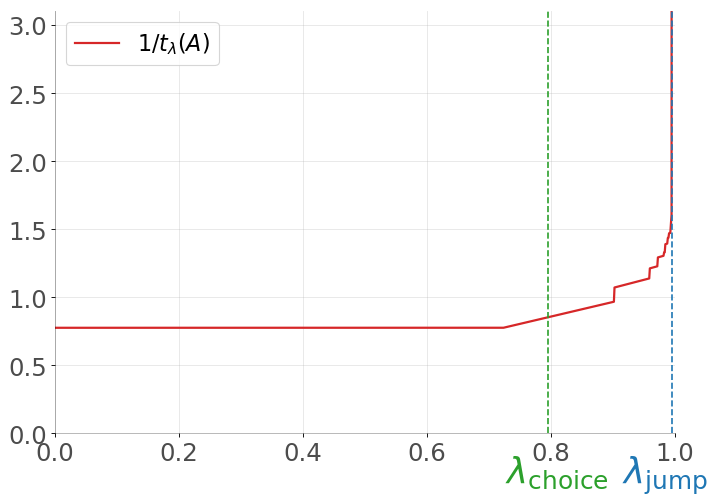}
         \caption{Choice of $\lambda$ - Swissroll}
     \end{subfigure}
     \hfill
          \begin{subfigure}[b]{0.48\textwidth}
         \centering
\includegraphics[width=\textwidth]{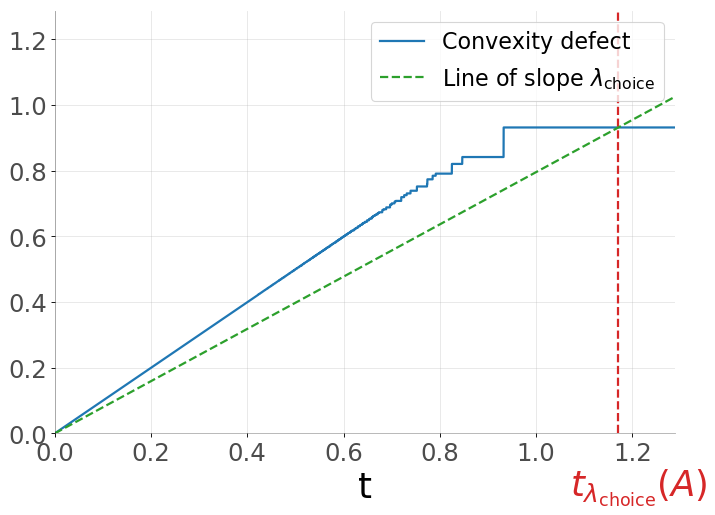}
         \caption{Choice of $t$ - Swissroll}
     \end{subfigure}

\caption{For a set $\mathcal{X}_{n}$ made of $10^{4}$ points sampled on the torus (resp. on the swiss roll), we compute $g_{\mathcal{X}_{n}}$ and $h(\cdot, \mathcal{X}_{n})$ up to the value $t=\ell_{K_{\max }}(\mathcal{X}_{n})$. The selected values of $\lambda$ are respectively $0.796$ and $0.792$, while the selected values of $t_{\text {sel}}(\mathcal{X}_{n})$ are $0.309$ and 1.126. In both cases, we also estimate the approximation rate $\varepsilon(\mathcal{X}_{n})$, respectively equal to $0.254$ and $0.891$. Both times, we indeed have $t_{\text {sel}}(\mathcal{X}_{n}) \geq \varepsilon(\mathcal{X}_{n})$, and furthermore, the Čech complex of parameter $2 t_{\text {sel}}(\mathcal{X}_{n})$ has the same homology as the torus (resp. the swiss roll).}
\label{fig10}
\end{figure}

In Figure \ref{fig8}, we display the graph convexity defect function $\tilde{h}(\cdot, \mathcal{X}_{n})$ for a set $\mathcal{X}_{n}$ made of $n=100$ uniformly sampled point on the unit circle $M$, with a tubular uniform noise of size $\gamma=0.1$. Both the "jump" phenomenon in the function $g_{\mathcal{X}_{n}}$ and the expected behavior of the function $h(\cdot, \mathcal{X}_{n})$ occur. We evaluate $\varepsilon(\mathcal{X}_{n})=0.16$, while $\lambda_{\text {choice}}(\mathcal{X}_{n})=0.60$ and $t_{\text {sel}}(\mathcal{X}_{n})=0.26$. According to \cite[Proposition 3.1]{geo:niyogi2008finding}, the Cech complex $\operatorname{Cech } ( \mathcal { X } _ { n }, 2 t ) \text { on } A \text { of radius } 2 t \text { has } $the same homology as $M$ as long as $t \geq \varepsilon(\mathcal{X}_{n})$. As a safety check, we compute the homology of $\operatorname{Cech}\left(\mathcal{X}_{n}, 2 t_{\mathrm{sel}}(\mathcal{X}_{n})\right)$, which is indeed equal to the homology of the circle. 

\begin{figure}
\centering
\includegraphics[width=0.4\textwidth]{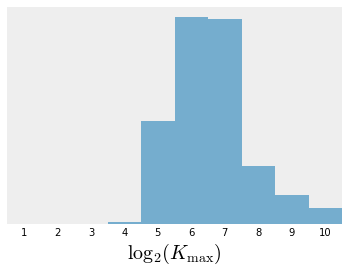}
\hspace{1cm}
\includegraphics[width=0.4\textwidth]{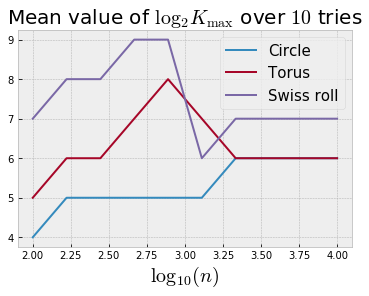}
\caption{Left. Distribution of $\log_{2} K_{\max }$ over the different point clouds (circle, torus and swiss roll of different sizes on 10 tries each). Right. For each class and each number of points $n$, we display the mean value of $\log _{2} K_{\max }$ over the 10 tries: in each class, it stays bounded as $n$ grows. Large values of $K_{\max }$ for the swiss roll dataset correspond to numbers of samples $n$ for which $\varepsilon(\mathcal{X}_{n})$ is too large $(n \leq 1000)$: the subquadratic behavior then does not occur and therefore the whole convexity defect function is computed.}\label{fig9}
\end{figure}

Actually, it is not necessary to compute the whole convexity defect function to compute $\tilde{t}_{\lambda}(\mathcal{X}_{n})$, as one can stop at the first value for which $\tilde{h}(t, \mathcal{X}_{n})<\lambda t$. This can be used to speed up the computation $t_{\text {sel}}(\mathcal{X}_{n})$. Given an integer $K$, we let $\ell_{K}(\mathcal{X}_{n})$ be half the maximum distance between a point of $\mathcal{X}_{n}$ and its $K$ th nearest neighbor in $\mathcal{X}_{n}$. We compute for each point $X_{i}$ in $\mathcal{X}_{n}$ its $K$ nearest neighbors $\mathcal{X}_{K}^{i}$ (using for instance a RP tree \cite{dasgupta2008random}). 
Then, for each point $X_{j}$ in $\mathcal{X}_{K}^{i}$, if $e=\left(X_{i}, X_{j}\right)$, we have $d_{H}(\operatorname{Conv}(e)  |  \mathcal{X}_{n})=d_{H}(\operatorname{Conv}(e)  |  \mathcal{X}_{K}^{i}).$  
The latter distance can be computed in $\mathcal{O}(K(D+\ln K))$ operations. There are at most $n K$ such edges, so that we compute $\tilde{h}(\cdot, \mathcal{X}_{n})$ up to $t=\ell_{K}(\mathcal{X}_{n})$ with $\mathcal{O}(n K^{2}(D+\ln K))$ operations. We then apply the slope selection procedure on the convexity defect function up to $\ell_{K}(\mathcal{X}_{n})$. If we select the maximal value possible, that is if $t_{\mathrm{sel}}(\mathcal{X}_{n})=\ell_{K}(\mathcal{X}_{n})$, then we did not go far enough in the computation of the convexity defect function. In that case, we repeat the procedure with $\bar{K}=2 K$. If $t_{\text {sel}}(\mathcal{X}_{n})<\ell_{K}(\mathcal{X}_{n})$, we stop. In practice, the maximal value $K_{\max }$ of $K$ is much smaller than $n$ and this approach leads to a considerable speed-up.

We test this faster algorithm on three classes of datasets. The first class is made of $n$ points uniformly sampled on a circle that lies on a random plane in $\mathbb{R}^{100}$, that are corrupted with uniform noise (in $\mathbb{R}^{100}$) of size $(\ln n / n)^{2 / d}$. The second class consists of points sampled on the torus of inner radius 1 and outer radius 4. The third class is made of points sampled on the swiss roll dataset from the SciPy Python library \cite{2020SciPy-NMeth}. For each class, we conduct 10 experiments for each value of $n, n$ ranging from $10^{2}$ to $10^{4}$. The value $K_{\max }$ was never larger than $2^{10}=1024$, and did not increase with $n$, see Figure \ref{fig9}. Increasing the ambient dimension in the first class did not significantly increase the computation time. We display in Figure \ref{fig10} the functions $\tilde{h}(\cdot, \mathcal{X}_{n})$ and $g_{\mathcal{X}_{n}}$ for two point clouds from this dataset: we observe once again the "jump" phenomenon occurring.

\section{Discussion and further works}

In this article, we introduced a particularly simple manifold estimator, based on a unique rule: add the convex hull of any subset of the set of observations which is of radius smaller than $t$. After proving that this leads to a minimax estimator for some choice of $t$, we explained how to select the parameter $t$ by computing the convexity defect function of the set of observations. The selection procedure actually allows us to find a parameter $t_{\lambda}(\mathcal{X}_{n})$ such that $\varepsilon(\mathcal{X}_{n}) / t_{\lambda}(\mathcal{X}_{n})$ is arbitrarily close to 1 (by choosing $\lambda$ close enough to 1 ). The selected parameter can therefore be used as a scale parameter in a wide range of procedures in geometric inference. We illustrated this general idea by showing how a data-driven minimax tangent space estimator can be created thanks to $t_{\lambda}(\mathcal{X}_{n})$. The main limitation to our procedure is its non-robustness to outliers. Indeed, even in the presence of one outlier in $\mathcal{X}_{n}$, the loss function $t \mapsto d_{H}(\operatorname{Conv}(t, \mathcal{X}_{n}), M)$ would be constant, equal to the distance between the outlier and the manifold $M$: with respect to the Hausdorff distance, all the estimators $\operatorname{Conv}(t, \mathcal{X}_{n})$ are then equally bad. Of course, even in that case, we would like to assert that some values of $t$ are "better" than others in some sense. A solution to overcome this issue would be to change the loss function, for instance by using Wasserstein distances on judicious probability measures built on the $t$-convex hulls $\operatorname{Conv}(t, \mathcal{X}_{n})$ in place of the Hausdorff distance

\section*{Acknowledgments}
I am grateful to Fréderic Chazal (Inria Saclay) and Pascal Massart (Université Paris-Sud) for thoughtful discussions and valuable comments on both mathematical and computational aspects of this work. I would also like to thank the anonymous reviewers for their helpful suggestions.

%%%%%%%%%%%%%%%%%%%%%%%%%%%%%%%%%%%%%%%%%%%%%%
%% Supplementary Material, if any, should   %%
%% be provided in {supplement} environment  %%
%% with title and short description.        %%
%%%%%%%%%%%%%%%%%%%%%%%%%%%%%%%%%%%%%%%%%%%%%%
%\stitle{???}
%\sdescription{???.}
\appendix

\section{Proof of Lemma \ref{lem4.12}}\label{appA}
Let $S=\sum_{k=1}^{K} \mathbf{1}\left\{N_{k}=2\right\}$. Let $\tilde{n}$ be the number of points of $\mathcal{X}_{n}$ in $\bigcup_{k} U_{k}$, so that $\tilde{n}$ follows a binomial distribution of parameters $n$ and $K m$. Recall that by construction, $K m \geq c_{0}$ for some constant $c_{0}$ (see Lemma \ref{lem4.9}). Conditionally on $\tilde{n}$, the random variable $S$ can be realized as the number of urns containing exactly two balls, in a model where $\tilde{n}$ balls are thrown uniformly in $K$ urns. Let $p_{i}=\binom{\tilde n}{i}K^{-i}\left(1-K^{-1}\right)^{\tilde{n}-i}$ be the probability that an urn contains exactly $i$ balls. We have $\mathbb{E}[S  |  \tilde{n}]=K p_{2}$, and
\begin{equation}\label{A1}
\begin{aligned}
\mathbb{E}\left[\exp \left(-C_{1} S\right)  |  \tilde{n}\right] & \leq \mathbb{E}\left[\exp \left(-C_{1} K p_{2} / 2\right) \mathbf{1}\left\{S \geq K p_{2} / 2\right\}  |  \tilde{n}\right]+\mathbb{P}\left(S<K p_{2} / 2  |  \tilde{n}\right) \\
& \leq \exp \left(-C_{1} K p_{2} / 2\right)+\mathbb{P}\left(\left|S-K p_{2}\right|>K p_{2} / 2  |  \tilde{n}\right).
\end{aligned}
\end{equation}
Let $v=2 K \max \left(2 p_{2}, 3 p_{3}\right)$. According to \cite[Proposition 3.5]{ben2017concentration}, if for some $s>0$,
\begin{equation}\label{A2}
K p_{2} / 2 \geq \sqrt{4 v s}+2 s / 3,
\end{equation}
then $\mathbb{P}\left(\left|S-K p_{2}\right|>K p_{2} / 2  |  \tilde{n}\right) \leq 4 e^{-s}.$ Recall that $n m^{2} \leq 1$ by assumption, and that $K \geq$ $c_{\mu, \delta} t^{-d} \geq c_{1} / m$. We therefore have $n / K^{2} \leq c_{1}^{-2}$. Assuming that $\tilde{n} \geq 3$ and using the inequality $\ln \left(1-K^{-1}\right) \geq-K^{-1}-K^{-2}$ for $K \geq 2$, we obtain the inequalities
\begin{equation}\label{A3}
p_{2} \geq \frac{(\tilde{n} / K)^{2}}{4 e^{c_{1}^{-2}}} e^{-\tilde{n} / K} \text { and } p_{3} \leq \frac{e^{3}}{6}(\tilde{n} / K)^{3} e^{-\tilde{n} / K} \leq c_{2} p_{2}(n / K)
\end{equation}
for some positive constant $c_{2}$. We consider two different regimes.

\begin{itemize}
  \item Assume first that $n / K \leq 2 /\left(3 c_{2}\right)$. Then $3 p_{3} \leq 2 p_{2}$ and one can check that $s=K p_{2} / 100$ satisfies \eqref{A2}. Inequality \eqref{A1} then yields that \[\mathbb{E}\left[\exp \left(-C_{1} S\right)  |  \tilde{n}\right] \leq 5 \exp \left(-C_{1}^{\prime} K p_{2}\right)\] for $C_{1}^{\prime}=$ $\min \left(C_{1} / 2,1 / 100\right)$. To conclude, we remark that for any $\alpha \in(0,1)$, by the Hoeffding inequality, the event $|\tilde{n}-n K m| \leq n K m \alpha$ holds with probability at least $1-\exp \left(-2 n \alpha^{2}\right).$ Letting $\alpha=1 / 2$, we obtain that, on this event,
\begin{equation*}
\frac{1}{2} n m \leq \frac{\tilde{n}}{K} \leq \frac{3}{2} n m \leq \frac{3}{2} \frac{n}{K} m K \leq \frac{1}{c_{2}}
\end{equation*}
where we used that $m K \leq 1$. Therefore, $p_{2} \geq c_{3}(n m)^{2} \geq c_{4}(n m)^{2} e^{-n m}$ for some constants $c_{3}$ and $c_{4}$. The probability of order $\exp \left(-2 n \alpha^{2}\right)$ being negligible, we obtain a final bound of order $\exp \left(-C_{1}^{\prime} c_{4} K(n m)^{2} e^{-n m}\right) \leq \exp \left(-C_{2} n \phi(n m)\right)$, concluding the proof in the regime $n / K \leq$ $2 /\left(3 c_{2}\right)$.

  \item Otherwise, we have $n / K>2 /\left(3 c_{2}\right)$ and we also assume that $|\tilde{n}-n K m| \leq \alpha n K m$ for some $\alpha \in(0,1)$ to fix (this happens with probability $1-\exp \left(-2 n \alpha^{2}\right)$ by Hoeffding's inequality). One can then check using \eqref{A3} that $s=c_{5} \tilde{n} e^{-\tilde{n} / K}$ satisfies \eqref{A2} if $c_{5}$ is chosen small enough. Furthermore, $s \leq c_{6} K p_{2}$ for some constant $c_{6}$ (using \eqref{A3}). The leading term in \eqref{A1} is therefore of the form $\exp \left(-c_{7} \tilde{n} e^{-n / K}\right)$. Let $\alpha=1 /(\ln n)^{3}$. We have, as $n m \geq c_{0} n / K \geq c_{8}$ and as $n m \leq(\ln n)^{2}$ (by assumption),
\begin{equation*}
c_{9} \leq n m(1-\alpha) \leq \frac{\tilde{n}}{K} \leq n m(1+\alpha) \leq n m+\frac{1}{\ln n}.
\end{equation*}
Therefore, $\tilde{n} e^{-\tilde{n} / K} \geq\left(c_{9} / 2\right) K e^{-n m}.$ The probability of order $\exp (-2 n \alpha^{2})$ is still negligible, and we obtain a final bound on $\mathbb{E}\left[\exp \left(-C_{1} S\right)\right]$ of order $\exp \left(-\left(c_{9} / 2\right) K e^{-n m}\right) \leq \exp \left(-c_{10} n \phi(n m)\right)$.
\end{itemize}

\section{Precise lower bound on the minimax risk}\label{appB}
We adapt the construction made in \cite{stat:kim2015tight} so that the lower bound on the minimax risk holds with an explicit constant. Let $0<d<D$ and $\tau_{\min }, f_{\min }, f_{\max }$ with $\omega_{d} f_{\min } \tau_{\min }^{d}<\kappa$. We let $M(\mu)$ be the underlying manifold of the law $\mu \in \mathcal{Q}_{\tau_{\min }, f_{\min }, f_{\text {max }}}^{d}$ The lowerbound is based on Le Cam's lemma:

\begin{lemma}\label{lemB1}
 Let $\mathcal{P}^{(1)}, \mathcal{P}^{(2)}$ be two subfamilies of $\mathcal{Q}_{\tau_{\min }, f_{\min }, f_{\max }}^{d}$ which are \textbackslash varepsilon-separated, in the sense that $d_{H}(M(\mu^{(1)}), M(\mu^{(2)})) \geq 2 \varepsilon$ for all $\mu^{(1)} \in \mathcal{P}^{(1)}$, $\mu^{(2)} \in \mathcal{P}^{(2)}$. Then
\begin{equation}\label{B1}
\mathcal{R}_{n}(M, \mathcal{Q}_{\tau_{\min }, f_{\min }, f_{\max }}^{d}) \geq \varepsilon\left|\left(\frac{1}{\# \mathcal{P}^{(1)}} \sum_{\mu^{(1)} \in \mathcal{P}^{(1)}} \mu^{(1)}\right) \wedge\left(\frac{1}{\# \mathcal{P}^{(2)}} \sum_{\mu^{(2)} \in \mathcal{P}^{(2)}} \mu^{(2)}\right)\right|,
\end{equation}
where $|\mu \wedge \nu|$ is the testing affinity between two distributions $\mu$ and $\nu$.
\end{lemma}

To obtain a lowerbound on the minimax risk, authors in \cite{stat:kim2015tight} exhibit two families of manifolds which are $\varepsilon$-separated, and consider the uniform distributions on them. Those manifolds are built by considering a base manifold $M_{0}$ which is locally flat, and by adding small bumps on the locally flat part. Such a construction leads to distributions having a density equal roughly to $1 / \operatorname{Vol}\left(M_{0}\right)$, a constant which might be smaller than $f_{\text {min }}$. If this is the case, then the corresponding submodels are not in $\mathcal{Q}_{\tau_{\min }, f_{\min }, f_{\text {max }}}^{d}$ and we cannot apply Le Cam's Lemma. Hence, we consider another base manifold, which is a sphere $M_{0}$ of radius $R$ slightly larger than $\tau_{\min }$, so that its volume is smaller than $1 / f_{\min }$ (this is possible as $\left.f_{\min } \omega_{d} \tau_{\min }^{d} \leq \kappa<1\right)$. The two families are then once again constructed by adding small bumps on $M_{0}$. We now detail this construction. Let $R, \delta>0$ be two parameters to be fixed later. Let $M_{0} \subseteq \mathbb{R}^{d+1} \subseteq \mathbb{R}^{D}$ be the $d$-sphere of radius $R$, and let $A$ be a maximal subset of $M_{0}$ of even size, which is $4 \delta$-separated. Note that, standard packing arguments (and the formula for the volume of a spherical cap) show that if $\delta / R$ is small enough, then the cardinality $2 m$ of $A$ satisfies $2 m \geq\left(\frac{c_{0} R}{\delta}\right)^{d}$ for some absolute constant $c_{0}$.

Let $\phi: \mathbb{R} \rightarrow \mathbb{R}$ be a smooth function such that $0 \leq \phi \leq 1, \phi \equiv 1$ on $[-1,1]$ and $\phi \equiv 0$ on $\mathbb{R} \backslash[-2,2]$. For $s \in\{\pm 1\}^{A}$, we build a diffeomorphism $\Phi_{s}^{\varepsilon}$ by letting for $x \in \mathbb{R}^{D}$
\begin{equation}
\Phi_{s}^{\varepsilon}(x)=x\left(1+\frac{\varepsilon}{R} \sum_{y \in A} s(y) \phi\left(\frac{\|x-y\|}{\delta}\right)\right).
\end{equation}
Recall that $\|N\|_{\text {op }}$ denotes the operator norm of a linear application $N$.

\begin{lemma}\label{lemB2} There exists two absolute constants $c_{0}, c_{1}, c_{2}>0$ such that the following holds. Assume that $\delta \leq R$ and that $c_{0} \varepsilon / \delta<1.$ Then, the function $\Phi_{s}^{\varepsilon}: \mathcal{B}(0,3 R) \rightarrow \mathbb{R}^{d+1}$ is a diffeomorphism on its image, with
\begin{equation}
\sup _{x \in \mathcal{B}(0,3 R)}\left\|\operatorname{Id}-d_{x} \Phi_{s}^{\varepsilon}\right\|_{\mathrm{op}} \leq c_{1} \varepsilon / \delta \text { and } \sup _{x \in \mathcal{B}(0,3 R)}\left\|d_{x}^{2} \Phi_{s}^{\varepsilon}\right\|_{\text {op }} \leq c_{2} \varepsilon / \delta^{2}.
\end{equation}
\end{lemma}

\begin{proof}
As $A$ is $4 \delta$-separated, at most one term in the sum in (B.2) is non-zero. A computation gives that the derivative of $\Phi_{B}$ is given by, for $x \in \mathcal{B}(0,3 R)$,

\begin{equation}\label{B4}
\begin{aligned}
&d_{x} \Phi_{s}^{\varepsilon}(h)=\\
& h+h \frac{\varepsilon}{R} \sum_{y \in A} s(y) \phi\left(\frac{|x-y|}{\delta}\right)+x \frac{\varepsilon}{R} \sum_{y \in A} \frac{1}{\delta} s(y) \phi^{\prime}\left(\frac{|x-y|}{\delta}\right) \frac{\langle x-y, h\rangle}{|x-y|}.
\end{aligned}
\end{equation}
Hence,
\begin{equation*}
\left\|\mathrm{Id}-d_{x} \Phi_{s}^{\varepsilon}\right\|_{\mathrm{op}} \leq \frac{\varepsilon}{R}\left(\|\phi\|_{\infty}+|x| \frac{\left\|\phi^{\prime}\right\|_{\infty}}{\delta}\right) \leq \frac{\varepsilon}{R}\left(\|\phi\|_{\infty}+3 R \frac{\left\|\phi^{\prime}\right\|_{\infty}}{\delta}\right) \leq c_{1} \frac{\varepsilon}{\delta}
\end{equation*}
where $c_{1}=c_{0}\|\phi\|_{\infty}+3\left\|\phi^{\prime}\right\|_{\infty}$. A similar computation gives that $\left\|d_{x}^{2} \Phi_{s}^{\varepsilon}\right\|_{\text {op }} \leq c_{2} \varepsilon / \delta^{2}$ for $c_{2}=$ $4\left\|\phi^{\prime}\right\|_{\infty}+3\left\|\phi^{\prime \prime}\right\|_{\infty}$. We eventually show the injectivity: if $\Phi_{s}^{\varepsilon}(x)=\Phi_{s}^{\varepsilon}\left(x^{\prime}\right)$, then $x$ and $x^{\prime}$ are colinear. Also, if $c_{0}=\|\phi\|_{\infty}+3\left\|\phi^{\prime}\right\|_{\infty}$, one can check using \eqref{B4} that the derivative of the function $r \in[0,3 R] \mapsto\left\langle\Phi_{s}^{\varepsilon}(r u), u\right\rangle$ for $u$ an unit vector is increasing, proving the injectivity.
\end{proof}

Therefore, from \cite[Theorem 14.19]{geo:federer1959curvature}, we infer that $M_{s}^{\varepsilon}:=\Phi_{s}^{\varepsilon}(M)$ is a manifold with reach larger than
\begin{equation}\label{B5}
\tau\left(M_{s}^{\varepsilon}\right) \geq R \min \left(1-c_{1} \varepsilon / \delta, \frac{\left(1-c_{1} \varepsilon / \delta\right)^{2}}{1+c_{1} \varepsilon / \delta+R c_{2} \varepsilon / \delta^{2}}\right)
\end{equation}

\begin{figure}
\centering
\includegraphics[width=0.7\textwidth]{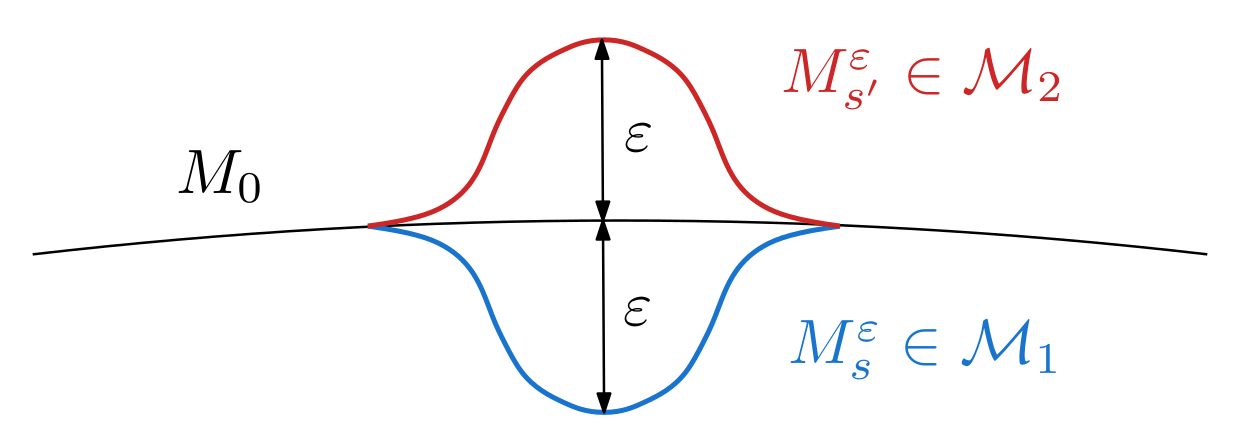}
\caption{An element $\mu^{(1)} \in \mathcal{P}^{(1)}$ has its first marginal supported on the blue manifold $M_{s}^{\varepsilon}$ (lower bump), whereas an element $\mu^{(2)} \in \mathcal{P}^{(2)}$ is supported on the red manifold $M_{s^{\prime}}^{\varepsilon}$ (upper bump).}\label{fig11}
\end{figure}

Denote by $J \Phi_{s}^{\varepsilon}$ the Jacobian of $\Phi_{\varepsilon}^{s}$. Then, the volume of $M_{s}^{\varepsilon}$ is controlled by
\begin{equation}\label{B6}
\begin{aligned}
\omega_{d} R^{d} \leq \operatorname{Vol}\left(M_{s}^{\varepsilon}\right) &=\int_{M_{0}} J \Phi_{s}^{\varepsilon}(x) \mathrm{d} x=\omega_{d} R^{d}+\sum_{y \in A} \int_{\mathcal{B}_{M_{0}}(y, 2 \delta)}\left(J \Phi_{s}^{\varepsilon}(x)-1\right) \mathrm{d} x \\
& \leq \omega_{d} R^{d}+2 m C_{d} c_{1} \frac{\varepsilon}{\delta} \operatorname{Vol}\left(\mathcal{B}_{M_{0}}(y, 2 \delta)\right) \leq \omega_{d} R^{d}\left(1+C_{d} c_{1} \frac{\varepsilon}{\delta}\right)
\end{aligned}
\end{equation}
where we used that $\operatorname{det}(N)-1 \leq C_{d}\|N-\mathrm{Id}\|_{\mathrm{op}}$ for some constant $C_{d}$ if $N$ is a matrix of size $d$ with operator norm smaller than 1, the fact that $2 m \operatorname{Vol}\left(\mathcal{B}_{M_{0}}(y, 2 \delta)\right) \leq \operatorname{Vol}\left(M_{0}\right)$, and Lemma \ref{lemB2}.

Let $R=\tau_{\min }+\frac{1}{2}\left(\frac{1}{\left(\omega_{d} f_{\min }\right)^{1 / d}}-\tau_{\min }\right)$ and $\delta=\sqrt{R \varepsilon} \nu$ where $\nu^{2}=\frac{2 c_{2} \tau_{\min }}{R-\tau_{\min }}.$ With this choice of parameters, one can check that, for $\varepsilon / \delta$ small enough, $\tau\left(M_{s}^{\varepsilon}\right) \geq \tau_{\min }$ (by \eqref{B5}) and $\operatorname{Vol}\left(M_{s}^{\varepsilon}\right) \leq 1 / f_{\min }$ (by \eqref{B6} and using that $\omega_{d} f_{\min } \tau_{\min }^{d} \leq \kappa<1$).

We define the family $\mathcal{M}^{(1)}$ of manifolds $M_{s}^{\varepsilon}$ where $s$ contains exactly $m$ signs $+1$ (and $m$ signs $-1$). The family $\mathcal{M}^{(2)}$ is defined likewise by considering $M_{s}^{\varepsilon}$ where $s$ contains exactly $m+1$ or $m-1$ signs $+1$. We then let $\mathcal{P}^{(1)}$ be the set of distributions $Q_{s}^{\varepsilon}$ where $Q_{s}^{\varepsilon}$ is the uniform distribution on a manifold of $M_{s}^{\varepsilon} \in \mathcal{M}^{(1)}$, so that $\mathcal{P}^{(1)}$ is a subset of $\mathcal{Q}_{\tau_{\min }, f_{\min }}^{d}$ fhe set $\mathcal{P}^{(2)}$ is defined likewise.

By construction, the two families $\mathcal{P}^{(1)}, \mathcal{P}^{(2)}$ are $2 \varepsilon$-separated (see Figure \ref{fig11}). Hence, we can apply Le Cam's lemma. The exact same computations than in \cite[Section 3]{stat:kim2015tight} show that the testing affinity between $\mathcal{P}^{(1)}$ and $\mathcal{P}^{(2)}$ converge to 1 as long as $4 m=n / \ln n$. Thus, Le Cam's Lemma \eqref{B1} yields
\begin{equation}
\liminf _{n} \frac{\mathcal{R}_{n}(M, \mathcal{Q}_{\tau_{\min }, f_{\min }, f_{\max }}^{d})}{(\ln n / n)^{2 / d}} \geq \liminf _{n}(m / 4)^{2 / d} \varepsilon.
\end{equation}
As $2 m \geq\left(c_{0} R / \delta\right)^{d}$, we therefore have
\begin{equation*}
\begin{aligned}
&\lim _{n} \inf \frac{\mathcal{R}_{n}(M, \mathcal{Q}_{\tau_{\min }, f_{\min }, f_{\max }}^{d})}{(\ln n / n)^{2 / d}} \geq \frac{c_{0}^{2}}{8^{2 / d}} \frac{R^{2}}{\delta^{2}} \varepsilon=\frac{c_{0}^{2}}{8^{2 / d}} \frac{R}{\nu^{2}} \\
&=\frac{c_{0}^{2}}{8^{2 / d}} \frac{R\left(R-\tau_{\min }\right)}{2 c_{2} \tau_{\min }} \geq \frac{c_{3}}{\left(\omega_{d} f_{\min }\right)^{1 / d} \tau_{\min }}\left(\frac{1}{\left(\omega_{d} f_{\min }\right)^{1 / d}}-\tau_{\min }\right),
\end{aligned}
\end{equation*}
for some absolute constant $c_{3}$, where we used that by definition, \[R-\tau_{\min }=\frac{1}{2}\left(\frac{1}{\left(\omega_{d} f_{\min }\right)^{1 / d}}-\tau_{\min }\right),\] and that $R \geq \frac{1}{2}\left(\omega_{d} f_{\min }\right)^{-1 / d}$. As $\tau_{\min } \leq \kappa /\left(\omega_{d} f_{\min }\right)^{1 / d}$, and as $\omega_{d}^{1 / d} \leq c \alpha_{d}^{1 / d}$ for some absolute constant $c$, we obtain the conclusion with constant $C=c_{3}(1-\kappa) / c$. Note that the lower bound actually holds on the smaller model $\mathcal{Q}_{\tau_{\min }, f_{\min }, f_{\min }}^{d}$, as we only considered uniform distributions in the proof.

%% if your bibliography is in bibtex format, uncomment commands:
\bibliographystyle{alpha} % Style BST file (imsart-number.bst or imsart-nameyear.bst)
\bibliography{biblio.bib}       % Bibliography file (usually '*.bib')

%% or include bibliography directly:
% \begin{thebibliography}{}
% \bibitem{b1}
% \end{thebibliography}

\end{document}